\RequirePackage{fix-cm}
\documentclass[twocolumn]{svjour3}          
\smartqed  
\usepackage{graphicx}
\usepackage[T1]{fontenc}
\usepackage[utf8]{inputenc}
\usepackage[francais,english]{babel}
\usepackage{amsmath,latexsym,amsfonts,bm,mathtools}
\usepackage{graphicx}
\usepackage[caption=false,font=footnotesize]{subfig}
\usepackage{floatrow}
\usepackage{algorithm2e}
\newfloatcommand{capbtabbox}{table}[][\FBwidth]
\newtheorem{thapp}{Theorem}
\newtheorem{algo}{Algorithm}

\begin{document}\sloppy

\title{A discrete framework to find the optimal matching between manifold-valued curves
}


\author{Alice Le Brigant
}


\institute{
              Institut Math\'ematique de Bordeaux,
              UMR 5251, Universit\'e de Bordeaux and CNRS, France\\
              \email{alice.lebrigant@math.u-bordeaux.fr}           
}

\date{}

\maketitle

\begin{abstract}
The aim of this paper is to find an optimal matching between manifold-valued curves, and thereby adequately compare their shapes, seen as equivalent classes with respect to the action of reparameterization. Using a canonical decomposition of a path in a principal bundle, we introduce a simple algorithm that finds an optimal matching between two curves by computing the geodesic of the infinite-dimensional manifold of curves that is at all time horizontal to the fibers of the shape bundle. We focus on the elastic metric studied in \cite{moi17} using the so-called square root velocity framework. The quotient structure of the shape bundle is examined, and in particular horizontality with respect to the fibers. These results are more generally given for any elastic metric. We then introduce a comprehensive discrete framework which correctly approximates the smooth setting when the base manifold has constant sectional curvature. It is itself a Riemannian structure on the product manifold $M^{n}$ of "discrete curves" given by $n$ points, and we show its convergence to the continuous model as the size $n$ of the discretization goes to $\infty$. Illustrations of optimal matching between discrete curves are given in the hyperbolic plane, the plane and the sphere, for synthetic and real data, and comparison with dynamic programming \cite{srivastava16} is established.

\end{abstract}

\section{Introduction}				
\label{intro}

\subsection{Context}

The study of curves and their shapes is a research area with numerous and varied applications, which is why it has known a great deal of activity over the past years. These curves can be closed or open, and take their values in a Euclidean space or more generally in a Riemannian manifold. To name a few examples, closed plane curves are central in shape analysis of objects \cite{younes12}; the study of trajectories on the Earth requires to deal with open curves on the sphere \cite{zhang16}; and in signal processing, locally stationary Gaussian processes can be represented by open curves in the hyperbolic plane, seen as the statistical manifold of Gaussian densities \cite{moi16}, \cite{moi17}. Here we are concerned with the study of open curves in a manifold $M$ of constant sectional curvature.

\begin{figure*}
\includegraphics[width=0.3\textwidth]{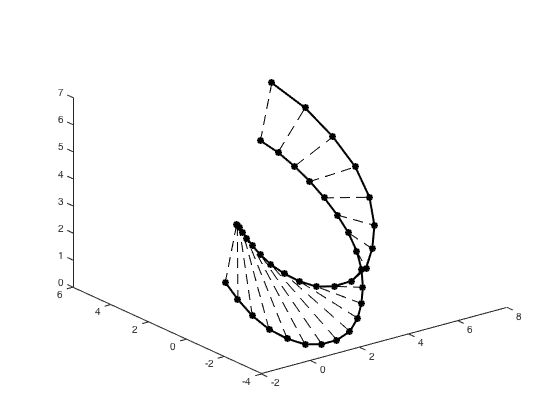}
\includegraphics[width=0.3\textwidth]{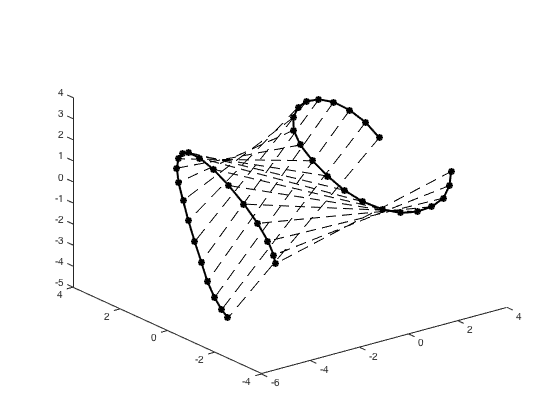}
\includegraphics[width=0.3\textwidth]{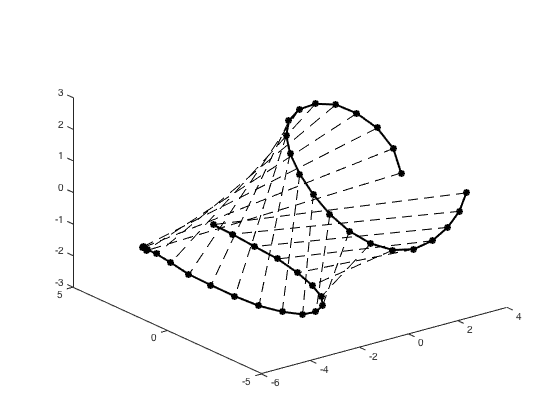}
\caption{\!Examples of (suboptimal) matchings between shapes in $\mathbb R^3$. Correspondence between points is shown with dashed lines.}
\label{fig:matchings}
\end{figure*}

There are naturally many ways to go about comparing curves in a manifold. One way is to see the space of manifold-valued curves as an (infinite-dimensional) manifold $\mathcal M$ itself, and equip it with a Riemannian metric $G$. Then the geodesic between two curves in $\mathcal M$ describes how one optimally deforms into the other, while its length gives a measure of dissimilarity : the geodesic distance. The advantage of this strategy is that it provides all the convenient tools of the Riemannian framework. An interesting property for the metric $G$, from the point of view of the applications, is invariance with respect to re-parameterization: for closed curves, this amounts to considering only the shape of an object; for an open curve representing the evolution in time of a given process, this allows us to analyze it regardless of speed or pace. A popular strategy is to consider the quotient space $\mathcal S$ of curves modulo reparameterization, where two curves are considered identical if they pass through the same points of $M$ but at different speeds, or equivalently when one can be obtained by reparameterizing the other. This quotient space is often called the \emph{shape space}. If the Riemannian metric $G$ defines the same scalar product at all points of $\mathcal M$ which project on the same "shape", then $G$ induces a Riemannian structure on the quotient space. Computing geodesics between two shapes for that metric can be done in two steps : (1) establishing an "optimal matching" through the choice of two parameterizations, one for each shape, that will put in correspondence the points on the first shape with points on the second shape, as in Figure \ref{fig:matchings}, and (2) computing the geodesic between the two parameterized curves obtained. Both steps depend on the choice of the metric.

\subsection{Previous work}

Since the simplest metric one can think of, the $L^2$-metric (slightly modified to stay constant along the fibers), induces a vanishing distance on the quotient space \cite{michor05}, different classes of metrics have been studied to perform shape analysis. The large class of Sobolev metrics involves higher order derivatives to overcome the vanishing problem of the $L^2$-metric \cite{michor07}. A first-order Sobolev metric for plane curves was introduced in \cite{younes98} and used in \cite{younes08}, which can be mapped to the $L^2$-metric by a change of coordinates, namely by considering the complex square root of the speed of the curve. This metric was modified by the authors of \cite{mio07} to define the family of \emph{elastic metrics} $G^{a,b}$, parameterized by two constants $a$ and $b$ which control the degree of bending and stretching of the curve. In \cite{srivastava11}, the authors show that for a certain choice of parameters, the elastic metric can again be mapped to the $L^2$-metric using the so-called \emph{square root velocity} (SRV) coordinates, where a curve is represented by its speed renormalized by the square root of its norm. The SRV framework was generalized in \cite{bauer12} for any elastic metric with weights $a$ and $b$ satisfying a certain relation. A quotient structure for the metric used in \cite{srivastava11} is carefully developed in \cite{lahiri15}, where the authors prove that if at least one of two curves is piecewise-linear, then there exists a minimizing geodesic between the two, and give a precise algorithm to solve the matching problem. In \cite{bruveris15}, it is proven that in the same framework, there always exists an optimal reparameterization realizing the minimal distance between two $C^1$ plane curves. Another approach is proposed in \cite{tumpach16}, where the authors restrict to arc-length parameterized curves and characterize the horizontal space of the quotient structure for these curves in the elastic framework.

Concerning manifold-valued curves, the geodesic equations for Sobolev metrics in the space of curves and in the shape space were given in \cite{bauer11} in terms of the gradient of the metric with respect to itself. A generalization of the SRV framework to manifold-valued curves was introduced in \cite{zhang15} and used in \cite{zhang16}, while another one was proposed in \cite{moi17}. Extension to curves in a Lie group or a homogeneous space can also be found in \cite{celledoni16}, \cite{celledoni17}, \cite{su17}. Both metrics in \cite{zhang15} and \cite{moi17} coincide with the metric of \cite{srivastava11} in the flat case, however they define different Riemannian structures when the base space has curvature. The difference lies in the way computations are made -- in \cite{zhang15} and \cite{zhang16} they are moved to the tangent spaces at the origins of the curves, resulting in simpler computations, whereas in \cite{moi17} they are done directly in the base manifold $M$, transporting data pointwise across $M$ from one curve to the other, thus making the comparison more sensitive to the local "geography" of $M$. When comparing the geodesic distances, this difference is measured by a curvature term (see \cite{moi17}, Prop. 3). In addition, unlike the metric of \cite{zhang15}, the one in \cite{moi17} can be written as an elastic metric (with $a=2b=1$). The work of \cite{zhang15} is applied to curves in the space of symmetric positive definite matrices, while the case of spherical trajectories is investigated in \cite{zhang16}, where the authors exhibit simplifications. In both cases, optimal matching between curves is achieved through dynamic programming. On the other hand, the Riemannian structure in \cite{moi17} is applied to curves in the hyperbolic plane, but the question of optimal matching is not studied. This is remedied here.

\subsection{Contributions of this paper}

The aim of this paper is three-fold : (1) study the quotient structure associated to the Riemannian metric $G$ of \cite{moi17}, and more generally to any elastic metric on manifold-valued curves, (2) exploit this knowledge of horizontality in an algorithm that computes an optimal matching between shapes -- and whose range of application goes beyond elastic metrics -- and (3) give a comprehensive discrete Riemannian framework on the finite-dimensional manifold of "discrete curves" that correctly approximates these procedures for the particular case of the SRV framework and constant sectional curvature. More specifically,

\begin{itemize}
\item[$\bullet$] we characterize the horizontal subspace associated to the quotient shape space, for any elastic metric and in particular for the SRV metric $G$. Namely, we decompose any infinitesimal deformation of a smooth curve into a vertical part, which reparameterizes the curve without changing its shape, and a $G$-orthogonal horizontal part. This is done in a similar way as in \cite{tumpach16} but without restriction to arc-length parameterized curves.

\item[$\bullet$] We write any path in the space of smooth curves as a horizontal path composed with a path of reparameterizations. We use this decomposition to define an algorithm that, for a fixed parameterization of one of two curves, approximates the horizontal geodesic linking it to the fiber of the other curve, thereby yielding the "closest parameterization" of the latter with respect to the fixed parameterization of the former. We refer to this correspondence as an \emph{optimal matching}. This algorithm can be used for any metric structure as long as one knows how to compute geodesics and characterize horizontality. Comparison with the popular dynamic programming approach \cite{srivastava16} is established in the simulations section.

\item[$\bullet$]We define a discrete version of $G$ that is a Riemannian metric on the finite-dimensional manifold $M^{n+1}$ of "discrete curves" given by $n+1$ points, when $M$ has constant sectional curvature $K\in\{-1,0,1\}$. We show that the energy of a path of discrete curves converges to the energy of the limit path as the size $n$ of the discretization goes to $\infty$. We give the geodesic equations for this metric, characterize the Jacobi fields, describe geodesic shooting, and show simulations on synthetic and real data in $\mathbb R^2,\mathbb R^3$, the hyperbolic plane and the sphere.
\end{itemize}

\subsection{Outline}

After reminding the continuous model previously introduced in \cite{moi17}, Section 2 describes the horizontal space of the quotient structure and a way to compute horizontal geodesics. In Section 3, we introduce the discretization, and give the convergence result toward the continuous model, which is later proved in Section 5. Section 4 shows results of simulations in the three settings of positive, zero and negative curvature.

\section{The continuous model} 			
\label{sec:contmodel}

\subsection{Notations}

Let $\left(M,\langle \cdot,\cdot \rangle\right)$ be a Riemannian manifold. We first introduce a few notations. The norm associated to the Riemannian metric $\langle \cdot, \cdot \rangle$ is denoted by $| \cdot |$, the Levi-Civita connection by $\nabla$ and the curvature tensor by $\mathcal R$. If $t\mapsto c(t)$ is a curve in $M$ and $t\mapsto w(t) \in T_{c(t)}M$ a vector field along $c$, we denote by $c_t:=dc/dt=c'$ the derivative of $c$ with respect to $t$ and by $\nabla_tw := \nabla_{c_t}w$, $\nabla^2_tw := \nabla_{c_t}\nabla_{c_t}w$ the first and second order covariant derivatives of $w$ along $c$. Parallel transport of a tangent vector $u \in T_{c(t_1)}M$ from $c(t_1)$ to $c(t_2)$ along $c$ is denoted by $P_c^{t_1,t_2}(u)$, or when there is no ambiguity on the choice of the curve $c$,  $u^{t_1,t_2}$, or even $u^\parallel$ to lighten notations in some cases. We associate to each curve $c$ its renormalized speed vector field $v := c' / |c'|$, and to each vector field $t \mapsto w(t)$ along $c$, its tangential and normal components $w^T:=\langle w,v\rangle v$ and $w^N:=w-w^T$. Finally, for all $x\in M$ we denote by $\exp_x:T_xM \rightarrow M$ the exponential map on $M$ and by $\log_x:M \rightarrow T_xM$ its inverse map.

\subsection{The space of smooth parameterized curves}

\subsubsection{The Riemannian structure}

We represent open oriented curves in $M$ by smooth immersions, i.e. smooth curves with velocity that doesn't vanish. The set $\mathcal M$ of smooth immersions in $M$ is an open submanifold of the Fr\'echet manifold $C^\infty([0,1],M)$ \cite{michor16} and its tangent space at a point $c$ is the set of infinitesimal deformations of $c$, which can be seen as vector fields along the curve $c$ in $M$
\begin{align*}
\mathcal{M}\!=\!
\{ c\in C^\infty([0,1],M)|c'(t) \neq 0, \forall t\in [0,1] \},\\
T_c\mathcal M\! =\! \{ w \in C^\infty([0,1],TM)| w(t) \in T_{c(t)}M, \forall t\in [0,1] \}.
\end{align*}
Reparametrizations are represented by increasing diffeomorphisms $\varphi : [0,1] \rightarrow [0,1]$ (so that they preserve the end points of the curves), and their set is denoted by $\text{Diff}^+([0,1])$. We adopt the so-called square root velocity (SRV) representation, i.e. we represent each curve $c\in \mathcal M$ by the pair formed by its starting point $x$ and its speed vector field renormalized by the square root of its norm, via the bijection $\mathcal M \rightarrow M\times T\mathcal M$
\begin{equation*}
c \mapsto \left( x:=c(0),\,\, q := \frac{c'}{\sqrt{|c'|}}\right).
\end{equation*}
The inverse of this function is simply given by $M\times T\mathcal M \ni (x,q) \mapsto \pi_{\mathcal M}(q)\in \mathcal M$, if $\pi_{\mathcal M}$ is the canonical projection $T\mathcal M \rightarrow \mathcal M$. The renormalization of the speed vector field in $q$ allows us to define a reparameterization invariant metric, as we will see shortly. For any tangent vector $w\in T_c\mathcal M$, consider a path of curves $s\mapsto c^w(s)\in \mathcal M$ such that $c^w(0)=c$ and $c^w_s(0):=\partial c^w/\partial s(0) =w$. We denote by $q^w:=c^w_t/|c^w_t|^{1/2}$ the square root velocity representation of $c^w$. With these notations, we equip $\mathcal M$ with a Riemannian metric $G$, defined by
\begin{equation}
\label{contmetric}
G_c(w,w) = |w(0)|^2 + \int_0^1 |\nabla_sq^w(0,t)|^2 \,\mathrm dt.
\end{equation}
This definition does not depend on the choice of $c^w$ and we can reformulate this scalar product in terms of (covariant) derivatives of $w$. Indeed, note that $\nabla_sq^w(0,t)=\nabla_{c^w_s(0,t)}(c^w_t/|c^w_t|^{\frac{1}{2}})=|c^w_t(0,t)|^{-\frac{1}{2}}(\nabla_{c^w_s(0,t)}c^w_t-1/2(\nabla_{c^w_s(0,t)}c^w_t)^T)$, which gives after inverting the derivatives according to $s$ and $t$,
\begin{equation*}
\nabla_sq^w(0,t)=|c'|^{-1}({\nabla_tw}^N+\tfrac{1}{2}{\nabla_tw}^T).
\end{equation*}
The scalar product can then be rewritten
\begin{align}
G_c&(w,w) = | w(0)|^2 + \int_0^1 \left(|\nabla_t w^N|^2+ \tfrac{1}{4} |\nabla_t w^T|^2 \right)\frac{\mathrm dt}{ |c'|},\nonumber\\
G_c&(w,w) = | w(0)|^2+ \int \left(|\nabla_\ell w^N|^2 + \tfrac{1}{4} |\nabla_\ell w^T|^2 \right) \mathrm d\ell,\label{contmetricbis}
\end{align}
where $\mathrm d\ell = |c'(t)|\mathrm dt$ and $\nabla_\ell =\frac{1}{|c'(t)|}\nabla_t$ respectively denote integration and covariant derivation according to arc length. This metric belongs to the class of so-called \emph{elastic metrics} parameterized by any $a,b\in \mathbb R$, which can be defined for manifold-valued curves as 
\begin{equation*}
G^{a,b}_c(w,w)=| w(0)|^2+ \int a^2\left(|\nabla_\ell w^N|^2 + b^2 |\nabla_\ell w^T|^2 \right) \mathrm d\ell.
\end{equation*}
With formulation \eqref{contmetricbis} it is clear that $G=G^{1,\frac{1}{2}}$ is invariant under the action of reparameterizing the curve and its tangent vectors by any increasing diffeomorphism $\varphi \in\text{Diff}^+([0,1])$,
\begin{equation}
\label{equivariance}
G_{c\circ \varphi}(w\circ \varphi, z\circ \varphi) = G_c(w,z), \quad \forall w,z\in T_c\mathcal M.
\end{equation}
This \emph{reparameterization invariance} property will allow us to induce a Riemmannian structure on the quotient space as we will see in Section \ref{sec:quotient}.

\subsubsection{Geodesics between parameterized curves}
\label{subsubsec:geod}

Two curves $c_0, c_1 \in \mathcal M$ can be compared using the geodesic distance induced by $G$, i.e. by computing the length of the shortest path of curves $[0,1] \ni s \mapsto c(s) \in \mathcal M$ from $c_0$ to $c_1$
\begin{equation}
\label{dist}
d_G(c_0,c_1) = \inf \left\{ L(c)
: \,\,c(0)=c_0, c(1)=c_1 \right\},
\end{equation}
where the length of a path $c$ can be written in terms of its SRV representation $(x,q) : [0,1] \rightarrow M\times T\mathcal M$
\begin{equation}
\label{srv}
s \mapsto \left(x(s):=c(s,0), \,q(s,\cdot):=\frac{c_t(s,\cdot)}{|c_t(s,\cdot)|^{1/2}}\right),
\end{equation}
as
\begin{equation*}
L(c)=\int_0^1 \sqrt{ |x_s(s)|^2 +\int_0^1 |\nabla_sq(s,t)|^2 \mathrm dt } \,\, \mathrm ds.
\end{equation*}
Note that here - and in all that follows - we indifferently use the notations $c(s,t)=c(s)(t)$, $q(s,t)=q(s)(t)$ for all $s,t\in[0,1]$. 
Now we recall a result shown in \cite{moi17}, which characterizes the geodesic paths of $\mathcal M$, i.e. those which achieve the infimum in \eqref{dist}, by searching for the critical points of the energy functional $E:\mathcal M\rightarrow \mathbb R_+$,
\begin{equation}
\label{contenergy}
E(c)=\int_0^1 \left( |x'(s)|^2 +\int_0^1 |\nabla_sq(s,t)|^2 \mathrm dt \right) \,\, \mathrm ds.
\end{equation}
\begin{proposition}[Geodesic equations]
A geodesic path $[0,1] \ni s \mapsto c(s) \in \mathcal M$ for $G$, or more specifically its SRV representative $s\mapsto (x(s),q(s))$ \eqref{srv}, verifies the equations
\begin{equation}
\begin{aligned}
\label{contgeodeq}
\nabla_sx_s(s) + r(s,0) =&  0,  \\
\nabla_s^2 q(s,t) + |q(s,t)| \left( r(s,t) + r(s,t)^T \right) =& 0, 
\end{aligned}
\end{equation}
for all $s,t\in[0,1]$, where the curvature term $r(s,t)$ integrates the vector field $\mathcal R(q,\nabla_sq)c_s$ parallel transported back to $t$ along $c(s,\cdot)$,
\begin{equation*}
r(s,t) = \int_t^1 \mathcal R(q,\nabla_sq)c_s(s,\tau)^{\tau,t} \mathrm d\tau, \quad t\in[0,1].
\end{equation*} 
\end{proposition}
\begin{remark}
In the flat case $M=\mathbb R^d$, the curvature term $r$ vanishes and we obtain $\nabla_sx_s(s)=0$, $\nabla_s^2q(s,t)=0$ for all $s$ and $t$. This means that the geodesic between two curves $(x_0,q_0)$ and $(x_1,q_1)$ in the SRV coordinates is composed of a straight line $s\mapsto x(s)$ and an $L^2$-geodesic $s \mapsto q(s,\cdot)$. In other words, the geodesic links the starting points with a straight line and linearly interpolates between the renormalized speeds. Notice that if this linear interpolation goes through zero, the geodesic does not exist in $\text{Imm}([0,1],M)$. This can be avoided by considering a larger space of curves, such as the set of absolutely continuous curves \cite{lahiri15}.
\end{remark}
\begin{figure}
\centering
\includegraphics[width=21em]{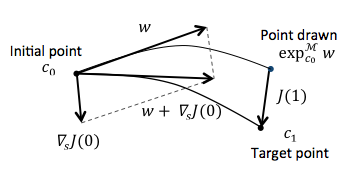}
\caption{Schematic illustration of geodesic shooting.}
\label{fig:schemgeodshoot}
\end{figure}
A possibility to construct the geodesics of $\mathcal M$ is to use geodesic shooting. By solving the geodesic equations \eqref{contgeodeq} we can construct the geodesic path starting from a given curve $c_0$ at a given speed $w\in T_{c_0}\mathcal M$ - this is the exponential map on $\mathcal M$. Given two curves $c_0, c_1$, geodesic shooting allows us to iteratively find the appropriate initial speed $w$ which will make the geodesic land on $c_1$. The idea is to "shoot" from $c_0$ in a certain direction using the exponential map, estimate the gap between the end point of the obtained geodesic and the target point $c_1$, "bring back" this information at $c_0$ using a Jacobi field and finally use this information to correct the shooting direction. 

Jacobi fields are vector fields that describe the way that geodesics spread out in the Riemannian manifold: for any geodesic $s\mapsto c(s)$ in $\mathcal M$ and Jacobi field $s\mapsto J(s)$ along $c$, there exists a family of geodesics $(-\delta, \delta) \ni a \mapsto c(a,\cdot)$ such that for all $s$, $c(0,s)=c(s)$ and 
\begin{equation*}
J(s) = \left.\frac{\partial}{\partial a}\right|_{a=0} c(a,s).
\end{equation*}
At a given step of the geodesic shooting algorithm, we consider the Jacobi field that measures the difference between the geodesic obtained by shooting and the desired geodesic between $c_0$ and $c_1$ : it has initial value $J(0)=0$ since both have same starting point $c_0$, and its end value can be estimated by $J(1)=\log^{L2}_{c_0}c_1$, where $\log^{L2}_c$ denotes the inverse of the exponential map for the $L^2$-metric on $\mathcal M$, simply given by $\log_{c_0}^{L2}(c_1)(t)=\log_{c_0(t)}(c_1(t))$ for $t\in[0,1]$. The shooting direction can then be corrected by the derivative $\nabla_sJ(0)$ of the Jacobi field at the origin, as shown in Figure \ref{fig:schemgeodshoot}. For more details, we refer the reader to  \cite{moi17}. 
\begin{algo}[Geodesic shooting] \label{alg:geodshoot}
\leavevmode\par \noindent
Input: $c_0, c_1\in \mathcal M$.\\
Initialization: $w = \log_{c_0}^{L2}(c_1)$.\\
Repeat until convergence :
\begin{enumerate}
\item compute the geodesic $s\mapsto c(s)$ starting from $c_0$ at speed $w$ by solving the geodesic equations \eqref{contgeodeq},
\item evaluate the difference $j := \log_{c(1)}^{L^2}(c_1)$ between the target curve $c_1$ and the extremity $c(1)$ of the obtained geodesic,
\item compute the initial derivative $\nabla_sJ(0)$ of the Jacobi field $s\mapsto J(s)$ along the path of curves $c$ verifying $J(0)=0$ and $J(1)=j$,
\item correct the shooting direction $w = w + \nabla_sJ(0)$.
\end{enumerate}
Output: geodesic $c$.
\end{algo}
This algorithm requires the characterization of the Jacobi fields for $G$ on $\mathcal M$, and a way to deduce the initial derivative $\nabla_s J(0)$ of a Jacobi field from its initial and final values $J(0)$, $J(1)$. Concerning these two points, we refer the reader to \cite{moi17} : the Jacobi fields of $\mathcal M$ are shown to be solutions of a linear PDE, which can be solved to obtain the final value $J(1)$  of a Jacobi field $J$ along a path of curves $c$ knowing its initial conditions $J(0)$ and $\nabla_s J(0)$. If we consider only Jacobi fields with initial value $J(0)=0$, then the function $\varphi : T_{c(0)}\mathcal M \rightarrow T_{c(1)}\mathcal M$, $\nabla_s J(0) \mapsto J(1)$ is a linear bijection between two vector spaces and its inverse map can be computed by considering the image of a basis of $T_{c(0)}\mathcal M$.
The equations characterizing the Jacobi fields in the discrete setting will be given in Section \ref{sec:dismodel}.

\subsection{The space of unparameterized curves}
\label{sec:quotient}

\subsubsection{The quotient structure}

In order to compare curves regardless of parameterization, we consider the quotient $\mathcal S = \mathcal M/\text{Diff}^+([0,1])$ of the space of curves by the diffeomorphisms group. This quotient is not a manifold, as it has singularities, i.e. points with non trivial isotropy group \cite{michor16}. That is why from now on we assume that $\mathcal M$ denotes the set of free immersions, i.e. immersions on which the diffeomorphism group acts freely. That way, the space of curves $\mathcal M$ and the quotient shape space $\mathcal S$ are respectively the total and base spaces of a principal bundle, the fibers of which are the sets of all the curves that are identical modulo reparameterization, i.e. that project on the same "shape". We denote by $\pi : \mathcal M\rightarrow S$ the projection of the fiber bundle and by $\bar c:=\pi(c)\in \mathcal S$ the shape of a curve $c\in \mathcal M$. The tangent bundle can then be decomposed
\begin{equation*}
T\mathcal M = \text{Ver} \oplus \text{Hor}
\end{equation*}
into a vertical subspace consisting of all vectors tangent to the fibers of $\mathcal{M}$ over $\mathcal{S}$, that is, those which have an action of reparameterizing the curve without changing its shape
\begin{align*}
\text{Ver}_c &= \ker T_c\pi= \{ m v, \,\,m \in \mathcal C\},\\
\mathcal C &= \{ m \in C^\infty([0,1], \mathbb R), \, m(0)=m(1)=0\},
\end{align*}
and a horizontal subspace defined as the orthogonal of the vertical subspace according to $G$
\begin{equation*}
\text{Hor}_c = \left(\text{Ver}_c \right)^{\perp_G} \!= \!\{ h\in T_c\mathcal M_f : G_c(h,mv)\!=0 \, \forall m\in \mathcal C\}.
\end{equation*}
If $G$ is constant along the fibers, i.e. verifies property \eqref{equivariance}, then there exists a Riemannian metric $\bar G$ on the shape space $\mathcal{S}$ such that $\pi$ is a Riemannian submersion from $(\mathcal M,G)$ to $\mathcal (S,\bar G)$,
\begin{equation*}
G_c(w^{hor},z^{hor}) = \bar G_{\pi( c)}\left( T_c\pi(w), T_c\pi(z) \right),
\end{equation*}
where $w^{hor}$ and $z^{hor}$ are the horizontal parts of tangent vectors $w$ and $z$, as well as the horizontal lifts of $T_c\pi(w)$ and $T_c\pi(z)$, respectively. This expression defines $\bar G$ in the sense that it does not depend on the choice of the representatives $c$, $w$ and $z$ (\cite{michor08}, \S 29.21). If a geodesic for $G$ has a horizontal initial speed, then its speed vector stays horizontal at all times - we say it is a horizontal geodesic - and projects on a geodesic of the shape space for $\bar G$ (\cite{michor08}, \S 26.12). To compute the distance between two shapes $\overline{c_0}$ and $\overline{c_1}$ in the quotient space we choose a representative $c_0$ of $\overline{c_0}$ and compute the distance (in $\mathcal M$) to the closest representative of $\overline{c_1}$, as shown schematically in Figure \ref{fig:fibre},
\begin{equation*}
\bar d\left( \overline{c_0} , \overline{c_1} \right) = \inf \left\{\, d\left(c_0, c_1\circ \varphi \right) \, | \, \, \varphi \in \text{Diff}^+([0,1]) \, \right\}.
\end{equation*}
\begin{figure}
\centering
\includegraphics[width=20em]{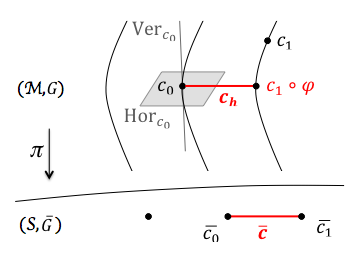}
\caption{Schematic representation of the shape bundle.}
\label{fig:fibre}
\end{figure}
By definition, the distance in the quotient space allows us to compare curves regardless of parameterization
\begin{equation*}
\bar d\left(\,\overline{c_0\circ\varphi},\overline{c_1\circ\psi}\,\right)=\bar d\left(\overline{c_0},\overline{c_1}\right),  \quad \forall \varphi,\psi\in\text{Diff}^+([0,1]).
\end{equation*}
We now characterize the horizontal subspace for any elastic metric $G^{a,b}$ and in particular for our metric $G^{1,\frac{1}{2}}$, and give the decomposition of a tangent vector.
\begin{proposition}[Horizontal part of a vector]
Let $c\in \mathcal M$ be a smooth immersion. Then $h \in T_c\mathcal M$ is horizontal for the elastic metric $G^{a,b}$ if and only if
\begin{align*}
&\left((a/b)^2-1\right)\langle \nabla_th, \nabla_tv\rangle - \langle \nabla_t^2h, v \rangle \\
&\hspace{9em}+ |c'|^{-1} \langle \nabla_tc', v \rangle\langle \nabla_th, v\rangle=0.
\end{align*}
In particular, for $a=2b=1$ we obtain
\begin{align*}
&\text{Hor}_c = \{ h \in T_c\mathcal M : \forall t\in[0,1], \,3\langle \nabla_th, \nabla_tv\rangle \\
&\hspace{5em}- \langle \nabla_t^2h, v \rangle + |c'|^{-1} \langle \nabla_tc', v \rangle\langle \nabla_th, v\rangle=0\}.
\end{align*}
Any tangent vector $w\in T_c\mathcal M$ can be decomposed in horizontal and vertical components $w = w^{hor} + w^{ver}$ given by $\,\,w^{ver} = mv,\,\, w^{hor} = w - mv$, where the real function $m\in \mathcal C$ verifies the ordinary differential equation
\begin{equation}
\label{mtt}
\begin{aligned}
&m'' - \langle \nabla_tc'/|c'|, v \rangle m' - 4 |\nabla_tv|^2 m\\
&\hspace{1em}= \langle \nabla_t^2w, v \rangle - 3\langle \nabla_tw, \nabla_tv\rangle -\langle \nabla_tc'/|c'|, v \rangle\langle \nabla_tw, v\rangle.
\end{aligned}
\end{equation}
\end{proposition}
\begin{proof}
Let $h \in T_c\mathcal M$ be a tangent vector. It is horizontal if and only if it is orthogonal to any vertical vector, that is any vector of the form $mv$ with $m\in C^\infty([0,1],\mathbb R)$ such that $m(0)=m(1)=0$. We have $\nabla_t(mv) = m'v+m\nabla_tv \,$ and since $\langle \nabla_tv, v\rangle =0$ we get $\nabla_t(mv)^N = m \nabla_tv$ and $\nabla_t(mv)^T = m'v$. The scalar product can then be written
\begin{align*}
&G^{a,b}_c(h,mv) = \langle h(0), m(0)v(0) \rangle \\
&+ \int_0^1 (a^2 \langle \nabla_th^N, \nabla_t(mv)^N \rangle + b^2 \langle\nabla_th^T, \nabla_t(mv)^T \rangle) \frac{\mathrm dt}{|c'|}\\
&=\int_0^1 (a^2 m \langle \nabla_th, \nabla_tv \rangle + b^2 m'\langle\nabla_th, v \rangle)  \frac{\mathrm dt}{|c'|}\\
&=\int_0^1 \!a^2 m \langle \nabla_th, \nabla_tv \rangle  \frac{\mathrm dt}{|c'|} - \!\int_0^1\!b^2 m \frac{d}{dt}\big(\langle \nabla_th, v\rangle |c'|^{-1}\big) \mathrm dt\\
&=\int_0^1 m/|c'|\Big( (a^2-b^2)\langle \nabla_th, \nabla_tv\rangle - b^2 \langle \nabla_t^2h, v \rangle\\
&\hspace{10.5em}+ b^2 \langle \nabla_tc', v \rangle\langle \nabla_th, v\rangle|c'|^{-1} \Big) \mathrm dt,
\end{align*}
by integration by parts. The vector $h$ is horizontal if and only if $G_c(h,mv)=0$ for all such $m$, and so multiplying by $|c'|/b^2$ gives the desired equation. Now consider a tangent vector $w$ and a real function $m : [0,1] \rightarrow \mathbb R$ such that $m(0)=m(1)=0$. Then according to the above, $w-mv$ is horizontal if and only if it verifies 
\begin{align*}
&3\langle \nabla_t(w-mv), \nabla_tv\rangle - \langle \nabla_t^2(w-mv), v \rangle \\
&\hspace{6.5em}+ |c'|^{-1} \langle \nabla_tc', v \rangle\langle \nabla_t(w-mv), v\rangle=0,
\end{align*}
i.e., since $\langle \nabla_tv,v\rangle=0$, $\langle \nabla_t^2v,v\rangle=-|\nabla_tv|^2$ and $\nabla_t^2(mv)= m'' v+ 2m'\nabla_tv +m\nabla_t^2v$, if
\begin{align*}
&3\langle \nabla_tw, \nabla_tv\rangle - 3 |\nabla_tv|^2 m - \langle \nabla_t^2w, v \rangle + m''\\
&\hspace{2em}- m |\nabla_tv|^2 + |c'|^{-1} \langle \nabla_tc', v \rangle(\langle \nabla_tw, v\rangle - m')=0,
\end{align*}
which is what we wanted.
\end{proof}

\subsubsection{Computing geodesics in the shape space} Recall that the geodesic path $s \mapsto \bar c(s)$ between the shapes of two curves $c_0$ and $c_1$ is the projection of the horizontal geodesic - if it exists - $s \mapsto c_h(s)$ linking $c_0$ to the fiber of $c_1$ in $\mathcal M$, i.e. such that $c_h(0)=c_0$, $c_h(1) \in\pi^{-1}(\overline{c_1})$ and $\partial_sc_h(s) \in \text{Hor}_{c_h(s)}$ for all $s\in[0,1]$,
\begin{equation*}
\bar c = \pi(c_h).
\end{equation*}
The end point of $c_h$ then gives the optimal reparameterization $c_1\circ \varphi$ of the target curve $c_1$ with respect to the initial curve $c_0$, i.e. such that
\begin{equation*}
\bar d(\overline{c_0},\overline{c_1}) = d(c_0,c_1\circ \varphi).
\end{equation*}
The parameterized curves $(c_0,c_1\circ \varphi)$ define what we call an \emph{optimal matching} between the shapes $\overline{c_0}$ and $\overline{c_1}$, in the sense that each point $c_0(t)$ on $\overline{c_0}$ is matched with the point $c_1(\varphi(t))$ on $\overline{c_1}$. Here we propose a method to approach the horizontal geodesic $c_h$, and thereby the corresponding optimal matching. To that end we decompose any path of curves $s\mapsto c(s)$ in $\mathcal M$ into a horizontal path composed with a path of reparameterizations, $c(s) = c^{hor}(s) \circ \varphi(s)$, or equivalently
\begin{equation}
\label{def}
c(s,t) = c^{hor}(s,\varphi(s,t)) \quad \forall s,t\in[0,1],
\end{equation}
where the path $[0,1]\ni s\mapsto c^{hor}(s) \in\mathcal M$ is such that $c^{hor}_s(s)\in \text{Hor}_{c^{hor}(s)}$ for all $s\in[0,1]$, and $[0,1] \ni s \mapsto \varphi(s) \in \text{Diff}^+([0,1])$ is a path of increasing diffeomorphisms. The horizontal and vertical parts of the speed vector of $c$ can be expressed in terms of this decomposition. Indeed, by taking the derivative of \eqref{def} with respect to $s$ and $t$ we obtain
\begin{subequations}
\begin{align}
c_s(s) &= c^{hor}_s(s) \circ \varphi(s) + \varphi_s(s) \cdot c^{hor}_t(s)\circ \varphi(s), \label{cs}\\
c_t(s) &= \varphi_t(s) \cdot c^{hor}_t(s) \circ \varphi(s), \label{ct}
\end{align}
\end{subequations}
and so with $v^{hor}(s,t) := c^{hor}_t(s,t)/|c^{hor}_t(s,t)|$, since $\varphi_t >0$, \eqref{ct} gives
\begin{equation*}
v(s) = v^{hor}(s) \circ \varphi(s).
\end{equation*}
We can see that the first term on the right-hand side of Equation \eqref{cs} is horizontal. Indeed, for any path of real functions $m : [0,1] \rightarrow C^\infty([0,1],\mathbb R)$, $s\mapsto m(s,\cdot)$ such that $m(s,0)=m(s,1)=0$ for all $s$, since $G$ is reparameterization invariant we have
\begin{align*}
&G\left(c^{hor}_s(s) \circ \varphi(s), \,m(s) \cdot v(s) \right)\\
&= G\left(c^{hor}_s(s)\circ \varphi(s), \,m(s) \cdot v^{hor}(s)\circ \varphi(s) \right)\\
&= G\left(c^{hor}_s(s), \,m(s) \circ \varphi(s)^{-1} \cdot v^{hor}(s) \right)\\
&= G\left(c^{hor}_s(s), \, \tilde m(s)\cdot v^{hor}(s)\right)
\end{align*}
with $\tilde m(s) = m(s) \circ \varphi(s)^{-1}$. Since $\tilde m(s,0) = \tilde m(s,1) = 0$ for all $s$, the vector $\tilde m(s)\cdot v^{hor}(s)$ is vertical and its scalar product with the horizontal vector $c^{hor}_s(s)$ vanishes. On the other hand, the second term on the right hand-side of Equation \eqref{cs} is vertical, since it can be written
\begin{equation*}
\varphi_s(s)\cdot c^{hor}_t\circ \varphi(s) = m(s)\cdot v(s),
\end{equation*}
with $m(s)=|c_t(s)|\varphi_s(s)/\varphi_t(s)$ verifying $m(s,0)=m(s,1)=0$ for all $s$. Finally, the vertical and horizontal parts of the speed vector $c_s(s)$ are given by
\begin{subequations}
\begin{align}
c_s(s)^{ver} &= m(s) \cdot v(s) = |c_t(s)|\varphi_s(s)/\varphi_t(s) \cdot v(s), \label{csver}\\
c_s(s)^{hor} &= c_s(s) - m(s)\cdot v(s) = c^{hor}_s(s)\circ \varphi(s).\label{cshor}
\end{align}
\end{subequations}
\begin{definition}
We call $c^{hor}$ the \emph{horizontal part} of the path $c$ with respect to $G$.
\end{definition}
\begin{proposition}
The horizontal part of a path of curves $c$ is at most the same length as $c$  
\begin{equation*} L_G(c^{hor}) \leq L_G(c).
\end{equation*}
\end{proposition}
\begin{proof}
Since the metric $G$ is reparameterization invariant, the squared norm of the speed vector of the path $c$ at time $s\in[0,1]$ is given by
\begin{align*}
&\|c_s(s,\cdot)\|_G^2 \\
&= \|c^{hor}_s(s,\varphi(s,\cdot))\|_G^2 + |\varphi_s(s,\cdot)|^2\|c^{hor}_t(s,\varphi(s,\cdot)\|_G^2\\
&=\|c^{hor}_s(s,\cdot)\|_G^2 + |\varphi_s(s,\cdot)|^2 \| c^{hor}_t(s,\cdot)\|_G^2,
\end{align*}
where $\|\cdot\|^2_G:=G(\cdot,\cdot)$. This gives $\|c^{hor}_s(s)\|_G\leq \|c_s(s)\|$ for all $s$ and so $L_G(c^{hor})\leq L_G(c)$.
\end{proof}
Now we will see how the horizontal part of a path of curves can be computed.
\begin{proposition}[Horizontal part of a path]
Let $s\mapsto c(s)$ be a path in $\mathcal M$. Then its horizontal part is given by $c^{hor}(s,t) = c(s,\varphi(s)^{-1}(t))$, where the path of diffeomorphisms $s\mapsto \varphi(s)$ is solution of the partial differential equation
\begin{equation}
\label{phicont}
\varphi_s(s,t) = m(s,t)/|c_t(s,t)|\cdot \varphi_t(s,t),
\end{equation}
and where $m(s) : [0,1] \rightarrow \mathbb R$, $t \mapsto m(s,t)$ is solution for all $s$ of the ordinary differential equation
\begin{align*}
&m_{tt} - \langle \nabla_tc_t/|c_t|, v \rangle m_t - 4 |\nabla_tv|^2 m \\
&\hspace{1em}= \langle \nabla_t^2c_s, v \rangle - 3\langle \nabla_tc_s, \nabla_tv\rangle -\langle \nabla_tc_t/|c_t|, v \rangle\langle \nabla_tc_s, v\rangle.
\end{align*}
\end{proposition}
\begin{proof}
We have seen in Equation \eqref{csver} that the vertical part of $c_s(s)$ can be written as $m(s)\cdot v(s)$ where $m(s) = |c_t(s)|\varphi_s(s)/\varphi_t(s)$, and as the norm of the vertical part of $c_s(s)$, $m(s)$ is solution of the ODE \eqref{mtt} for all $s$.
\end{proof}

\begin{figure}\label{fig:optmatch}
\centering
\includegraphics[width=20em]{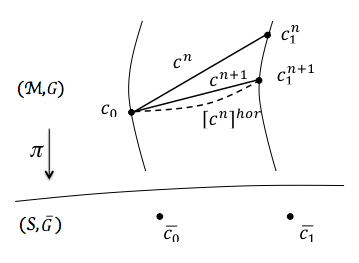}
\caption{Schematic illustration of the $n^{th}$ step of the optimal matching algorithm.}
\label{fig:optmatch}
\end{figure}
If we take the horizontal part of the geodesic linking two curves $c_0$ and $c_1$, we will obtain a horizontal path linking $c_0$ to the fiber of $c_1$ which will no longer be a geodesic path. However this path reduces the distance between $c_0$ and the fiber of $c_1$, and gives a "better" representative $c_1\circ \varphi(1)$ of the target curve. By computing the geodesic between $c_0$ and this new representative, we are guaranteed to reduce once more the distance to the fiber. The algorithm that we propose simply iterates these two steps, as illustrated in Figure \ref{fig:optmatch}.
\begin{algo}[Optimal matching] \label{alg:optmatch}
\leavevmode\par \noindent
Input: $c_0,c_1\!\in \mathcal M$.\\
Initialization: $n=0$, $c_1^n= c_1$.\\
Repeat until convergence:
\begin{enumerate}
\item construct the geodesic $s\mapsto c^n(s)$ between $c_0$ and $c_1^n$,
\item compute the horizontal part $s\mapsto [c^n]^{hor}(s)$ of $c^n$,
\item set $c^{n+1}_1 = [c^n]^{hor}(1)$,
\item set $n= n+1$.
\end{enumerate}
Output: horizontal geodesic $c^n$.
\end{algo}
Let us specify why the obtained geodesic is horizontal at the limit. The series of lengths $\left(L(c^n)\right)_{n\geq 0}$ and $\left(L([c^n]^{hor})\right)_{n\geq0}$ are non negative decreasing and verify at each step $n$
\begin{equation*}
L(c^n)\geq L([c^n]^{hor}) \geq L(c^{n+1}),
\end{equation*}
which means that they converge to the same limit. The same is true for the energies $E(c^n)$ and $E([c^n]^{hor})$, and since the $s$-derivative of $[c^n]^{hor}$ is equal to the horizontal part of the $s$-derivative of $c^n$, we get
\begin{equation*}
\int \|c^n_s(s)^{ver}\|^2_G\, \mathrm ds = E(c^n)-E([c^n]^{hor}) \underset{n\to\infty}{\longrightarrow} 0,
\end{equation*}
and so $\|c^n_s(s)^{ver}\|_G$ converges to zero for almost all $s$. Since it is enough that a geodesic be horizontal at one given time for it to be horizontal for all time $s$ (\cite{michor08}, \S 26.12), we have the following result.
\begin{proposition}[Horizontality of the solution] At the limit, the geodesic between the fibers computed in Algorithm \ref{alg:optmatch} is horizontal
\begin{equation*}
\forall s \quad \|c^n_s(s)^{ver}\| \underset{n\to\infty}{\longrightarrow} 0.
\end{equation*}
\end{proposition}
\begin{remark}
In this work, we will carry out step 1 using geodesic shooting. However it is important to stress that Algorithm \ref{alg:optmatch} is a general method that can be applied to any metric structure (not only elastic metrics) for which one knows how to compute geodesics and characterize the horizontal subspace of the shape bundle. It can be seen as an alternative method to the popular dynamic programming approach \cite{srivastava16}, with which we establish comparisons in Section \ref{sec:simu}. But before that, let us first introduce a formal discretization of the continuous model presented so far.
\end{remark}

\section{The discrete model}				
\label{sec:dismodel}

Applications usually give access to a finite number of observations of a continuous process and provide series of points instead of continuous curves. It is therefore important to discretize the framework presented above and to consider the finite-dimensional space of "discrete curves". From now on we restrict to base manifolds $M$ of constant sectional curvature $K$. This allows us to get an explicit formula for the Jacobi fields of $M$ (Lemma \ref{lemjacobi}) and thus derive a precise approximation of the smooth Riemannian structure of Section \ref{sec:contmodel}. Generalization to any Riemannian manifold for which the Jacobi fields are computable should not be problematic. For more complex manifolds, a faster and more approximate solution would be to directly discretize the smooth equations, at the cost of the precision of the discrete approximation. 

\subsection{The Riemannian structure}

We consider the product manifold $M^{n+1}$ of "discrete curves" given by $n+1$ points, for a fixed $n\in \mathbb N^*$. Its tangent space at a given point $\alpha = (x_0, \hdots, x_n)$ is
\begin{equation*}
T_\alpha M^{n+1} \!=\! \{ w = (w_0, \hdots, w_n) : w_k \in T_{x_k}M, \,\forall k \}.
\end{equation*}
Assuming that there exists a connecting geodesic between $x_k$ and $x_{k+1}$ for all $k$ -- which seems reasonable considering that the points $x_k$ should be "close" since they correspond to the discretization of a continuous curve -- and that two consecutive points are always distinct ($x_k\neq x_{k+1}$), we use the following notations 
\begin{equation}
\label{not}
\tau_k = \log_{x_k} x_{k+1}, \,\, q_k = \sqrt{n}\, \tau_k/\sqrt{|\tau_k|}, \,\, v_k = \tau_k/|\tau_k|,
\end{equation}
as well as ${w_k}^T = \langle w_k,v_k \rangle v_k$ and $w_k^N = w_k - {w_k}^T$ to refer to the tangential and normal components of a tangent vector $w_k\in T_{x_k}M$. Given a tangent vector $w\in T_\alpha M^{n+1}$, we consider a path of piecewise geodesic curves $[0,1]^2\ni(s,t) \mapsto c^w(s,t)\in M$ such that $c^w(0,\frac{k}{n})=x_k$ for $k=0,\hdots,n$, $c^w(s,\cdot)$ is a geodesic of $M$ on the interval $[\frac{k}{n},\frac{k+1}{n}]$ for all $s\in[0,1]$ and $k$ -- and in particular $c^w_t(0,\frac{k}{n})=n\tau_k$ -- and such that $c^w_s(0,\frac{k}{n})=w_k$, as illustrated in Figure \ref{fig:discret}. Then we define the squared norm of $w$ by
\begin{figure}
\centering
\includegraphics[width=3.5cm]{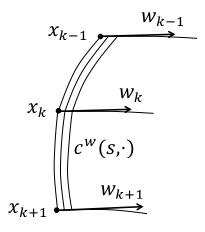}
\caption{Schematic representation of a path of piecewise geodesic curve associated to a pair $(\alpha,w)$.}
\label{fig:discret}
\end{figure}
\begin{equation}
\label{dismetric}
G^n_\alpha(w,w) = |w_0|^2 + \frac{1}{n} \sum_{k=0}^{n-1} | \nabla_s q^w(0,\tfrac{k}{n})|^2.
\end{equation}
This definition is a discrete analog of \eqref{contmetric}, and just as in the continuous case, it does not depend on the choice of $c^w$. Indeed, we can also obtain a discrete analog of \eqref{contmetricbis}.
\begin{proposition}
The metric $G^n$ can also be written
\begin{equation*}
G_\alpha^n(w,w) = |w_0|^2 + \sum_{k=0}^{n-1} \Big( \big|(D_\tau w)_k^N \big|^2+ \tfrac{1}{4} \big|(D_\tau w)_k^T\big|^2 \Big)\frac{1}{|\tau_k|},
\end{equation*}
where the map $D_\tau : T_\alpha M^{n+1} \rightarrow T_\alpha M^{n+1}$, $w \mapsto D_\tau w =\big((D_\tau w)_0, \hdots, (D_\tau w)_n\big)$ is defined by
\begin{align*}
(D_\tau w)_k &:=\tfrac{1}{n} \nabla_t c^w_s(0,\tfrac{k}{n})\\
&= ( {w_{k+1}}^\parallel - w_k)^T + b_k^{-1}( {w_{k+1}}^\parallel - a_k w_k)^N,
\end{align*}
if $w_{k+1}^\parallel$ denotes the parallel transport of $w_{k+1}$ from $x_{k+1}$ to $x_k$ along the geodesic, and the coefficients $a_k$ and $b_k$ take the following values depending on the sectional curvature $K$ of the base manifold $M$
\begin{equation}
\label{akbk}
\left\{ \begin{matrix*}[l] 
a_k =\cosh|\tau_k|, \,\, & b_k=\sinh|\tau_k|/|\tau_k|,  \,\,& \text{if } K=-1,\\ 
a_k=1, \,\, & b_k= 1, \,\,&\text{if } K=0,\\ 
a_k=\cos|\tau_k|, \,\, & b_k=\sin|\tau_k|/|\tau_k|, \,\,& \text{if } K=+1.
\end{matrix*} \right.
\end{equation}
\end{proposition}
\begin{remark}
Notice that in the flat case our definition gives $(D_\tau w)_k = w_{k+1} - w_k$. In the non-flat case, when the discretization gets "thinner", i.e. $n\rightarrow \infty$ and $|\tau_k|\rightarrow 0$ while $n|\tau_k|$ stays bounded for all $0\leq k\leq n$, we get $(D_\tau w)_k \underset{n\rightarrow\infty}{=} {w_{k+1}}^\parallel - w_k + o(1)$.
\end{remark}
Before we prove this proposition, let us recall a well-known result about Jacobi fields that will prove useful to derive the equations in the discrete case.
 \begin{lemma}
\label{lemjacobi}
Let $\gamma : [0,1] \rightarrow M$ be a geodesic of a manifold $M$ of constant sectional curvature $K$, and $J$ a Jacobi field along $\gamma$. Then the parallel transport of $J(t)$ along $\gamma$ from $\gamma(t)$ to $\gamma(0)$ is given by 
\begin{align*}
J(t)^{t,0} &= J^T(0) + \tilde a_k(t) J^N(0)  \\
&\hspace{7em}+ t \,\nabla_tJ^T(0) +  \tilde b_k(t)\nabla_tJ^N(0),
\end{align*}
for all $t\in[0,1]$, where
\begin{equation*}
\left\{\begin{matrix*}[l]
\tilde a_k(t)=\cosh\left( |\gamma'(0)| t \right), &\tilde b_k(t)=\frac{\sinh(|\gamma'(0)| t)}{|\gamma'(0)|},  & K=-1,\\
\tilde a_k(t)=1, &\tilde b_k(t)= t, & K=0,\\
\tilde a_k(t)=\cos\left( |\gamma'(0)| t \right), &\tilde b_k(t)=\frac{\sin(|\gamma'(0)| t)}{|\gamma'(0)|}, & K=+1.
\end{matrix*} \right.
\end{equation*}
\end{lemma}
\begin{proof}[Proof of Lemma 1]
For the sake of completeness, the proof is reminded in the appendix.
\end{proof}
\begin{proof}[Proof of Proposition 6]
Let $\alpha\in M^{n+1}$ be a "discrete curve" and $w\in T_\alpha M^{n+1}$ a tangent vector at $\alpha$. Consider a path of piecewise geodesic curves $s\mapsto c^w(s)$ that verifies all the conditions given above to define $G^n(w,w)$, and set $(D_\tau w)_k :=\frac{1}{n} \nabla_t c^w_s(0,\frac{k}{n})$. Then by definition, the vector field $J_k(u) = c^w_s(s,\frac{k+u}{n})$, $u\in[0,1]$ is a Jacobi field along the geodesic linking $x_k$ to $x_{k+1}$, verifying $J_k(0)=w_k$, $J_k(1)=w_{k+1}$ and $\nabla_uJ_k(0) = (D_\tau w)_k$. Applying Lemma \ref{lemjacobi} gives 
\begin{equation*}
{w_{k+1}}^\parallel = {w_k}^T + a_k{w_k}^N + {(D_\tau w)_k}^T + b_k{(D_\tau w)_k}^N.
\end{equation*}
This gives $({w_{k+1}}^\parallel)^T = {w_k}^T + {(D_\tau w)_k}^T$ and $({w_{k+1}}^\parallel)^N = a_k{w_k}^N + b_k{(D_\tau w)_k}^N$ and so
\begin{align*}
(D_\tau w)_k &= {(D_\tau w)_k}^T+{(D_\tau w)_k}^N\\
&=( {w_{k+1}}^\parallel - w_k)^T + b_k^{-1}( {w_{k+1}}^\parallel - a_k w_k)^N.
\end{align*}
Finally, we observe that the covariant derivative involved in the definition of $G^n$ can be written
\begin{align*}
\nabla_s q^w(0,\tfrac{k}{n})&=| c^w_t(0,\tfrac{k}{n})|^{-\frac{1}{2}}(\nabla_s c^w_t(0,\tfrac{k}{n}) - \tfrac{1}{2}\nabla_s c^w_t(0,\tfrac{k}{n})^T)\\
&=|n\tau_k|^{-\frac{1}{2}}\big(n(D_\tau w)_k - \tfrac{1}{2}n{(D_\tau w)_k}^T), 
\end{align*}
i.e.
\begin{equation*}
\nabla_s q^w(0,\tfrac{k}{n}) = (n/|\tau_k|)^{1/2}\big({(D_\tau w)_k}^N + \tfrac{1}{2}{(D_\tau w)_k}^T\big). 
\end{equation*}
Injecting this into \eqref{dismetric} gives the desired formula.
\end{proof}
Now we present the main result of this section, that is, the convergence of the discrete model toward the continuous model. 
\begin{definition}
Let $\alpha=(x_0,\hdots, x_n)\in M^{n+1}$ be a discrete curve, and $t\mapsto c(t)\in M$ a smooth curve. We say that $\alpha$ is \emph{the discretization of size $n$} of $c$ when $c(\frac{k}{n})=x_k$ for all $k=0,\hdots,n$. If $s\mapsto \alpha(s)=(x_0(s), \hdots, x_n(s)) \in M^{n+1}$ is a path of discrete curves and $s\mapsto c(s) \in \mathcal M$ a path of smooth curves, then $\alpha$ is \emph{the discretization of size $n$} of $c$ when $\alpha(s)$ is the discretization of $c(s)$ for all $s\in [0,1]$, i.e. when $x_k(s)=c(s,\frac{k}{n})$ for all $s$ and $k$. We will still use this term if $c$ is not smooth, and speak of the only path of piecewise-geodesic curves of which $\alpha$ is the discretization.
\end{definition}
Let $[0,1]\ni s\mapsto \alpha(s)=(x_0(s),\hdots,x_n(s))\in M^{n+1}$ be a path of discrete curves. Defining $\tau_k(s)$ and $q_k(s)$ as in \eqref{not} for all $s\in[0,1]$, the path $\alpha$ can be represented by its SRV representation $[0,1]\rightarrow M\times T_\alpha{M^{n+1}}$,
\begin{equation}
\label{SRVdis}
s\mapsto \Big(x_0(s), (q_k(s))_{0\leq k\leq n-1}\Big).
\end{equation}
To compute the squared norm of its speed vector $\alpha'(s)$, consider the path of piecewise geodesic curves $[0,1]^2\ni(s,t) \mapsto c(s,t)\in M$ such that $c(s,\frac{k}{n})=x_k(s)$ and $c_t(s,\frac{k}{n})=n\tau_k(s)$ for all $s$ and $k$. Then, notice that we have
\begin{equation}
\label{dtaualpha}
\begin{aligned}
\nabla_sq(s,\tfrac{k}{n}) &= \nabla_sq_k(s),\\
(D_\tau \alpha'(s))_k &= \tfrac{1}{n} \nabla_t c_s(s,\tfrac{k}{n}) =\nabla_s\tau_k(s),
\end{aligned}
\end{equation}
and so the squared norm of the speed vector of $\alpha$ can be expressed in terms of the SRV representation
\begin{equation*}
G^n(\alpha'(s),\alpha'(s)) = |{x_0}'(s)|^2 + \frac{1}{n} \sum_{k=0}^{n-1} |\nabla_sq_k(s) |^2.
\end{equation*}
In the following result, we show that if $s\mapsto \alpha(s)$ is the discretization of a path $s\mapsto c(s) \in \mathcal M$ of continuous curves, then its energy with respect to $G^n$,
\begin{equation}
\label{disenergy}
E^n(\alpha) = \frac{1}{2} \int_0^1 \left( |x_0'(s)|^2 + \frac{1}{n} \sum_{k=0}^{n-1} |\nabla_sq_k(s) |^2 \right) \mathrm ds,
\end{equation}
gets closer to the energy \eqref{contenergy} of $c$ with respect to $G$ as the size of the discretization grows.
\begin{theorem}[Convergence of the discrete model to the continuous model]
\label{thm}
Let $s \mapsto c(s)$ be a $C^1$-path of $C^2$-curves with non vanishing derivative with respect to $t$. This path can be identified with an element $(s,t)\mapsto c(s,t)$ of $C^{1,2}([0,1]\times[0,1],M)$ such that $c_t \neq 0$. Consider the $C^1$-path in $M^{n+1}$, $s \mapsto \alpha(s)=(x_0(s), \hdots, x_n(s))$, that is the discretization of size $n$ of $c$. Then there exists a constant $\lambda>0$ such that for $n$ large enough, the difference between the energies of $c$ and $\alpha$ is bounded by
\begin{equation*}
| E(c) - E^n(\alpha)| \leq \frac{\lambda}{n}\, (\inf |c_t|)^{-1} |c_s|_{2,\infty}^2 \left(1+|c_t|_{1,\infty}\right)^3,
\end{equation*}
where $E$ and $E^n$ are the energies with respect to metrics $G$ and $G^n$ respectively and where
\begin{align*}
|c_t|_{1,\infty} &:= |c_t|_\infty + |\nabla_tc_t|_\infty,\\
|c_s|_{2,\infty} &:= |c_s|_\infty + |\nabla_tc_s|_\infty + |\nabla^2_tc_s|_\infty,
\end{align*}
if $|w|_\infty \!:= \!\!\!\!\underset{s,t\in[0,1]}{\sup}|w(s,t)|$ denotes the supremum over both $s$ and $t$ of a vector field $w$ along $c$.
\end{theorem}
\begin{remark}
Note that since we assume that $c$ is a $C^1$-path of $C^2$-curves, the following norms are bounded for $i=1,2$,
\begin{equation*}
|c_t|_\infty, \,|c_s|_\infty, \,|\nabla^i_tc_t|_\infty, \,|\nabla^i_tc_s|_\infty <\infty.
\end{equation*}
\end{remark}
\begin{proof}[Proof of Theorem \ref{thm}]
The proof is put off to Section \ref{sec:proofthm}.
\end{proof}
Now that we have established a formal Riemannian setting to study discrete curves defined by a series of points, and that we have studied its link to the continuous model, we need to derive the equations of the corresponding geodesics and Jacobi fields to apply the methods described in Section \ref{sec:contmodel}. For the sake of readability, we first introduce some notations.

\subsection{Computing geodesics in the discrete setting}

\subsubsection{Notations}

The purpose of the notations that we introduce here is to lighten the equations derived in the rest of the paper. 
For any discrete curve $\alpha=(x_0,\hdots,x_n)\in M^{n+1}$ we define for all $0\leq k \leq n$, using the coefficients $a_k$ and $b_k$ defined by \eqref{akbk} and \eqref{not}, the functions $f_k, g_k : T_{x_k}M \rightarrow T_{x_k}M$,
\begin{align*}
f_k &: w \,\mapsto\, w^T + a_k {w}^N,\\
g_k &: w \,\mapsto\, |q_k| ( 2w^T + b_k w ^N ).
\end{align*}
and for $k=0,\hdots, n-1$, 
the functions $f_k^{(-)}, g_k^{(-)} : T_{x_{k+1}}M \rightarrow T_{x_k}M$ by
\begin{align*}
f_k^{(-)} = f_k \circ P^{x_{k+1},x_k}_{\gamma_k}, \quad g_k^{(-)} = g_k \circ P^{x_{k+1},x_k}_{\gamma_k},
\end{align*}
where $\gamma_k$ denotes the geodesic between $x_k$ and $x_{k+1}$, which we previously assumed existed. Notice that when the discretization gets "thinner", that is $n\rightarrow \infty$, $|\tau_k|\rightarrow 0$ while $n|\tau_k|$ stays bounded for all $0\leq k\leq n$, we get in the non flat setting, for any fixed $w\in T_{x_{k+1}}M$, $f_k(w) = w + o(1/n)$ and $g_k(w) = |q_k| (w + {w}^T)+ o(1/n)$ - in the flat setting, these are always equalities. Now if we consider a path $s\mapsto \alpha(s)=(x_0(s),\hdots,x_n(s))$ of discrete curves, we can define for each $s$ the functions 
\begin{align*}
&f_k(s), \,g_k(s) :T_{x_k(s)}M \rightarrow T_{x_k(s)}M,\\
&f_k(s)^{(-)}, \,g_k(s)^{(-)} : T_{x_{k+1}(s)}M \rightarrow T_{x_k(s)}M, 
\end{align*}
for $0\leq k\leq n$ and $0\leq k \leq n-1$ respectively, corresponding to the discrete curve $\alpha(s)$. It is of interest for the rest of this paper to compute the covariant derivatives of these maps with respect to $s$. 
\begin{lemma}
\label{lem:ns}
The first and second order covariant derivatives of $f_k$ and $g_k$ with respect to $s$ are functions $T_{x_k(s)}M \rightarrow T_{x_k(s)}M$ defined by
\begin{align*}
&\nabla_sf_k(w)= \partial_sa_k w^N \\
&\hspace{2em}+ (1 - a_k)  \big( \langle w, \nabla_sv_k\rangle v_k + \langle w,v_k\rangle \nabla_sv_k \big),\\
&\nabla_sg_k(w)= \partial_s|q_k|/|q_k|  g_k(w) + |q_k| \partial_sb_k w^N \\
&\hspace{2em}+ |q_k|( 2 - b_k)  \big( \langle w,\nabla_sv_k\rangle v_k + \langle w,v_k \rangle \nabla_sv_k \big),\\
&\nabla_s^2f_k (w)=\partial^2_sa_k w^N - 2 \partial_sa_k\big(\langle w,\nabla_sv_k\rangle v_k \\
&\hspace{2em}+ \langle w,v_k\rangle \nabla_sv_k\big)+(1-a_k)  \big(\langle w,\nabla_s^2v_k\rangle v_k\\
& \hspace{2em}+ 2\langle w,\nabla_sv_k\rangle \nabla_sv_k + \langle w,v_k\rangle \nabla_s^2v_k\big),\\
&\nabla_s^2g_k(w)= \partial_s\big(\partial_s|q_k|/|q_k|\big)  g_k(w)+ \partial_s|q_k|/|q_k|  \nabla_sg_k(w)  \\
&\hspace{2em}+(\partial_s|q_k|\partial_sb_k+|q_k|\partial^2_sb_k) w^N\\
&\hspace{2em}+ \,|q_k|(2-b_k) \big(\langle w,\nabla_s^2v_k\rangle v_k + 2\langle w, \nabla_sv_k\rangle \nabla_sv_k \\
&\hspace{2em}+ \langle w,v_k \rangle \nabla_s^2v_k \big) + \big(\partial_s|q_k|(2-b_k)  \\
&\hspace{2em}- 2|q_k|\partial_sb_k\big) \big(\langle w,\nabla_sv_k\rangle v_k + \langle w,v_k\rangle \nabla_sv_k\big).
\end{align*}
\end{lemma}
\begin{proof}
For any vector field $s\mapsto w(s)\in T_{x_k(s)}M$ along $s\mapsto x_k(s)$ we have by definition
\begin{align*}
&\nabla_s\big(f_k(w)\big) = \nabla_sf_k(w) + f_k(\nabla_sw), \\
&\nabla_s\big(g_k(w)\big) = \nabla_sg_k(w) + g_k(\nabla_sw),\\
&\nabla_s^2\big( f_k(w) \big)= \nabla_s^2f_k(w) + 2\, \nabla_sf_k(\nabla_sw) + f_k(\nabla_s^2 w),\\
&\nabla_s^2\big( g_k(w) \big) = \nabla_s^2g_k(w) + 2\, \nabla_sg_k(\nabla_sw) + g_k(\nabla_s^2w).
\end{align*}
Noticing that $\nabla_s({w}^T)=(\nabla_sw)^T + \langle w,\nabla_sv_k\rangle v_k + \langle w,v_k\rangle \nabla_sv_k$ and $\nabla_s(w^N)=\nabla_sw - \nabla_s({w}^T)$, the formulas given in Lemma \ref{lem:ns} result from simple calculation.
\end{proof}
Using these functions, we can deduce the covariant derivatives of $f_k^{(-)}$ and $g_k^{(-)}$. Denoting by $\gamma_k(s)$ the geodesic of $M$ linking $x_k(s)$ to $x_{k+1}(s)$ for all $s\in[0,1]$ and $0\leq k \leq n-1$, we have the following result.
\begin{lemma}
\label{lemfun}
The covariant derivatives of the functions $f_k^{(-)}$ and $g_k^{(-)}$ with respect to $s$ are functions $T_{x_{k+1}(s)}M \rightarrow T_{x_k(s)}M$ given by
\begin{align*}
\nabla_s\big(f_k^{(-)}\big) &: w \mapsto (\nabla_sf_k)^{(-)}(w) + f_k\big(\mathcal R\left(Y_k, \tau_k \right)({w_{k+1}}^\parallel)\big),\\
\nabla_s\big(g_k^{(-)}\big) &: w \mapsto (\nabla_sg_k)^{(-)}(w) + g_k\big(\mathcal R\left(Y_k, \tau_k \right)({w_{k+1}}^\parallel)\big),
\end{align*}
where 
\begin{align}
&(\nabla_sf_k)(s)^{(-)} = \nabla_sf_k(s) \circ P^{x_{k+1}(s),x_k(s)}_{\gamma_k(s)},\nonumber\\
&(\nabla_sg_k)(s)^{(-)} = \nabla_sg_k(s) \circ P^{x_{k+1}(s),x_k(s)}_{\gamma_k(s)},\nonumber\\
&Y_k = ({x_k}')^T + b_k ({x_k}')^N + \tfrac{1}{2}{\nabla_s\tau_k}^T + K\frac{1-a_k}{|\tau_k|^2} \,{\nabla_s\tau_k}^N,\label{yk}
\end{align}
if $K$ is the sectional curvature of the base manifold.
\end{lemma}
\begin{proof} The proof is given in Appendix A. \end{proof}

\subsubsection{Geodesic equations and exponential map}

With these notations, we can characterize the geodesics for metric $G^n$. The geodesic equations can be derived in a similar way as in the continuous case, that is by searching for the critical points of the energy \eqref{disenergy}. We obtain the following characterization in terms of the SRV representation \eqref{SRVdis}.
\begin{proposition}[Discrete geodesic equations]
\label{prop:geodeq}
A path $s \mapsto \alpha(s) = \left( x_0(s), \hdots, x_n(s) \right)$ in $M^{n+1}$ is a geodesic for metric $G^n$ if and only if its SRV representation $s\mapsto \big(x_0(s), (q_k(s))_{k}\big)$ verifies the following differential equations
\begin{equation}
\label{disgeodeq}
\begin{aligned}
&\nabla_s{x_0}' + \frac{1}{n} \Big( R_0 + f_0^{(-)}(R_1)  \\
&\hspace{7em}+ \hdots + f_0^{(-)}\circ \cdots \circ f_{n-2}^{(-)} (R_{n-1})\Big)  = 0, \\
&\nabla_s^2q_k + \frac{1}{n} \,\,g_k^{(-)}\Big( R_{k+1} + f_{k+1}^{(-)}(R_{k+2}) \\
& \hspace{7em}+ \hdots + f_{k+1}^{(-)} \circ \cdots \circ f_{n-2}^{(-)}(R_{n-1})\Big) = 0,
\end{aligned}
\end{equation}
for all $k=0, \hdots, n-1$, with the notations \eqref{not} and $R_k := \mathcal R(q_k,\nabla_sq_k){x_k}'$.
\end{proposition}
\begin{proof} The proof is given in Appendix B.\end{proof}
\begin{remark}[Link with the continuous setting]
Let $[0,1]\ni s\mapsto c(s,\cdot)\in \mathcal M$ be a $C^1$ path of smooth curves and $[0,1]\ni s\mapsto \alpha(s)\in M^{n+1}$ the discretization of size $n$ of $c$. We denote as usual by $q:=c_t/|c_t|^{1/2}$ and $(q_k)_k$ their respective SRV representations. When $n\rightarrow \infty$ and $|\tau_k|\rightarrow 0$ while $n|\tau_k|$ stays bounded for all $0\leq k\leq n$, the coefficients of the discrete geodesic equation \eqref{disgeodeq} for $\alpha$ converge to the coefficients of the continuous geodesic equation \eqref{contgeodeq} for $c$, i.e.
\begin{align*}
&\nabla_s{x_0}'(s) = - r_0(s) + o(1),\\
&\nabla_s^2q_k(s) = - |q_k(s)| (r_{k}(s) + r_{k}(s)^T) +o(1),
\end{align*}
for all $s\in[0,1]$ and $k=0,\hdots, n-1$, where $r_{n-1}=0$ and for $k = 1,\hdots, n-2$,
\begin{equation*}
r_k(s) := \frac{1}{n} \!\sum_{\ell=k+1}^{n-1} \!\!P_c^{\frac{l}{n},\frac{k}{n}}\big( \mathcal R(q, \nabla_sq)c_s(s,\tfrac{\ell}{n})\big) \,\underset{n\to\infty}{\rightarrow} \, r(s,\tfrac{k}{n}),
\end{equation*}
with the exception that the sum starts at $\ell=0$ for $r_0$. More details on this can be found in Appendix B.
\end{remark}
\begin{remark}[Euclidean case and existence of geodesics]\label{rk:l2geod}
Just as in the continuous case, when $M=\mathbb R^d$, the curvature terms $R_k$'s vanish and we obtain
\begin{equation*}
{x_0}''(s)=0,\quad q_k''(s)=0, \,\, k=0,\hdots,n-1, \quad \forall s\in[0,1],
\end{equation*}
i.e. the geodesics are composed of straight lines in the SRV coordinates. We can again avoid the problem of the $q_k(s)$'s going through zero by allowing two consecutive components $x_k$ and $x_{k+1}$ to be equal, and setting $q_k=0$ when that happens. In that case, we get a complete finite-dimensional manifold, which is by the Hopf-Rinow theorem geodesically complete, i.e. any two curves can be linked by a minimizing geodesic. Indeed, since the SRV coordinates of geodesics are straight lines, a sequence in $(\mathbb R^d)^{n+1}$ converges if and only if its SRV coordinates $x_0, q_0, \hdots, q_{n-1}$ converge in $\mathbb R^d$ (the sequence subscript is omited), and so the completeness of $(\mathbb R^d)^{n+1}$ follows from that of $\mathbb R^d$. The question of whether this property still holds in the non flat case is postponed to future work.
\end{remark}
Using equations \eqref{disgeodeq} we can now build the exponential map, that is, an algorithm allowing us to approximate the geodesic of $M^{n+1}$ starting from a point $(x^0_0, \hdots, x^0_n)\in M^{n+1}$ at speed $(u_0, \hdots, u_n)$ with $u_k \in T_{x_k}M$ for all $k=0,\hdots, n$. In other words, we are looking for a path $[0,1] \ni s\mapsto \alpha(s) = (x_0(s), \hdots, x_n(s))$ such that $x_k(0)=x^0_k$ and ${x_k}'(0) = u_k$ for all $k$, and that verifies the geodesic equations \eqref{disgeodeq}. Assume that we know at time $s\in[0,1]$ the values of $x_k(s)$ and ${x_k}'(s)$ for all $k=0,\hdots,n$. Then we propagate using
\begin{align*}
x_k(s+\epsilon) &= \log_{x_k(s)} \epsilon {x_k}'(s), \\
{x_k}'(s+\epsilon) &= \left({x_k}'(s) + \epsilon \nabla_s{x_k}'(s)\right)^{s,s+\epsilon}.
\end{align*}
In the following proposition, we see how we can compute the acceleration $\nabla_s{x_k}'$ for each $k$. 
\begin{proposition}[Discrete exponential map]
\label{prop:expmap}
Let $[0,1] \ni s\mapsto \alpha(s) = (x_0(s), \hdots, x_n(s))$ be a geodesic path in $M^{n+1}$. For all $s\in [0,1]$, the coordinates of its acceleration $\nabla_s\alpha'(s)$ can be iteratively computed in the following way
\begin{equation*}
\begin{aligned}
&\nabla_s{x_0}' =-  \frac{1}{n} \Big( R_0 + f_0^{(-)}(R_1)\\
&\hspace{7em} + \hdots + f_0^{(-)}\circ \cdots \circ f_{n-2}^{(-)} (R_{n-1})\Big),\\
&\nabla_s{x_{k+1}}'^\parallel = \nabla_sf_k({x_k}') + f_k(\nabla_s{x_k}') + \frac{1}{n} \nabla_sg_k(\nabla_sq_k) \\
&\hspace{7em}+ \frac{1}{n} g_k(\nabla_s^2q_k) +\mathcal R( \tau_k,Y_k)({x_{k+1}}'^\parallel),
\end{aligned}
\end{equation*}
for $k=0,\hdots,n-1$, where the $R_k$'s are defined as in Proposition \ref{prop:geodeq}, the symbol $\cdot^\parallel$ denotes the parallel transport from $x_{k+1}(s)$ back to $x_k(s)$ along the geodesic linking them, the maps $\nabla_sf_k$ and $\nabla_sg_k$ are given by Lemma \ref{lem:ns}, $Y_k$ is given by Equation \eqref{yk} and
\begin{equation*}
\begin{aligned}
& \nabla_s\tau_k = (D_\tau \alpha')_k, \quad \nabla_sv_k = \frac{1}{|\tau_k|} \left( \nabla_s\tau_k - {\nabla_s\tau_k}^T \right),\\
&\nabla_sq_k = \sqrt{\frac{n}{|\tau_k|}} \left( \nabla_s\tau_k - \frac{1}{2}{\nabla_s\tau_k}^T \right),\\
&\nabla_s^2q_k = - \frac{1}{n} \,\,g_k^{(-)}\Big( R_{k+1} + f_{k+1}^{(-)}(R_{k+2}) \\
&\hspace{7em}+ \hdots + f_{k+1}^{(-)} \circ \cdots \circ f_{n-2}^{(-)}(R_{n-1})\Big).
\end{aligned}
\end{equation*}
\end{proposition}
\begin{proof}The proof is given in Appendix B.\end{proof}
The equations of Proposition \ref{prop:expmap} allow us to iteratively construct a geodesic $s\mapsto \alpha(s)$ in $M^{n+1}$ for metric $G^n$ from the knowledge of its initial conditions $\alpha(0)$ and $\alpha'(0)$. The next step is to construct geodesics under boundary constraints, i.e. to find the shortest path between two elements $\alpha_0$ and $\alpha_1$ of $M^{n+1}$.

\subsubsection{Jacobi fields and geodesic shooting}

As explained in Section \ref{subsubsec:geod} for the continuous model, we solve the boundary value problem using geodesic shooting. To do so, recall that we need to characterize the Jacobi fields for the metric $G^n$, since these play a role in the correction of the shooting direction at each iteration of the algorithm. Recall also that for any geodesic $s\mapsto \alpha(s)$ in $M^{n+1}$ and Jacobi field $s\mapsto J(s)$ along $\alpha$, there exists a family of geodesics $(-\delta, \delta) \ni a \mapsto \alpha(a,\cdot)$ such that $\alpha(0,s)=\alpha(s)$ for all $s$ and 
\begin{equation*}
J(s) = \left.\frac{\partial}{\partial a}\right|_{a=0} \alpha(a,s).
\end{equation*}
\begin{proposition}[Discrete Jacobi fields]
\label{prop:jacobi}
Let $[0,1] \ni s\mapsto \alpha(s) = (x_0(s), \hdots, x_n(s))$ be a geodesic path in $M^{n+1}$, $[0,1]\ni s\mapsto J(s)=(J_0(s),\hdots, J_n(s))$ a Jacobi field along $\alpha$, and $(-\delta, \delta) \ni a \mapsto \alpha(a,\cdot)$ a corresponding family of geodesics, in the sense just described. Then $J$ verifies the second order linear ODE
\begin{align*}
&\nabla_s^2J_0 = \mathcal R({x_0}', J_0){x_0}' - \frac{1}{n}\Big( \nabla_aR_0 + f_0^{(-)}(\nabla_aR_1) + \hdots\\
& +  f_0^{(-)}\circ\cdots\circ f_{n-2}^{(-)}(\nabla_aR_{n-1})\Big)\\
&\hspace{0em} -\frac{1}{n} \sum_{k=0}^{n-2} \sum_{\ell = 0}^k f_0^{(-)}\circ \cdots \circ\nabla_a\big(f_\ell^{(-)}\big)\circ\cdots \circ f_{k}^{(-)}(R_{k+1}),\\
&{\nabla_s^2J_{k+1}}^\parallel = f_k(\nabla_s^2J_k) + 2 \nabla_sf_k(\nabla_sJ_k) + \nabla_s^2f_k(J_k) \\
&+ \frac{1}{n}  g_k(\nabla_s^2\nabla_aq_k)+ \frac{2}{n} \nabla_sg_k(\nabla_s\nabla_aq_k)+ \frac{1}{n}\nabla_s^2g_k(\nabla_aq_k)\\
&+ 2\mathcal R(\tau_k, Y_k)({\nabla_sJ_{k+1}}^\parallel) +\mathcal R(\nabla_s\tau_k,Y_k)({J_{k+1}}^\parallel) \\
&+ \mathcal R(\tau_k,\nabla_sY_k)({J_{k+1}}^\parallel) +\mathcal R(\tau_k,Y_k)\Big(\mathcal R(Y_k,\tau_k)({J_{k+1}}^\parallel)\Big), 
\end{align*}
for all $0\leq k \leq n-1$, where $R_k := \mathcal R(q_k,\nabla_sq_k){x_k}'$ and the various covariant derivatives according to $a$ can be expressed as functions of $J$ and $\nabla_sJ$, 
\begin{align*}
& \nabla_aR_k = \mathcal R\big(\nabla_aq_k,\nabla_sq_k\big){x_k}' + \mathcal R\big(q_k,\nabla_s\nabla_aq_k \\
&+ \mathcal R(J,{x_k}')q_k \big){x_k}' + \mathcal R\big(q_k,\nabla_sq_k)\nabla_sJ_k,\\
&\nabla_aq_k = \sqrt{\frac{n}{|\tau_k|}} \left( \nabla_a\tau_k - \frac{1}{2}{\nabla_a\tau_k}^T \right),\quad \nabla_a\tau_k = (D_\tau J)_k, \\
&\nabla_av_k = \frac{1}{|\tau_k|} \left( \nabla_a\tau_k - {\nabla_a\tau_k}^T \right),\\
&\nabla_s\nabla_aq_k = n \, {g_k}^{-1}\big( (\nabla_sJ_{k+1})^\parallel + \mathcal R(Y_k,\tau_k)({J_{k+1}}^\parallel) \\
&- \nabla_sf_k(J_k) - f_k(\nabla_sJ_k) \big)  + n \, \nabla_s\big({g_k}^{-1}\big)\big({J_{k+1}}^\parallel - f_k(J_k)\big), \\
&\nabla_s^2\nabla_aq_k = - \frac{1}{n} \!\sum_{\ell=k+1}^{n-1} g_k^{(-)} \circ f_{k+1}^{(-)}\circ\cdots\circ f_{\ell-1}^{(-)}(\nabla_aR_\ell)\\
&+\mathcal R(\nabla_s{x_k}',J_k)q_k + \mathcal R({x_k}',\nabla_sJ_k)q_k + 2\mathcal R({x_k}',J_k)\nabla_sq_k \\
&- \frac{1}{n}\! \sum_{\ell=k+1}^{n-1} \!\sum_{j = k}^{\ell-1} g_k^{(-)}\circ \cdots \circ\nabla_a\big(f_j^{(-)}\big)\circ\cdots \circ f_{\ell-1}^{(-)}(R_\ell),\\
&\nabla_sY_k = (\nabla_s{x_k}')^T + b_k(\nabla_s{x_k}')^N+ \partial_sb_k ({x_k}')^N\\
& + (1-b_k)\big(\langle {x_k}',\nabla_sv_k\rangle v_k \langle {x_k}',v_k\rangle \nabla_sv_k\big) +\tfrac{1}{2}(\nabla_s^2\tau_k)^T \\
&+ K\frac{1-a_k}{|\tau_k|^2}(\nabla_s^2\tau_k)^N + \partial_s\Big(K\frac{1-a_k}{|\tau_k|^2}\Big)(\nabla_s\tau_k)^N \\
&+\Big(\tfrac{1}{2}-K\frac{1-a_k}{|\tau_k|^2}\Big)(\langle \nabla_s\tau_k,\nabla_sv_k\rangle v_k + \langle \nabla_s\tau_k,v_k\rangle \nabla_sv_k),
\end{align*}
with the notation conventions $f_{k+1}^{(-)} \circ \hdots \circ f_{k-1}^{(-)}:=Id$, $\sum_{\ell=n}^{n-1}:=0$ and with the maps
\begin{align*}
&\nabla_a\big(f_k^{(-)}\big)(w)=(\nabla_af_k)^{(-)}(w) + f_k\Big(\mathcal R(Z_k,\tau_k)({w_{k+1}}^\parallel)\Big),\\
&\nabla_a\big(g_k^{(-)}\big)(w)=(\nabla_ag_k)^{(-)}(w) + g_k\Big(\mathcal R(Z_k,\tau_k)({w_{k+1}}^\parallel)\Big),\\
&\nabla_s\big({g_k}^{-1}\big)(w) = \partial_s{|q_k|}^{-1} |q_k| {g_k}^{-1}(w) +\! |q_k|^{-1}\partial_s(b_k^{-1})\, {w}^N\\
&\hspace{1em}+ |q_k|^{-1} \, \big(1/2 - b_k^{-1}\big)\big(\langle w,\nabla_sv_k\rangle v_k + \langle w,v_k\rangle\nabla_sv_k\big),
\end{align*}
and
\begin{equation*}
Z_k = {J_k}^T + b_k {J_k}^N + \tfrac{1}{2}{\nabla_a\tau_k}^T + K\frac{1-a_k}{|\tau_k|^2}{\nabla_a\tau_k}^N.
\end{equation*}
\end{proposition}
\begin{proof}The proof is given in Appendix B.\end{proof}
The equations of Proposition \ref{prop:jacobi} allow us to iteratively compute the Jacobi field $J$ along a geodesic $\alpha$ - and in particular, its end value $J(1)$ - from the knowledge of the initial conditions $\{J_k(0), 0\leq k\leq n\}$ and $\{\nabla_sJ_k(0), 0\leq k\leq n\}$. Indeed, if at time $s\in[0,1]$ we have $J_k(s)$ and $\nabla_sJ_k(s)$ for all $k=0,\hdots,n$, then we can propagate using
\begin{align*}
J_k(s+\epsilon) &= \big( J_k(s) + \epsilon \nabla_sJ_k(s) \big)^{x_k,x_{k+1}},\\
\nabla_sJ_k(s+\epsilon) &= \big( \nabla_sJ_k(s) + \epsilon \nabla_s^2J_k(s) \big)^{x_k,x_{k+1}},
\end{align*}
where $\nabla_s^2J_k(s)$ is deduced from $\nabla_s^2J_{k-1}(s)$ using Proposition \ref{prop:jacobi}. 
We can now apply Algorithm \ref{alg:geodshoot}, where we replace the smooth geodesic equations \eqref{contgeodeq} by the discrete geodesic equations \eqref{disgeodeq} and we solve them using the exponential map described in Proposition \ref{prop:expmap}. Notice that in $M^{n+1}$, the $k^{th}$ component of the $L^2$-logarithm map between two elements $\alpha_0=(x_0^0,\hdots,x_n^0)$ and $\alpha_1=(x_0^1,\hdots,x_n^1)$ is given by $\log_{x^0_k}(x^1_k)$.
\begin{algo}[Discrete geodesic shooting] \label{alg:geodshootdis}
\leavevmode\par \noindent
Input: $\alpha_0=(x^0_0,\hdots,x^0_n), \alpha_1=(x^1_0,\hdots,x^1_n)$.\\
Initialization: $w =\log_{\alpha_0}^{L^2}(\alpha_1)$.\\
Repeat until convergence :
\begin{enumerate}
\item compute the geodesic $s\mapsto \alpha(s)$ starting from $\alpha_0$ at speed $w$ using Proposition \ref{prop:expmap},
\item evaluate the difference $j := \log_{\alpha(1)}^{L^2}(\alpha_1)$ between the target curve $\alpha_1$ and the extremity $\alpha(1)$ of the obtained geodesic,
\item compute the initial derivative $\nabla_sJ(0)$ of the Jacobi field $s\mapsto J(s)$ along $\alpha$ verifying $J(0)=0$ and $J(1)=j$ using Proposition \ref{prop:jacobi},
\item correct the shooting direction $w = w + \nabla_sJ(0)$.
\end{enumerate}
Output : geodesic $\alpha(s)$.
\end{algo}
Recall that the map $\varphi : T_{\alpha(0)}\mathcal M \rightarrow T_{\alpha(1)}\mathcal M$, $\nabla_sJ(0) \mapsto J(1)$ associating to the initial derivative $\nabla_sJ(0)$ of a Jacobi field with initial value $J(0)=0$ its end value $J(1)$, is a linear bijection between two vector spaces which can be obtained using Proposition \ref{prop:jacobi}. Its inverse map can be computed by considering the image of a basis of $T_{c(0)}\mathcal M$.

\subsection{A discrete analog of unparameterized curves}

The final step in building our discrete model is to introduce a discretization of the quotient shape space. There seems to be no natural, intrinsic definition of the shape of a discrete curve, as by definition we are lacking information : we only have access to a finite number $n+1$ of points. Therefore we will make the assumption that we know the equations of the underlying curves, that is, that for each discrete curve $\alpha$, we have access to the shape $\bar \alpha$ of the smooth curve of which $\alpha$ is the discretization. In practice, we can set $\bar \alpha$ to be the shape of an optimal interpolation. The goal, for two elements $\alpha_0,\alpha_1$ of shapes $\overline{\alpha_0}, \overline{\alpha_1}$, is to redistribute the $n+1$ points on $\overline{\alpha_1}$ to minimize the discrete distance to the $n+1$ points $\alpha_0$ on $\overline{\alpha_0}$, and obtain
\begin{equation}
\label{alphaopt}
\alpha_1^{opt} = \text{argmin} \{ d_n(\alpha_0, \alpha)\,|\, \alpha\text{ has shape }\overline{c_1} \},
\end{equation}
where $d_n$ is the geodesic distance associated to the discrete metric $G^n$. We approximate $\alpha_1^{opt}$ using Algorithm \ref{alg:optmatch}, i.e. by iteratively computing the "horizontal part" of the geodesic linking $\alpha_0$ to an iteratively improved discretization of $\overline{\alpha_1}$. Since there is no "discrete shape bundle", we simply define the vertical and horizontal spaces in $\alpha$ as the discrete analogs of the ones of the smooth case
\begin{align*}
&\text{Ver}^n_\alpha := \{ mv :m=(m_k)_k\in \mathbb R^{n+1}, m_0=m_n=0\}, \\
&\text{Hor}^n_\alpha :=\{ h\in T_\alpha M^{n+1}: G^n(h,mv)=0\\
&\hspace{7em} \forall m=(m_k)_k\in \mathbb R^{n+1}, m_0=m_n=0\},
\end{align*}
where $v=(v_k)_k$ is still defined by \eqref{not}. Similarly to the continuous case, we can show the following result.
\begin{proposition}[Discrete horizontal space]
\label{prop:hordis}
Let $\alpha\in M^{n+1}$ and $h\in T_\alpha M^{n+1}$. Then $h\in \text{Hor}^n_\alpha$ if and only if it verifies
\begin{align*}
\big\langle(D_\tau h)_k,v_k\big\rangle-4\frac{|\tau_k|}{|\tau_{k-1}|}&\Big\langle(D_\tau h)_{k-1},b_{k-1}^{-1}{v_k}^\parallel\\
&+(\tfrac{1}{4}-b_{k-1}^{-1})\lambda_{k-1}v_{k-1}\Big\rangle=0.
\end{align*}
with the notation $\lambda_k :=\langle v_{k+1}^\parallel, v_k\rangle$. Any tangent vector $w\in T_\alpha M^{n+1}$ can be uniquely decomposed into a sum $w=w^{ver}+w^{hor}$ where $w^{ver}=mv\in \text{Ver}_\alpha^n$, $w^{hor}=w-mv \in \text{Hor}_\alpha^n$ and the components $(m_k)_k$ verify $m_0=m_1=0$ and the following recurrence relation
\begin{equation}\label{mk}
A_k m_{k+1} + B_k m_k + C_k m_{k-1} = D_k,
\end{equation}
with coefficients
\begin{align*}
A_k &= \lambda_k,\\
B_k &= -1 -4\frac{|\tau_k|}{|\tau_{k-1}|}(b_{k-1}^{-2}+\lambda_{k-1}^2(\tfrac{1}{4}-b_{k-1}^{-2})),\\
C_k &=\frac{|\tau_k|}{|\tau_{k-1}|}\lambda_{k-1},
\end{align*}
\begin{align*}
D_k &=\big\langle (D_\tau w)_k,v_k\big\rangle- 4\frac{|\tau_k|}{|\tau_{k-1}|}\Big(b_{k-1}^{-1}\big\langle (D_\tau w)_{k-1},v_k^\parallel\big\rangle \\
&+ (\tfrac{1}{4}-b_{k-1}^{-1})\lambda_{k-1}\big\langle (D_\tau w)_{k-1},v_{k-1}\big\rangle \Big).
\end{align*}
\end{proposition}
\begin{proof}
Let $h \in T_\alpha\mathcal M$ be a tangent vector. It is horizontal if and only if it is orthogonal to any vertical vector, that is any vector of the form $mv$ with $m=(m_k)_k\in \mathbb R^{n+1}$ such that $m_0=m_n=0$. Recall that by definition
\begin{equation*}
(D_\tau w)_k := ({w_{k+1}}^\parallel - w_k)^T + b_k^{-1}({w_{k+1}}^\parallel - a_kw_k)^N,
\end{equation*}
and so with the notation $\lambda_k:=\langle {v_{k+1}}^\parallel,v_k\rangle$, we get
\begin{align*}
(D_\tau (mv))_k^T &= m_{k+1}({v_{k+1}}^\parallel)^T - m_kv_k \\
&= (m_{k+1}\lambda_k - m_k) v_k,\\
(D_\tau (mv))_k^N &= b_k^{-1}m_{k+1}({v_{k+1}}^\parallel)^N\\
&=b_k^{-1}m_{k+1}({v_{k+1}}^\parallel-\lambda_k v_k).
\end{align*}
The scalar product between $h$ and $mv$ is then
\begin{align*}
G^n_\alpha(h,mv) &= \sum_{k=0}^{n-1} \Big( b_k^{-1} m_{k+1} \big\langle (D_\tau h)_k, {v_{k+1}}^\parallel-\lambda_k v_k \big\rangle \\
+& \tfrac{1}{4} (m_{k+1}\lambda_k - m_k) \big\langle(D_\tau h)_k, v_k \big\rangle \Big) |\tau_k|^{-1}\\
&= \sum_{k=0}^{n-1} \frac{m_{k+1}}{|\tau_k|} \Big( b_k^{-1} \big\langle (D_\tau h)_k, {v_{k+1}}^\parallel-\lambda_k v_k \big\rangle \\
+& \tfrac{1}{4} \lambda_k \big\langle(D_\tau h)_k, v_k \big\rangle \Big) \!-\!\tfrac{1}{4} \sum_{k=0}^{n-1}\frac{m_{k}}{|\tau_k|}\big\langle(D_\tau h)_k, v_k \big\rangle.
\end{align*}
Changing the indices in the first sum and taking into account that $m_0=m_n=0$, we obtain
\begin{align*}
&\sum_{k=1}^{n-1}m_k\Big(|\tau_{k-1}|^{-1}\big\langle (D_\tau h)_{k-1}, b_{k-1}^{-1}{v_k}^\parallel \\
&+(\tfrac{1}{4}-b_{k-1}^{-1})\lambda_{k-1}v_{k-1}\big\rangle - \tfrac{1}{4}|\tau_k|^{-1}\big\langle (D_\tau h)_k,v_k \big\rangle \Big)=0.
\end{align*}
Since this is true for all such $m$ the summand is equal to zero for all $k$ and we get the desired equation. The decomposition of a tangent vector $w$ into a vertical part $mv$ and a horizontal part $w-mv$ with $m=(m_k)_k\in \mathbb R^{n+1}$ such that $m_0=m_n=0$, is then simply characterized by the fact that $w-mv$ verifies this equation.
\end{proof}

In the discrete case, soving the ODE \eqref{mtt} to find the horizontal part of a vector simply boils down to solving the recurrence relation \eqref{mk}, allowing us to compute the coefficients of the PDE \eqref{phicont}. Now we present an algorithm to solve a discrete version of this PDE and compute the discrete analog of the horizontal part of a (discrete) path of discrete curves.
\begin{algo}[Horizontal part of a path]\label{alg:hordis}
\leavevmode\par \noindent
Input : $\alpha(s)=(x_0(s),\hdots,x_n(s))$, $s=0,\frac{1}{m},\hdots,1$.\\
Initialization : 
for $k=0,\hdots,n$,
\begin{align*}
\varphi(0)(\tfrac{k}{n}) = \tfrac{k}{n},\quad x_k^{hor}(0)=x_k(0).
\end{align*}
For $j=0,\hdots,m-1$,
\begin{enumerate}
\item set $s=j/m$, $m_0(s)=m_{n+1}(s)=0$ and solve
\begin{equation*}
\left(\begin{matrix} B_1 & A_1 & \cdots & 0\\
C_2 & \ddots & \ddots & \vdots\\
\vdots & \ddots & \ddots & A_{n-2}\\
0 & \cdots & C_{n-1} & B_{n-1}\end{matrix} \right)\left(\begin{matrix}
m_1\\ \vdots\\ \vdots\\ m_{n-1}\end{matrix}\right) = \left(\begin{matrix}
D_1\\ \vdots\\ \vdots\\ D_{n-1}\end{matrix}\right),
\end{equation*}
\item for $k=0,\hdots,n$,
\begin{align*}
&\Delta\varphi(s)(\tfrac{k}{n}) = \left\{ \begin{aligned} 
&n\big(\varphi(s)(\tfrac{k+1}{n})-\varphi(s)(\tfrac{k}{n})\big)\quad \text{if }m_k\geq 0,\\ 
&n\big(\varphi(s)(\tfrac{k}{n})-\varphi(s)(\tfrac{k-1}{n})\big) \quad \text{if }m_k< 0,
\end{aligned}\right.\\
&\varphi_s(s)(\tfrac{k}{n}) = \frac{m_k(s)}{|n\tau_k(s)|}\Delta \varphi(s)(k),\\
&\varphi(s+\tfrac{1}{m})(\tfrac{k}{n}) = \varphi(s)(\tfrac{k}{n}) + \tfrac{1}{m} \varphi_s(s)(\tfrac{k}{n}),
\end{align*}
\item interpolate between the $\{x_k(s+\tfrac{1}{m}),k=0,\hdots,n\}$ to obtain $N+1$ values $\{y_\ell(s+\tfrac{1}{m}),\ell=0,\hdots,N\}$ and interpolate between the $\{\varphi(s+\tfrac{1}{m})(\tfrac{k}{n}),k=0,\hdots,n\}$ to obtain $N+1$ values $\{\psi(s+\tfrac{1}{m})(\tfrac{\ell}{n}),\ell=0,\hdots,N\}$,
\item for $k=0,\hdots,n$,
\begin{align*}
&\text{find }\ell \text{ s.t. } \tfrac{k}{n}\leq \psi(s+\tfrac{1}{m})(\tfrac{\ell}{N})<\tfrac{k+1}{n},\\
&\text{set } x_k^{hor}(s+\tfrac{1}{m}) = y_\ell(s+\tfrac{1}{m}).
\end{align*}
\end{enumerate}
Output: $\alpha^{hor}(s)=(x_0^{hor}(s),..,x_n^{hor}(s)), s=0,\frac{1}{m},\hdots,1$.
\end{algo}
Step 1 computes the coefficients of the PDE. The $A_k,B_k,C_k,D_k$ are computed using the definitions of Proposition \ref{prop:hordis} for $w=\alpha'(s)$, i.e. $(D_\tau \alpha'(s))_k=\nabla_s\tau_k(s)$. Step 2 solves the PDE. The increment $\Delta\varphi$ is a discretization of the $t$-derivative, and so it is crucial to make it depend on the sign of $m$ in order to follow the characteristic curves. Steps 3 and 4 simply inverse the reparameterization $\varphi$ obtained after step 2 in order to deduce the horizontal part of $\alpha$. It is important to note that for $s=1$, interpolation between the points $\{x_k(s),k=0,\hdots,n\}$ should be achieved so as to remain on the initial shape $\overline{\alpha(1)}$, obtained e.g. by spline interpolation of the points of $\alpha(1)$. Finally, we can perform optimal matching in the discrete case.
\begin{algo}[Discrete optimal matching]\label{alg:optmatchdis}
\leavevmode\par \noindent
Input: $\alpha_0=(x_0^0,\hdots,x_n^0), \alpha_1=(x_0^1,\hdots,x_n^1)$.\\
Initialization : $\tilde \alpha_1 = \alpha_1$.\\
Repeat until convergence :
\begin{enumerate}
\item compute geodesic $s\mapsto \alpha(s)=(x_0(s),\hdots,x_n(s))$ between $\alpha_0$ and $\tilde \alpha_1$ using Algorithm \ref{alg:geodshootdis},
\item compute $s\mapsto \alpha^{hor}(s)=(x_0^{hor}(s),\hdots,x_n^{hor}(s))$ the horizontal part of $\alpha$ using Algorithm \ref{alg:hordis},
\item set $\tilde \alpha_1 = \alpha^{hor}(1)$.
\end{enumerate}
Output : $\alpha_1^{opt} = \tilde \alpha_1$.
\end{algo}

\section{Simulations}				
\label{sec:simu}

We test the optimal matching (OM) algorithm in several settings : for curves in the hyperbolic half-plane $\mathbb H^2$, for curves in $\mathbb R^2$ and $\mathbb R^3$, and for curves on the sphere $\mathbb S^2$. Regarding the geometry of $\mathbb H^2$ and the useful algorithms such as the exponential map and the logarithm map, we refer the reader to \cite{moi17}. Concerning the geometry of $\mathbb S^2$, we have used the same formulas as those given in appendix in \cite{zhang16}.

We start by comparing Algorithm \ref{alg:optmatchdis} to the popular dynamic programming (DP) method, presented e.g. in \cite{srivastava16}. This alternative approach to optimal matching between two curves $c_1$ and $c_2$ considers a discrete grid of $[0,1]\times[0,1]$ and tries to find the optimal non-decreasing path $\varphi(t)$ in that grid that puts into correspondence $c_1(t)$ and $c_2(\varphi(t))$. To do so, at each point $(t_1,t_2)$ of the grid, it tests all the possible matchings between the pieces of curves $c_1([0,t_1])$ and $c_2([0,t_2])$, and keeps only the one that gives the shortest geodesic. This boils down to testing all the non-decreasing paths that lead to this point. Since this is computationally costly, the algorithm only tests the paths going through the points located at the bottom-left in a square of a certain size, as shown in Figure \ref{fig:dp}. Even though the computations are additive, this method requires to compute a large number of geodesics. Results of comparison of Algorithm \ref{alg:optmatchdis} to this method are shown in Figure \ref{fig:dynamicprog} for a pair of curves in $\mathbb R^2$ and three pairs of curves in $\mathbb R^3$. The DP method is carried out for a square of side $s=7$. The first row displays the geodesics between the initial parameterized curves (which are all arc-length parameterized except for the second example) and the second and third rows the horizontal geodesics obtained with the optimal matching and dynamic programing algorithms respectively. We can see that the two methods give very similar results : in the first example, at each extremity, a portion of the circle is matched to almost a single point of the segment, in order to maximize the portions in correspondence where the speeds are positively colinear. In the second example, we consider two curves in $\mathbb R^3$ that are identical modulo translation and parameterization. This difference in the parameterization induces an artificial deformation of the geodesic (first row), which is "straightened out" in the horizontal geodesics given by both methods. The last two examples illustrate the fact that arc-length parameterization does not always yield a relevant matching : in the third example, it seems more natural to put into correspondence the portions of the curves that are "before the turn", and to match together the ones that are "after the turn", as given by both the OM algorithm and DP. The fourth row shows the corresponding optimal matchings, i.e. the optimal reparameterizations $\varphi$ such that $c_1(t)$ is matched to $c_2(\varphi(t))$, in red (OM) and black (DP), and we can check in the fifth row the horizontality of the solutions by looking at the ratio, in norm, of the vertical part of the speed vector of the geodesic divided by its horizontal part. We can see that this ratio is largely reduced from the initial geodesic (dashed black line) to the horizontal ones (full red and black lines). Finally, the lengths of the geodesics of the first three lines of Figure \ref{fig:dynamicprog} are given in Table \ref{fig:table} in the corresponding order, showing that the horizontal geodesics are indeed shorter. The lengths on the first row give the distance between the initial \emph{parameterized} curves, while the lengths on the second and third row yield a distance between the \emph{shapes}. 

To summarize, it seems that both methods give very similar results when tested on the same metric : they both tend to put into correspondence the parts of the curves that have same shape and orientation. However the dynamic programing approach requires the computation of a large number of geodesics between pieces of curves, whereas in the examples shown here (Figures \ref{fig:H2vsR2} to \ref{fig:plane}), the number of iterations required range from 4 to 12, resulting in the same amount of geodesic computations. It should also be noted that it is usually the first iteration that gives the biggest "jump" on the fiber, i.e. that results in the most important reparameterization of the target curve. It could therefore also be used to get a good initialization for some approximate faster method. 

\begin{figure}
\centering
\includegraphics[width=0.68\textwidth]{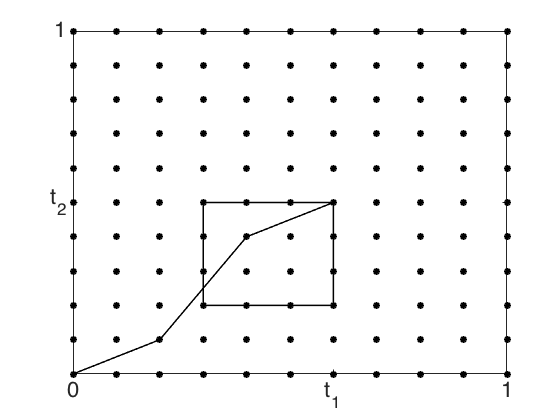}
\caption{Illustration of the dynamic programming algorithm.}
\label{fig:dp}
\end{figure}
\begin{figure}
\centering
\includegraphics[width=0.58\textwidth]{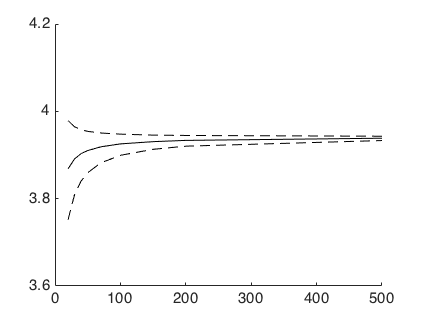}
\caption{Stability of the norm of the speed of a geodesic obtained by geodesic shooting as the number of points used to compute it increases.}
\label{fig:geodshoot}
\end{figure}

\begin{figure*}
\centering
\subfloat{\includegraphics[width=0.25\textwidth]{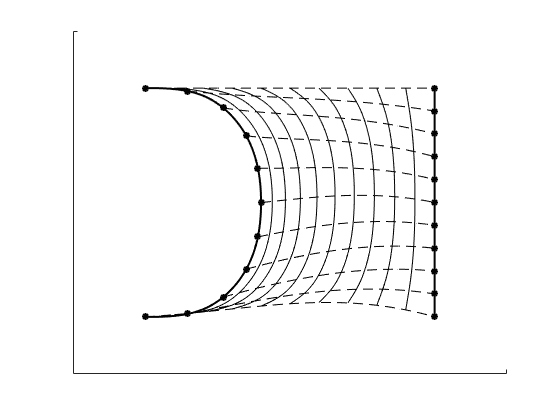}}\hspace*{-1em}
\subfloat{\includegraphics[width=0.3\textwidth]{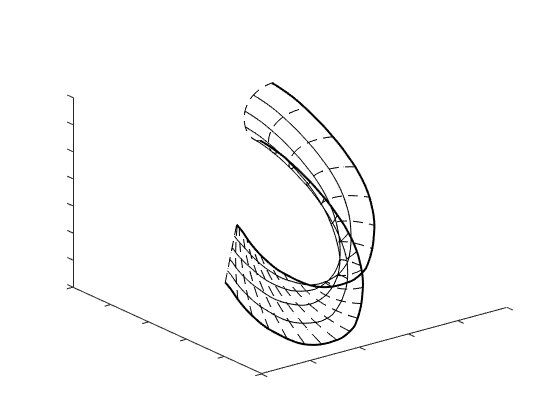}}\hspace*{-3em}
\subfloat{\includegraphics[width=0.3\textwidth]{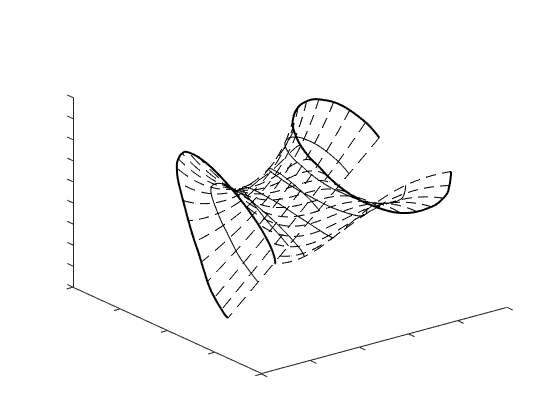}}\hspace*{-3em}
\subfloat{\includegraphics[width=0.3\textwidth]{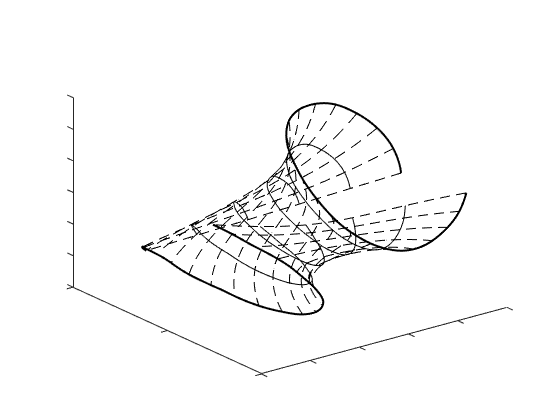}}\vspace*{-1em}\\
\subfloat{\includegraphics[width=0.25\textwidth]{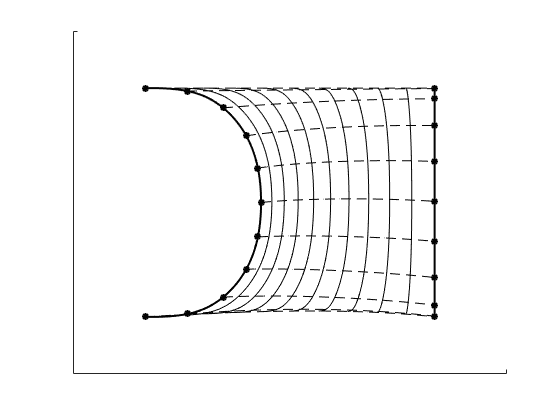}}\hspace*{-1em}
\subfloat{\includegraphics[width=0.3\textwidth]{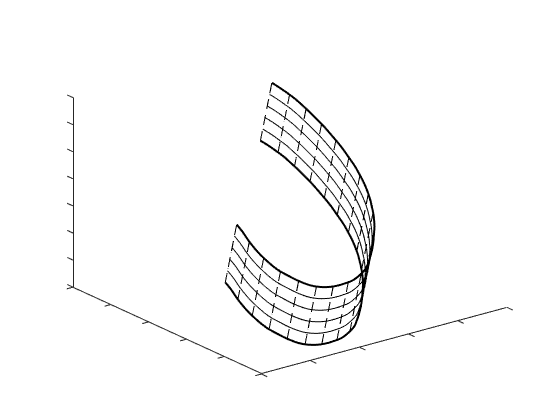}}\hspace*{-3em}
\subfloat{\includegraphics[width=0.3\textwidth]{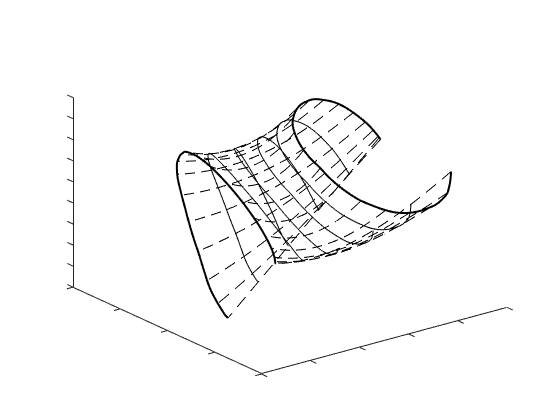}}\hspace*{-3em}
\subfloat{\includegraphics[width=0.3\textwidth]{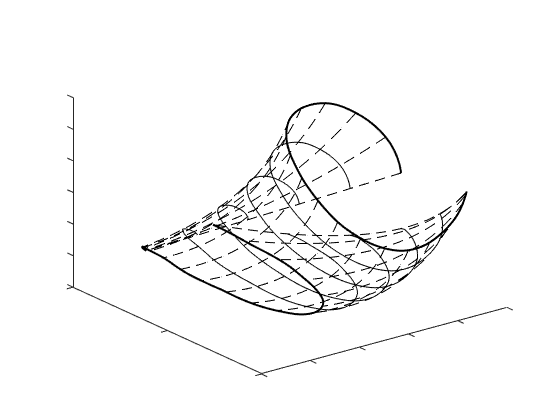}}\vspace*{-1em}\\
\subfloat{\includegraphics[width=0.25\textwidth]{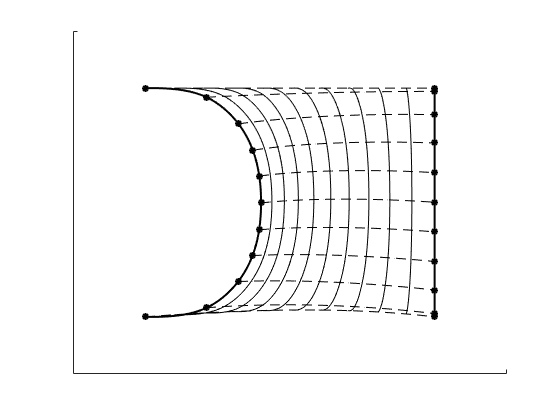}}\hspace*{-1em}
\subfloat{\includegraphics[width=0.3\textwidth]{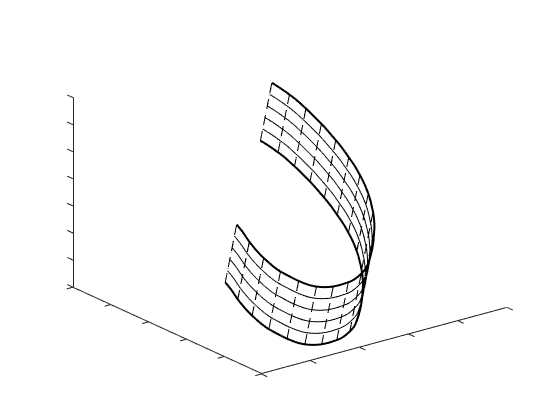}}\hspace*{-3em}
\subfloat{\includegraphics[width=0.3\textwidth]{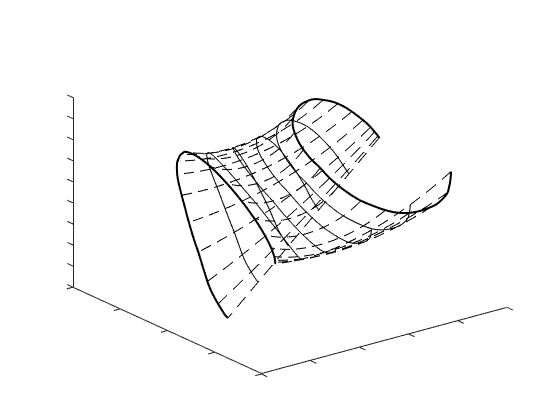}}\hspace*{-3em}
\subfloat{\includegraphics[width=0.3\textwidth]{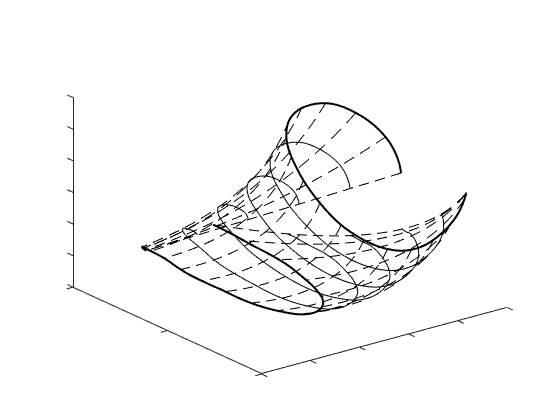}}\vspace*{-1em}\\
\subfloat{\includegraphics[width=0.25\textwidth]{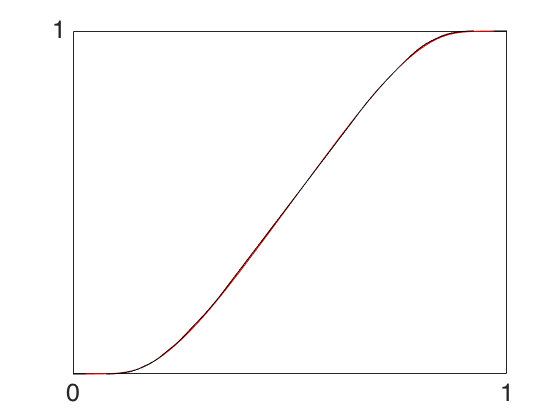}}\hspace*{-0.5em}
\subfloat{\includegraphics[width=0.25\textwidth]{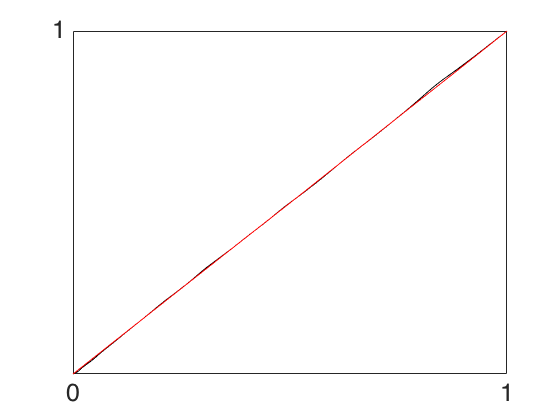}}\hspace*{-0.5em}
\subfloat{\includegraphics[width=0.25\textwidth]{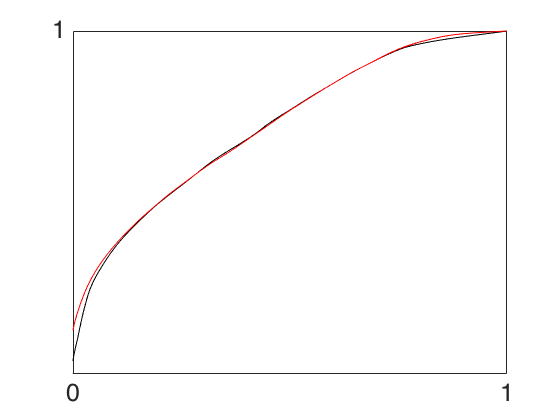}}\hspace*{-0.5em}
\subfloat{\includegraphics[width=0.25\textwidth]{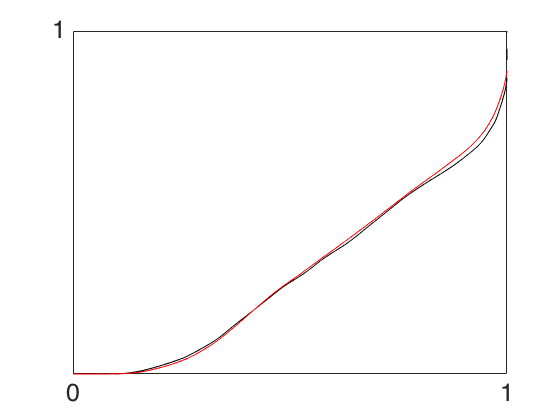}}\vspace*{-0em}\\
\subfloat{\includegraphics[width=0.25\textwidth]{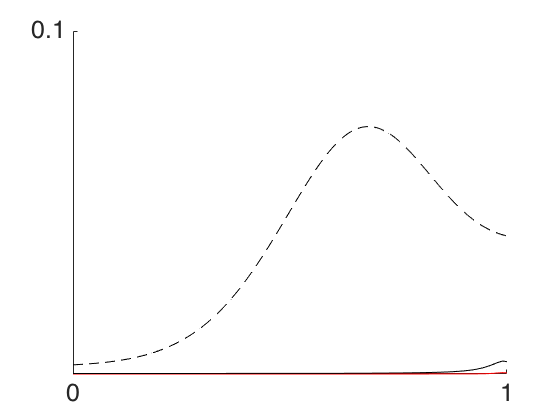}}\hspace*{-0.5em}
\subfloat{\includegraphics[width=0.25\textwidth]{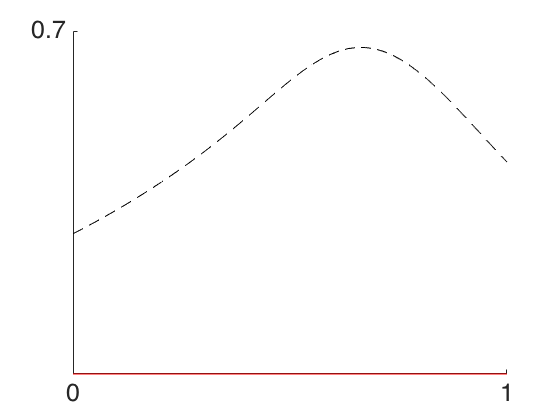}}\hspace*{-0.5em}
\subfloat{\includegraphics[width=0.25\textwidth]{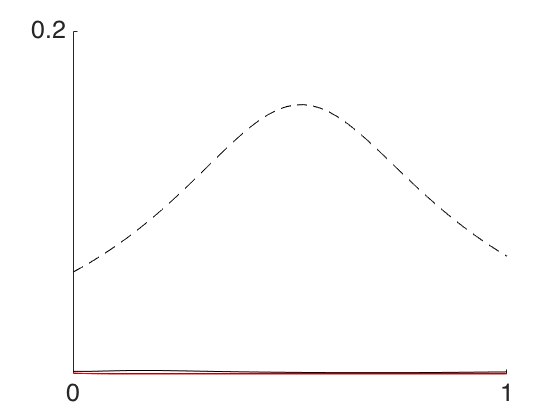}}\hspace*{-0.5em}
\subfloat{\includegraphics[width=0.25\textwidth]{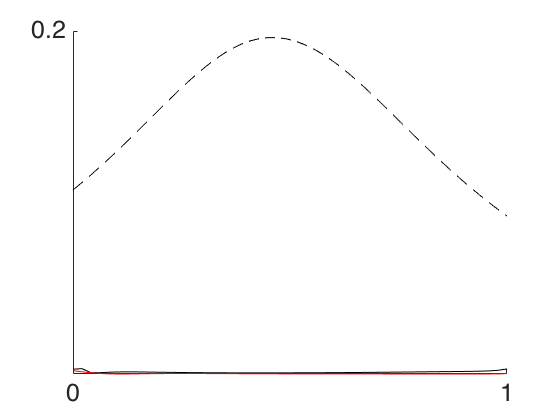}}
\caption{Optimal matching between curves in $\mathbb R^2$ or $\mathbb R^3$ in the SRV framework achieved using Algorithm \ref{alg:optmatchdis} compared to dynamic programming. The first line shows the geodesics between the initial parameterizations, the second line the horizontal geodesics obtained with Algorithm \ref{alg:optmatchdis}, the third line the horizontal geodesics obtained with dynamic programming, the fourth line the optimal matchings for our method (in red) and dynamic programming (in black). The fifth line shows the ratio vertical/horizontal, in norm, of the speed of the initial geodesic (dashed line), the horizontal geodesic obtained with Algorithm \ref{alg:optmatchdis} (red) and the horizontal geodesic obtained with dynamic programming (black).}
\label{fig:dynamicprog}
\end{figure*}

\begin{table*}
\caption{Length of the geodesics shown in the first three lines of Figure \ref{fig:dynamicprog}, displayed in the same order.}
\label{fig:table}
\begin{tabular}{|c|c|c|c|c|c|}
\hline
4.400 & 4.349 & 10.785 &10.478 \\
\hline 
4.300 & 3.000 & 9.680 & 9.305\\
\hline
4.301 & 3.000 & 9.691 & 9.307\\
\hline
\end{tabular}
\end{table*}

We then consider examples where the base manifold has negative curvature. Figure \ref{fig:H2vsR2} shows horizontal geodesics obtained using the OM algorithm as well as the corresponding optimal matchings in $[0,1]\times[0,1]$, for plane curves (first row) and the same curves taking their values in the hyperbolic half-plane (second row). We can see that the geometry of the base manifold significantly influences the optimal matching between two given curves. To evaluate the performance of the geodesic shooting algorithm used to perform OM, we display in Figure \ref{fig:geodshoot} the evolution of the norm of the speed of the geodesic obtained for the two curves on the bottom-right corner of Figure \ref{fig:H2vsR2} (the vertical and horizontal segments of $\mathbb H^2$) as the number of points used to compute the geodesic increases (from 20 to 500). The evolutions of the maximum and minimum values $\max_{s\in[0,1]}|c_s(s)|$, $\min_{s\in[0,1]}|c_s(s)|$ are shown in dashed lines, and that of the mean value with a full line. We can see that the more refined the discretization, the closer we get to a geodesic. It is to be noted that these two segments are the same as those considered in \cite{su17} to illustrate an extension of the SRV framework to curves in Lie groups. Since this metric structure coincides with the one studied here in the planar case, it is not surprising to find similar results for the plane curves. However, the results in $\mathbb H^2$ are quite different. 

\begin{figure*}
\centering
\subfloat{\includegraphics[width=0.25\textwidth]{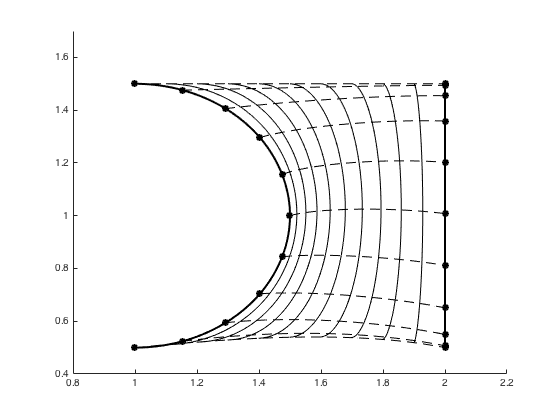}}\hspace*{-1em}
\subfloat{\includegraphics[width=0.25\textwidth]{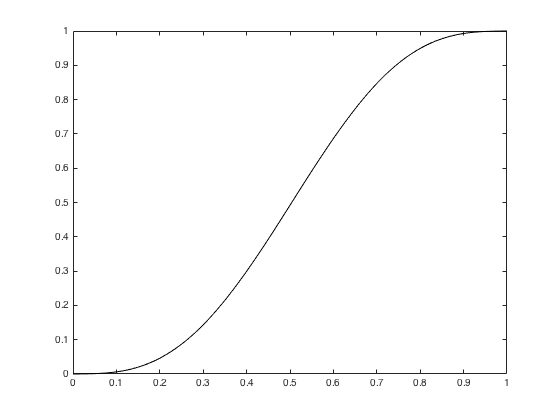}}\hspace*{-1em}
\subfloat{\includegraphics[width=0.25\textwidth]{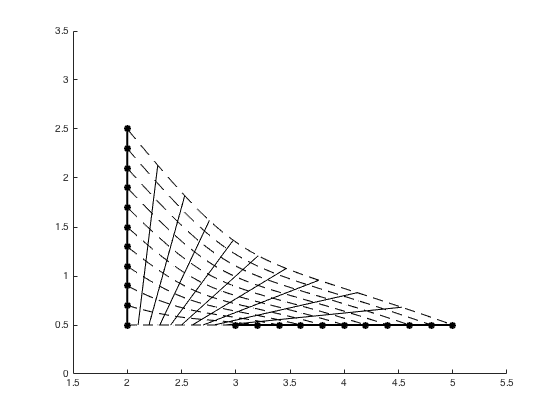}}\hspace*{-1em}
\subfloat{\includegraphics[width=0.25\textwidth]{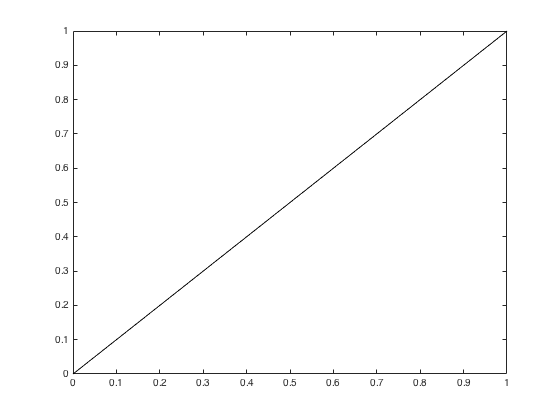}}\\
\subfloat{\includegraphics[width=0.25\textwidth]{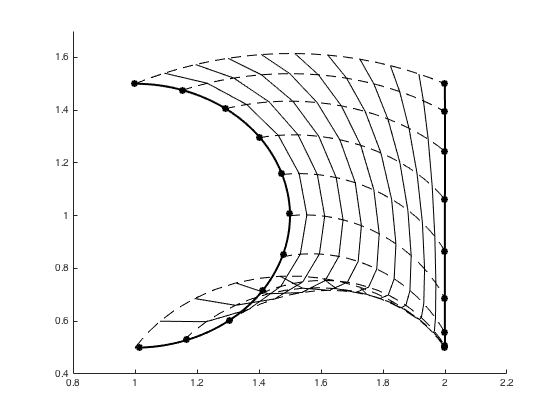}}\hspace*{-1em}
\subfloat{\includegraphics[width=0.25\textwidth]{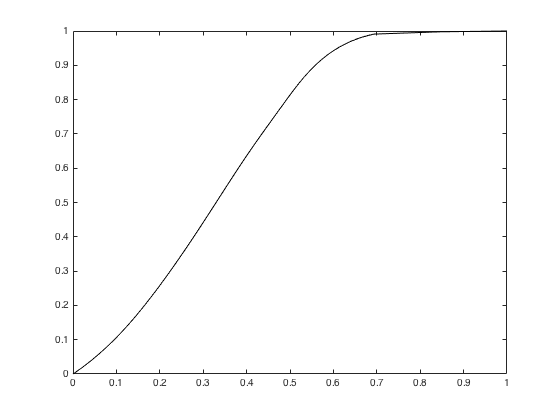}}\hspace*{-1em}
\subfloat{\includegraphics[width=0.25\textwidth]{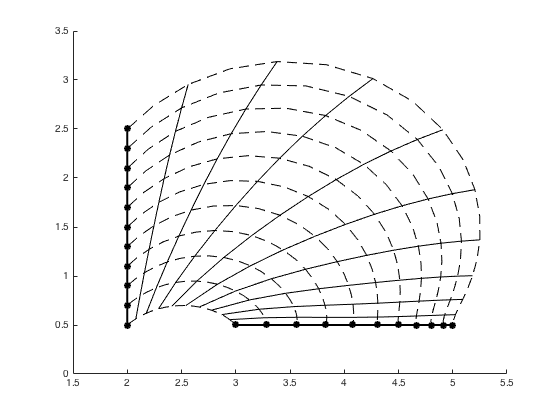}}\hspace*{-1em}
\subfloat{\includegraphics[width=0.25\textwidth]{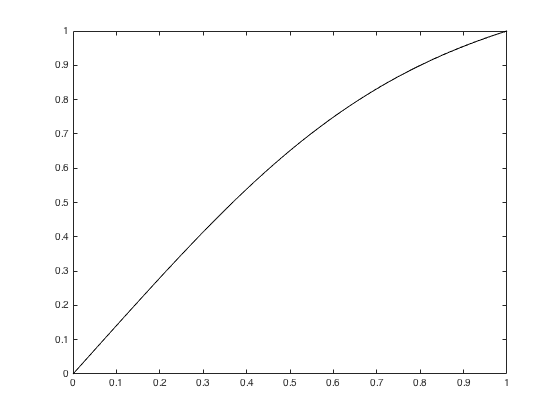}}
\caption{Horizontal geodesics and the corresponding optimal matching between curves in $\mathbb R^2$ (first line) and $\mathbb H^2$ (second line).}
\label{fig:H2vsR2}
\end{figure*}

\begin{figure*}
\centering
\subfloat{\boxed{\includegraphics[width=0.4\textwidth]{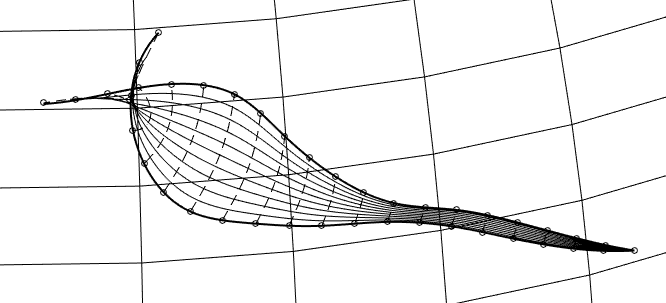}}}\hspace*{1em}
\subfloat{\boxed{\includegraphics[width=0.4\textwidth]{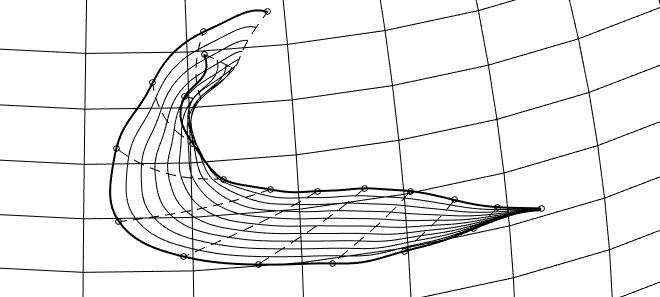}}}\\
\subfloat{\boxed{\includegraphics[width=0.4\textwidth]{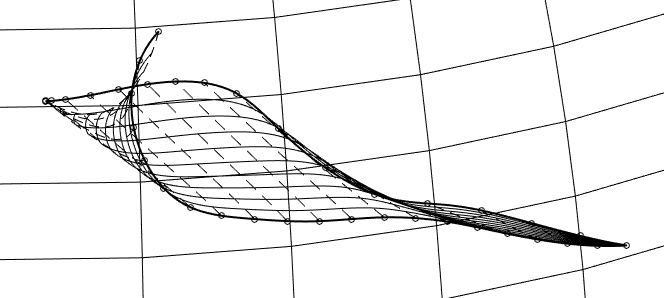}}}\hspace*{1em}
\subfloat{\boxed{\includegraphics[width=0.4\textwidth]{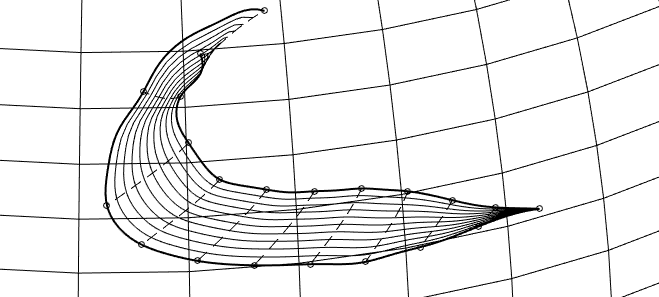}}}
\caption{Geodesics between hurricane tracks in $\mathbb S^2$ before (top) and after (bottom) optimal matching.}
\label{fig:hurricane}
\end{figure*}

\begin{figure*}
\centering
\subfloat{\includegraphics[width=0.3\textwidth]{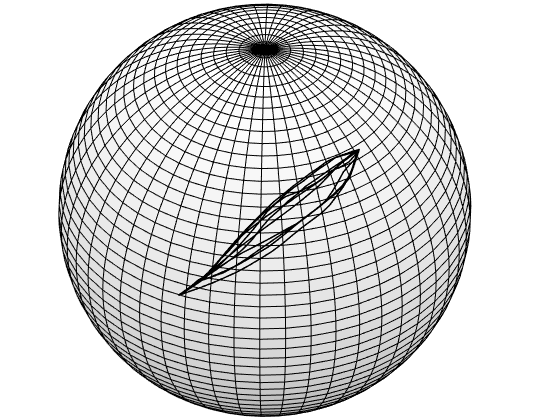}}\hspace*{1em}
\subfloat{\includegraphics[width=0.3\textwidth]{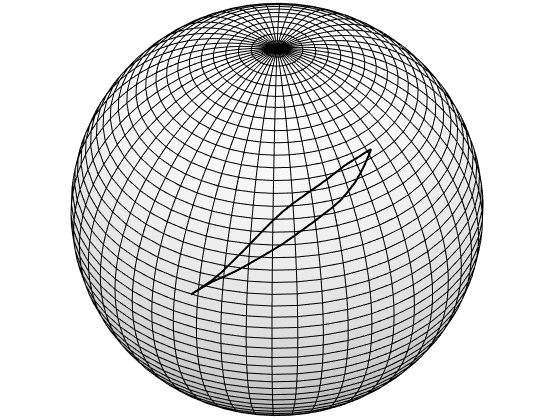}}
\caption{Plane trajectories between Paris and Caracas (left) and Fréchet mean curves of the two clusters given by agglomerative hierarchical clustering (right).}
\label{fig:plane}
\end{figure*}

Finally, we show illustrations on the sphere. Figure \ref{fig:hurricane} shows geodesics before (top row) and after (bottom row) optimal matching, for two pairs of curves representing hurricane tracks from the NASA Tropical Storm Tracks Database\footnote{https://ghrc.nsstc.nasa.gov/storms/.}. Once again we can see that the optimal matching algorithm seems to "flatten out" the deformation between two curves. In the last figure, we show a set of plane trajectories between Paris and Caracas downloaded from the IAGOS-MOZAIC database \footnote{http://iagos.sedoo.fr/.}. It is visually clear that there are two clusters, probably due to different weather conditions. These clusters are easily retrieved using agglomerative hierarchical clustering based on the horizontal geodesic distance. We can find the centers of these clusters by computing the Fréchet means, i.e. for each cluster, the curve that minimizes the sum of the squared distances to all the curves of the cluster. This can be achieved using a Karcher flow algorithm, summarized as follows. 
\begin{algo}[Karcher flow]\label{alg:karcher}
\leavevmode\par \noindent
Input: $\alpha_1, \hdots, \alpha_N$.\\
Initialization : $\hat\alpha = \alpha_1$.\\
Repeat until convergence :
\begin{enumerate}
\item for $i=1,\hdots,N$, 
\begin{itemize}
\item compute the horizontal geodesic $s\mapsto \alpha^{hor}(s)$ from $\hat\alpha$ to $\alpha_i$ using Algorithm \ref{alg:optmatchdis},
\item set $w_i = (\alpha^{hor})'(0)$,
\item update $\alpha_i = \alpha^{hor}(1)$.
\end{itemize}
\item set $w = \frac{1}{N}\sum_{i=1}^N w_i$,
\item update $\hat\alpha = \exp_{\hat\alpha}w$.
\end{enumerate}
Output : $\hat\alpha$.
\end{algo}
The obtained mean curves are shown on the right-hand side of Figure \ref{fig:plane}.

\section{Proof of Theorem \ref{thm}}		
\label{sec:proofthm}

We conclude this paper with the proof of Theorem \ref{thm}. Let us first remind the result.
\begin{thapp}
Let $s \mapsto c(s)$ be a $C^1$-path of $C^2$-curves with non vanishing derivative with respect to $t$. This path can be identified with an element $(s,t)\mapsto c(s,t)$ of $C^{1,2}([0,1]\times[0,1],M)$ such that $c_t \neq 0$. Consider the $C^1$-path in $M^{n+1}$, $s \mapsto \alpha(s)=(x_0(s), \hdots, x_n(s))$, that is the discretization of size $n$ of $c$. Then there exists a constant $\lambda>0$ such that for $n$ large enough, the difference between the energies of $c$ and $\alpha$ is bounded by
\begin{equation*}
| E(c) - E^n(\alpha)| \leq \frac{\lambda}{n}\,  (\inf |c_t|)^{-1} |c_s|_{2,\infty}^2 \left(1+|c_t|_{1,\infty}\right)^3,
\end{equation*}
where $E$ and $E^n$ are the energies with respect to metrics $G$ and $G^n$ respectively and where
\begin{align*}
|c_t|_{1,\infty} &:= |c_t|_\infty + |\nabla_tc_t|_\infty,\\
|c_s|_{2,\infty} &:= |c_s|_\infty + |\nabla_tc_s|_\infty + |\nabla^2_tc_s|_\infty,
\end{align*}
and $|w|_\infty := \sup_{s,t\in[0,1]}|w(s,t)|$ denotes the supremum over both $s$ and $t$ of a vector field $w$ along $c$.
\end{thapp}
\begin{proof}[Proof of Theorem 1]
To prove this result, we introduce the unique path $\hat c$ of piecewise geodesic curves of which $\alpha$ is the $n$-discretization. It is obtained by linking the points $x_0(s), x_1(s), \hdots, x_n(s)$ of $\alpha$ by pieces of geodesics for all times $s\in[0,1]$
\begin{align*}
&\,\,\hat c(s,\tfrac{k}{n}) = c(s,\tfrac{k}{n}) = x_k(s),\\
&\left.\hat c(s,\cdot)\right|_{\left[\frac{k}{n},\frac{k+1}{n}\right]} \text{ is a geodesic},
\end{align*} 
for $k=0,\hdots, n$. Then the difference between the energy of the path of curves $E(c)$ and the discrete energy of the path of discrete curves $E^n(\alpha)$ can be controlled in two steps :
\begin{equation*}
|E(c)-E^n(\alpha)| \leq |E(c)-E(\hat c)| + |E(\hat c)- E^n(\alpha)|.
\end{equation*}
\textbf{Step 1}. We first consider the difference between the continuous energies of the smooth and piecewise geodesic curves
\begin{align*}
|E(c) - E(\hat c)| &= \left| \int_0^1 \!\!\int_0^1\!\! \left( |\nabla_sq(s,t)|^2 \!-\! |\nabla_s\hat q(s,t)|^2 \right) \mathrm dt \,\mathrm ds \right|\\
&\leq  \int_0^1 \!\!\int_0^1 \left| |\nabla_sq(s,t)|^2 \!-\! |\nabla_s\hat q(s,t)|^2 \right| \mathrm dt \,\mathrm ds\\
&\leq  \int_0^1 \!\!\int_0^1 \left( |\nabla_sq(s,t)| + |\nabla_s\hat q(s,t)| \right) \cdot \\
&\hspace{3.5em}| \nabla_sq(s,t)^{t,\frac{k}{n}} - \nabla_s\hat q(s,t)^{t,\frac{k}{n}} | \mathrm dt \,\mathrm ds.
\end{align*}
Note that the parallel transports $\nabla_sq(s,t)^{t,\frac{k}{n}}$ and $\nabla_s\hat q(s,t)^{t,\frac{k}{n}}$ are performed along different curves -- $c(s,\cdot)$ and $\hat c(s,\cdot)$ respectively. Let us fix $s\in[0,1]$, $0\leq k \leq n$ and $t\in\left[\frac{k}{n},\frac{k+1}{n}\right]$. From now on we will omit "$s$" in the notation $w(s,t)$ to lighten notations. Using the notation $w^\parallel(t):=w(t)^{t,\frac{k}{n}}$ to denote the parallel transport of a vector field $w$ from $t$ to $\frac{k}{n}$ along its baseline curve, the difference we need to control is
\begin{align*}
&| \nabla_sq^\parallel - \nabla_s \hat q^\parallel| \\
&= \big| |c_t|^{-\frac{1}{2}}(\nabla_sc_t - \tfrac{1}{2} {\nabla_sc_t}^T)^\parallel \!-\! |\hat c_t|^{-\frac{1}{2}}(\nabla_s\hat c_t - \tfrac{1}{2} {\nabla_s\hat c_t}^T)^\parallel \big| \\
&= \big| (\nabla_s c_t-\tfrac{1}{2}{\nabla_s c_t}^T)^\parallel ( |c_t|^{-\frac{1}{2}} - |\hat c_t|^{-\frac{1}{2}}) \\
&\hspace{2em}+ |\hat c_t|^{-\frac{1}{2}}\big( ({\nabla_sc_t}^\parallel - {\nabla_s\hat c_t}^\parallel) -\tfrac{1}{2} ({\nabla_sc_t}^\parallel - {\nabla_s\hat c_t}^\parallel)^T \big) \big|.
\end{align*}
Since $|w-\frac{1}{2}w^T|\leq |w|$ for any vector $w$, we can write
\begin{equation}
\label{diffnsq0}
\begin{aligned}
&| \nabla_sq^\parallel - \nabla_s \hat q^\parallel| \leq | \nabla_sc_t| \cdot \big| |c_t|^{-1/2} \\
&\hspace{5em}- |\hat c_t|^{-1/2}\big| + |\hat c_t|^{-1/2} |{\nabla_sc_t}^\parallel - {\nabla_s\hat c_t}^\parallel |.
\end{aligned}
\end{equation}
Let us first consider the difference $|c_t^\parallel - \hat c_t^\parallel |$. Since $\hat c_t(t)^{t, \frac{k}{n}}=\hat c_t(\frac{k}{n})$, we can write
\begin{align*}
&|c_t(t)^{t,\frac{k}{n}} - \hat c_t(t)^{t,\frac{k}{n}}| \\
&\hspace{5em}\leq |c_t(t)^{t,\frac{k}{n}} - c_t(\tfrac{k}{n}) | + |c_t(\tfrac{k}{n}) - \hat c_t(\tfrac{k}{n}) |.
\end{align*}
The first term is smaller than $1/n \cdot |\nabla_tc_t|_\infty$. To bound the second term, we place ourselves in a local chart $(\varphi,U)$ centered in $c(\tfrac{k}{n})=c(s,\tfrac{k}{n})$, such that $c([0,1]\times[0,1])\subset U$. After identification with an open set of $\mathbb R^d$ -- where $d$ is the dimension of the manifold $M$-- using this chart, we get
\begin{align*}
|c_t(\tfrac{k}{n}) - \hat c_t(\tfrac{k}{n}) | &\leq \left|c_t(\tfrac{k}{n}) - n\left( c(\tfrac{k+1}{n}) - c(\tfrac{k}{n}) \right)\right| \\
&+ \left|\hat c_t(\tfrac{k}{n}) - n\left( c(\tfrac{k+1}{n}) - c(\tfrac{k}{n}) \right)\right|.
\end{align*}
Since a geodesic locally looks like a straight line (see e.g. \cite{emery84}) there exists a constant $\lambda_1$ such that 
\begin{equation*}
\big|\hat c_t(\tfrac{k}{n})- n(c(\tfrac{k+1}{n}) - c(\tfrac{k}{n}) )\big|\leq \lambda_1 \big|c(\tfrac{k+1}{n}) - c(\tfrac{k}{n})\big|^2,
\end{equation*} 
and so
\begin{equation*}
|c_t(\tfrac{k}{n}) - \hat c_t(\tfrac{k}{n}) | \leq \tfrac{1}{2n} |c_{tt}|_\infty + \tfrac{\lambda_1}{n} |c_t|_\infty^2.
\end{equation*}
The second derivative in $t$ of the coordinates of $c$ in the chart $(U,\varphi)$ can be written ${c_{tt}}^\ell={\nabla_tc_t}^\ell - \Gamma_{ij}^\ell {c_t}^i{c_t}^j$ for $\ell = 1,\hdots,d$, and so there exists a constant $\lambda_2$ such that $|c_{tt}| \leq \lambda_2 \left(|\nabla_tc_t|_\infty +|c_t|_\infty^2\right)$, and 
\begin{equation}
\label{diffct}
\big|c_t^\parallel - \hat c_t^\parallel\big| \leq \tfrac{\lambda_3}{n}\left(|c_t|_{1,\infty} + |c_t |_{1,\infty}^2\right).
\end{equation}
This means that for $n$ large enough, we can write e.g.
\begin{equation}
\label{ctbar}
\tfrac{1}{2}\inf |c_t|\leq |\hat c_t| \leq \tfrac{3}{2}|c_t|_{\infty}.
\end{equation}
From \eqref{diffct} we can also deduce that 
\begin{equation}
\label{diffctmoinsundemi}
\begin{aligned}
&\big||c_t|^{-\frac{1}{2}} - |\hat c_t|^{-\frac{1}{2}}\big| = \frac{\big||c_t| - |\hat c_t|\big|}{|c_t|^{\frac{1}{2}}+|\hat c_t|^{\frac{1}{2}}} \leq \frac{|c_t^\parallel - \hat c_t^\parallel |}{|c_t|^{\frac{1}{2}}+|\hat c_t|^{\frac{1}{2}}}\\
&\hspace{6em}\leq \tfrac{\lambda_3}{n}(\inf |c_t|)^{-\frac{1}{2}} \!\!\left(|c_t|_{1,\infty} + |c_t |_{1,\infty}^2\right).
\end{aligned}
\end{equation}
Let us now consider the difference $| {\nabla_sc_t}^\parallel - {\nabla_s \hat c_t}^\parallel |$. Since $c_s(s,\tfrac{k}{n}) = \hat c_s(s,\tfrac{k}{n})$, we get
\begin{align*}
&|\nabla_sc_t(t)^{t,\frac{k}{n}}- \nabla_s \hat c_t(t)^{t,\frac{k}{n}} | \leq \left| \nabla_t c_s(t)^{t,\frac{k}{n}} - \nabla_tc_s(\tfrac{k}{n}) \right|\\
&\hspace{2em} + \left| \nabla_t c_s(\tfrac{k}{n}) - n \left( c_s(\tfrac{k+1}{n})^{\frac{k+1}{n},\frac{k}{n}} - c_s(\tfrac{k}{n})\right) \right|\\
&\hspace{2em}+ \left| \nabla_t \hat c_s(t)^{t,\frac{k}{n}} - \nabla_t\hat c_s(\tfrac{k}{n}) \right| \\
&\hspace{2em}+ \left| \nabla_t \hat c_s(\tfrac{k}{n}) - n \left( \hat c_s(\tfrac{k+1}{n})^{\frac{k+1}{n},\frac{k}{n}} - \hat c_s(\tfrac{k}{n})\right) \right|
\end{align*}
and so
\begin{equation}
\label{diffnsct0}
\begin{aligned}
&|\nabla_sc_t(t)^{t,\frac{k}{n}}- \nabla_s \hat c_t(t)^{t,\frac{k}{n}} |\leq \tfrac{3}{2n} |\nabla_t^2c_s|_\infty + \tfrac{3}{2n}|\nabla_t^2\hat c_s|_\infty.
\end{aligned}
\end{equation}
We can decompose $\nabla_t^2\hat c_s(s,t) = \nabla_t\nabla_s\hat c_t(s,t) = \nabla_s\nabla_t\hat c_t(s,t) + \mathcal R(\hat c_t,\hat c_s)\hat c_t (s,t)$, and since $\nabla_t\hat c_t(s,t)=0$ and $| \mathcal R(X,Y)Z| \leq |K| \cdot(|\langle Y,Z\rangle| |X| + |\langle X,Z,\rangle| |Y|) \leq 2 |K|\cdot |X|\cdot |Y| \cdot |Z|$ by Cauchy Schwarz, we get using Equation \eqref{ctbar}
\begin{equation}
\label{ntntcs}
|\nabla_t^2\hat c_s|\leq 2 \,| \hat c_t|^2 \,|\hat c_s| \leq \tfrac{9}{2}|c_t|_\infty^2|\hat c_s|.
\end{equation}
To bound $|\hat c_s|$ we apply Lemma \ref{lemjacobi} to the Jacobi field $J : [0,1]\ni u\mapsto \hat c_s(s,\frac{k+u}{n})$ along the geodesic $\gamma(u) = \hat c(s,\frac{k+u}{n})$, that is
\begin{equation}
\label{ju}
\begin{aligned}
&J(u)^{u,0} = J(0)^T + a_k(u) J(0)^N\\
&\hspace{8em}+u\nabla_tJ(0)^T + b_k(u)\nabla_tJ(0)^N
\end{aligned}
\end{equation}
where, since $\gamma'(0)=\frac{1}{n}\hat c_t(s,\frac{k}{n})=\tau_k(s)$, the coefficients are defined by
\begin{equation*}
\left\{\begin{matrix*}[l]
a_k(u)=\cosh\!\left(|\tau_k| u \right), \,&b_k(u)=\frac{\sinh(|\tau_k| u)}{|\tau_k|},  \,& K=-1,\\
a_k(u)=1, \,& b_k(u)= u, \,& K=0,\\
a_k(u)=\cos\!\left(|\tau_k| u\right), \,& b_k(u)=\frac{\sin(|\tau_k| u)}{|\tau_k|}, \,& K=+1.
\end{matrix*} \right.
\end{equation*}
This gives $J(1)^{1,0} = J(0)^T + a_k(1) J(0)^N+\nabla_tJ(0)^N + b_k(1)\nabla_tJ(0)^N$ and so
\begin{align*}
&\nabla_tJ(0)^T = \big(J(1)^{1,0}-J(0)\big)^T\\
&\nabla_tJ(0)^N = b_k(1)^{-1}\big(J(1)^{1,0}-a_k(1)J(0)\big)^N.
\end{align*}
Injecting this into \eqref{ju}, we obtain since $u=nt-k$ and $\hat c_s(s,\tfrac{k}{n})=c_s(s,\frac{k}{n})$,
\begin{equation}
\label{cshat}
\begin{aligned}
&\hat c_s(t)^{t,\frac{k}{n}} = c_s(\tfrac{k}{n})^T + a_k(nt-k) c_s(\tfrac{k}{n})^N \\
&+(nt-k) \big( c_s(\tfrac{k+1}{n})^{\frac{k+1}{n},\frac{1}{n}} - c_s(\tfrac{k}{n})\big)^T \\ 
&+ \tfrac{b_k(nt-k)}{b_k(1)} \big(c_s(\tfrac{k+1}{n})^{\frac{k+1}{n},\frac{1}{n}} - a_k(1)c_s(\tfrac{k}{n})\big)^N.
\end{aligned}
\end{equation}
When $n \to \infty$, $a_k(1)\to 1$, $b_k(1)\to 1$, and since $0\leq nt-k \leq 1$, $a_k(nt-k) \to 1$, $b_k(nt-k)\to 1$ also. Therefore, for $n$ large enough we can see that $|\hat c_s| \leq \lambda_4|c_s|_\infty$ for some constant $\lambda_4$. Injecting this into \eqref{ntntcs} gives 
\begin{equation*}
|\nabla_t^2\hat c_s|_\infty \leq \tfrac{9\lambda_4}{2}|c_t|_\infty^2|c_s|_\infty, 
\end{equation*}
and so the difference \eqref{diffnsct0} can be bounded by
\begin{equation}
\label{diffnsct}
\begin{aligned}
|{\nabla_sc_t}^\parallel - {\nabla_s \hat c_t}^\parallel | &\leq \tfrac{3}{2n} \big(|\nabla_t^2c_s|_\infty + \tfrac{9\lambda_4}{2} |c_t|_{\infty}^2|c_s|_\infty \big)\\
& \leq \tfrac{\lambda_5}{n} \,|c_s|_{2,\infty}\big(1+ |c_t|_{1,\infty}^2 \big),
\end{aligned}
\end{equation}
for some constant $\lambda_5$. Injecting \eqref{ctbar}, \eqref{diffctmoinsundemi} and \eqref{diffnsct} in Equation \eqref{diffnsq0} we obtain
\begin{align}
\label{diffnsqfin}
|\nabla_sq^\parallel - \nabla_s\hat q^\parallel| &\leq \tfrac{\lambda_3}{n} (\inf |c_t|)^{-\frac{1}{2}} |c_s|_{2,\infty} \left( |c_t|_{1,\infty} + |c_t|_{1,\infty}^2\right)\nonumber\\
&+ \tfrac{\lambda_5 \sqrt{2}}{n} (\inf|c_t|)^{-\frac{1}{2}} |c_s|_{2,\infty} \left( 1+ |c_t|_{1,\infty}^2\right),\nonumber\\
|\nabla_sq^\parallel - \nabla_s\hat q^\parallel|&\leq \tfrac{\lambda_{6}}{n}\, (\inf |c_t|)^{-\frac{1}{2}} |c_s|_{2,\infty} \left(1+|c_t|_{1,\infty}\right)^2,
\end{align}
for some constant $\lambda_6$. To conclude this first step, let us bound the sum\vspace{-0.5em}
\begin{equation}
\label{sommensq0}
\begin{aligned}
&|\nabla_sq|+|\nabla_s\hat q| &\\
&= |c_t|^{-\frac{1}{2}}|\nabla_sc_t - \tfrac{1}{2}{\nabla_sc_t}^T| + |\hat c_t|^{-\frac{1}{2}}|\nabla_s\hat c_t -\tfrac{1}{2}{\nabla_s\hat c_t}^T|\\
&\leq (\inf|c_t|)^{-\frac{1}{2}} |\nabla_tc_s|_\infty + \sqrt{2} (\inf|c_t|)^{-\frac{1}{2}}|\nabla_t\hat c_s|_\infty.
\end{aligned}
\end{equation}
Taking the derivative according to $t$ on both sides of \eqref{cshat}, we get since $n|\tau_k(s)|=|\hat c_t(s,\frac{k}{n})|$,
\begin{align*}
&\nabla_t\hat c_s(t)^{t,\frac{k}{n}} = |\hat c_t(\tfrac{k}{n})| e_k(nt-k) c_s(\tfrac{k}{n})^N \\
&+n\left( c_s(\tfrac{k+1}{n})^\parallel - c_s(\tfrac{k}{n})\right)^T \\
&+ n \,\tfrac{ a_k(nt-k)}{b_k(1)} \left(c_s(\tfrac{k+1}{n})^\parallel - a_k(1) c_s(\tfrac{k}{n})\right)^N,
\end{align*}
since derivation of the coefficients give $b_k'(u)=a_k(u)$ and $a_k'(u) = |\tau_k|e_k(u)=\frac{1}{n}|\hat c_t(\frac{k}{n})|e_k(u)$, where
\begin{equation*}
e_k(u)=\left\{\begin{matrix*}[l]
\sinh\left(  |\tau_k| u \right), & \text{if } K=-1,\\
0 & \text{if } K=0,\\
-\sin\left( |\tau_k| u \right), & \text{if } K=+1.
\end{matrix*} \right.
\end{equation*}
Since the coefficients $e_k(nt-k)$, $a_k(nt-k)/b_k(1)$ and $a_k(1)$ are bounded for $n$ large enough, and since $|\hat c_t|\leq \frac{3}{2}|c_t|_\infty$, we can write for some constant $\lambda_7$,
\begin{equation}
\label{ntcsbar}
\begin{aligned}
|\nabla_t\hat c_s|_\infty &\leq \lambda_7\left( |\hat c_t|_\infty |c_s|_\infty + |\nabla_tc_s|_\infty \right)\\
&\leq \tfrac{3\lambda_7}{2}|c_s|_{2,\infty}\left( 1+ |c_t|_{1,\infty}\right).
\end{aligned}
\end{equation}
Inserting this into \eqref{sommensq0} gives
\begin{equation}
\label{sommensq}
\begin{aligned}
&|\nabla_sq|+|\nabla_s\hat q| \\
&\leq (\inf |c_t|)^{-1/2} \big(|\nabla_tc_s|_\infty + \tfrac{3\lambda_7}{\sqrt{2}}|c_s|_{2,\infty}\left( 1+ |c_t|_{1,\infty}\right)\big)\\
&\leq \lambda_8 (\inf |c_t|)^{-1/2} |c_s|_{2,\infty}(1+|c_t|_{1,\infty}). 
\end{aligned}
\end{equation}
Finally, we are able to bound the difference between the energies of the smooth and piecewise-geodesic paths by combining Equations \eqref{diffnsqfin} and \eqref{sommensq}
\begin{equation*}
|E(c) - E(\hat c)| \leq \frac{\lambda_6\lambda_8}{n}\,(\inf|c_t|)^{-1} |c_s|_{2,\infty}^2 \left(1+|c_t|_{1,\infty}\right)^3.
\end{equation*}
\textbf{Step 2}. Let us now consider the difference of energy between the path of piecewise geodesic curves and the path of discrete curves. Since $\nabla_sq_k(s) = \nabla_s\hat q(s,\tfrac{k}{n})$ for all $s\in[0,1]$ and $0\leq k\leq n$, we can write
\begin{align*}
&|E(\hat c) - E^n(\alpha)| \\
&= \left| \int_0^1\left( \int_0^1 |\nabla_s\hat q(s,t)|^2 \mathrm dt - \frac{1}{n} \sum_{k=0}^{n-1} |\nabla_sq_k(s)|^2 \right) \mathrm ds \right| \\
&\leq \sum_{k=0}^{n-1} \int_0^1  \int_{\frac{k}{n}}^{\frac{k+1}{n}} \left| \,|\nabla_s\hat q(s,t)|^2 - | \nabla_s\hat q(s,\tfrac{k}{n}) |^2 \,\right| \mathrm dt \, \mathrm ds \\
&\leq \sum_{k=0}^{n-1} \int_0^1 \!\!\int_{\frac{k}{n}}^{\frac{k+1}{n}} \big(|\nabla_s\hat q(s,t)|+ | \nabla_s\hat q(s,\tfrac{k}{n}) | \big) \cdot\\
&\hspace{8em}\big|\nabla_s\hat q(s,t)^{t,\frac{k}{n}} - \nabla_s\hat q(s,\tfrac{k}{n}) \big| \, \mathrm dt \, \mathrm ds.
\end{align*}
We fix once again $s\in[0,1]$, $0\leq k \leq n$ and $t\in\hspace{-0.1em}\left[\frac{k}{n},\frac{k+1}{n}\right]$. As in step 1, we will omit "$s$" in most notations. Since $|\hat c_t(t)|= |\hat c_t(\frac{k}{n})|$, we get
\begin{align*}
&\big| \nabla_s\hat q(t)^{t,\frac{k}{n}} - \nabla_s\hat q(\tfrac{k}{n}) \big| \\
& \leq  \Big||\hat c_t(\tfrac{k}{n})|^{-\frac{1}{2}} \Big(\nabla_s\hat c_t(t)^{t,\frac{k}{n}} - \nabla_s\hat c_t(\tfrac{k}{n}) \\
& \hspace{10em}- \tfrac{1}{2} \big(\nabla_s\hat c_t(t)^{t,\frac{k}{n}} - \nabla_s\hat c_t(\tfrac{k}{n})\big)^T\Big)\Big|\\
&\leq |\hat c_t(\tfrac{k}{n})|^{-\frac{1}{2}} \big|\nabla_s\hat c_t(t)^{t,\frac{k}{n}} - \nabla_s\hat c_t(\tfrac{k}{n})\big|.
\end{align*}
Considering once again the Jacobi field 
\begin{equation*}
J(u) := \hat c_s(\tfrac{k+u}{n}), \quad u\in[0,1],
\end{equation*} 
along the geodesic $\gamma(u) = \hat c(\frac{k+u}{n})$, Equation \eqref{ju} gives
\begin{align*}
&\hat c_s(t)^{t,\frac{k}{n}} = c_s(\tfrac{k}{n})^T \!+ a_k(nt-k) c_s(\tfrac{k}{n})^N \\
&\hspace{2em}+(t-\tfrac{k}{n}) \nabla_t\hat c_s(\tfrac{k}{n})^T + b_k(nt-k) \tfrac{1}{n}\nabla_t\hat c_s(\tfrac{k}{n})^N.
\end{align*}
Recall that $b_k'(u)=a_k(u)$ and $a_k'(u) = |\tau_k|e_k(u)$, and so taking the derivative with respect to $t$ and decomposing $\nabla_t\hat c_s(\tfrac{k}{n})^T=\nabla_t\hat c_s(\tfrac{k}{n})-\nabla_t\hat c_s(\tfrac{k}{n})^N$, we obtain
\begin{align*}
\nabla_t\hat c_s(t)^{t,\frac{k}{n}} -  \nabla_t\hat c_s(\tfrac{k}{n}) &= |\hat c_t(\tfrac{k}{n})| e_k(nt-k) \cdot c_s(\tfrac{k}{n})^N   \\
&+ \big(a_k(nt-k)-1\big)\nabla_t\hat c_s(\tfrac{k}{n})^N.
\end{align*}
Noticing that $\frac{e_k(nt-k)}{(nt-k)|\tau_k|} \to 1$ and $\frac{a_k(nt-k)-1}{(nt-k)|\tau_k|}\to 0$ when $n\to \infty$, we can deduce that for $n$ large enough, 
\begin{align*}
&|e_k(nt-k) | \leq 2(nt-k) |\tau_k| \leq 2|\tau_k| =\tfrac{2}{n}|c_t| \leq \tfrac{2}{n} |c_t|_\infty,\\
&|a_k(nt-k)-1| \leq (nt-k)|\tau_k| \leq |\tau_k| = \tfrac{1}{n}|c_t| \leq \tfrac{1}{n}|c_t|_\infty.
\end{align*}
This gives
\begin{align*}
&\int_{\frac{k}{n}}^{\frac{k+1}{n}} \big|\nabla_t\hat c_s(s,t)^{t,\frac{k}{n}} -  \nabla_t\hat c_s(s,\tfrac{k}{n^2})|\, \mathrm dt \\
&\hspace{8em}\leq \tfrac{2}{n^2}\big( |c_t|_\infty^2 |c_s|_\infty  + |c_t|_\infty|\nabla_t\hat c_s|_\infty\big).
\end{align*}
Recall from \eqref{ntcsbar} and \eqref{sommensq} that 
\begin{align*}
&|\nabla_t\hat c_s|_\infty \leq \tfrac{3\lambda_7}{2}|c_s|_{2,\infty}\left( 1+ |c_t|_{1,\infty}\right),\\
&|\nabla_s\hat q|_\infty \leq \tfrac{3\lambda_7}{\sqrt{2}}(\inf |c_t|)^{-\frac{1}{2}}|c_s|_{2,\infty}\left( 1+ |c_t|_{1,\infty}\right),
\end{align*}
and so
\begin{align*}
\label{int}
&\int_{\frac{k}{n}}^{\frac{k+1}{n}} \!\big(|\nabla_s\hat q(t)|+ | \nabla_s\hat q(\tfrac{k}{n}) | \big) \cdot | \nabla_s\hat q(t)^{t,\frac{k}{n}}-\nabla_s\hat q(\tfrac{k}{n})| \,\mathrm dt \nonumber\\
&\leq 2 |\nabla_s\hat q|_\infty \sqrt{2}(\inf|c_t|)^{-\frac{1}{2}}\!\!\int_{\frac{k}{n}}^{\frac{k+1}{n}}\!\! \big|\nabla_t\hat c_s(t)^{t,\frac{k}{n}} -  \nabla_t\hat c_s(\tfrac{k}{n})|\, \mathrm dt\\
&\leq 6\lambda_7 (\inf|c_t|)^{-1}|c_s|_{2,\infty}(1+|c_t|_{1,\infty}) \tfrac{2}{n^2} \big(|c_t|_\infty^2 |c_s|_\infty \\
&\hspace{10em} + |c_t|_\infty\tfrac{3\lambda_7}{2}|c_s|_{2,\infty}\left( 1+ |c_t|_{1,\infty}\right)\big)\\
&\leq \tfrac{\lambda_9}{n^2}\, (\inf|c_t|)^{-1} |c_s|_{2,\infty}^2(1+|c_t|_{1,\infty})^3.
\end{align*}
Finally, we obtain
\begin{equation*}
|E(\hat c) - E^n(\alpha)| \leq \, \frac{\lambda_9}{n} \, (\inf|c_t|)^{-1} |c_s|^2_{2,\infty} \left(1+|c_t|_{1,\infty}\right)^3,
\end{equation*}
which completes the proof.
\end{proof}

\section*{Acknowledgments}

 This research was supported by Thales Air Systems and the french MoD DGA. MOZAIC/CARIBIC/IAGOS data were created with support from the European Commission, national agencies in Germany (BMBF), France (MESR), and the UK (NERC), and the IAGOS member institutions (http://www.iagos.org/partners). The participating airlines (Lufthansa, Air France, Austrian, China Airlines, Iberia, Cathay Pacific, Air Namibia, Sabena) supported IAGOS by carrying the measurement equipment free of charge since 1994. The data are available at http://www.iagos.fr thanks to additional support from AERIS.


\section*{Appendix A} 

\textbf{Lemma 1} \textit{
Let $\gamma : [0,1] \rightarrow M$ be a geodesic of a manifold $M$ of constant sectional curvature $K$, and $J$ a Jacobi field along $\gamma$. Then the parallel transport of $J(t)$ along $\gamma$ from $\gamma(t)$ to $\gamma(0)$ is given by 
\begin{align*}
J(t)^{t,0} &= J^T(0) + \tilde a_k(t) J^N(0)  \\
&\hspace{7em}+ t \,\nabla_tJ^T(0) +  \tilde b_k(t)\nabla_tJ^N(0),
\end{align*}
for all $t\in[0,1]$, where
\begin{equation*}
\left\{\begin{matrix*}[l]
\tilde a_k(t)=\cosh\left( |\gamma'(0)| t \right), &\tilde b_k(t)=\frac{\sinh(|\gamma'(0)| t)}{|\gamma'(0)|},  & K=-1,\\
\tilde a_k(t)=1, &\tilde b_k(t)= t, & K=0,\\
\tilde a_k(t)=\cos\left( |\gamma'(0)| t \right), &\tilde b_k(t)=\frac{\sin(|\gamma'(0)| t)}{|\gamma'(0)|}, & K=+1.
\end{matrix*} \right.
\end{equation*}
}
\begin{proof}
As a Jacobi field along $\gamma$, $J$ satisfies the well-known equation (see e.g. \cite{jost})
\begin{equation*}
\nabla_t^2 J(t) = - \mathcal R(J(t),\gamma' (t))\gamma'(t).
\end{equation*}
If $M$ is flat, we get $\nabla_t^2J(t)=0$ and so $J(t)=J(0)+t\nabla_tJ(0)$. If not, we can decompose $J$ in the sum $J=J^T + J^N$ of two vector fields that parallel translate along $\gamma$, $J^T=\langle J,v\rangle v$ with $v=\gamma'/|\gamma'|$, and $J^N = J-J^T$. Since $\langle \nabla_t^2 J, \gamma' \rangle=0$ and $\gamma'$ is parallel along $\gamma$, we get by integrating twice that
\begin{align*}
\langle J(t),\gamma' (t) \rangle&=\langle \nabla_tJ(0),\gamma' (0) \rangle t+\langle J(0),\gamma'(0)\rangle,\\
\langle J(t),v(t) \rangle&= \langle \nabla_tJ(0),v(0) \rangle t+\langle J(0),v(0) \rangle.
\end{align*}
Since 
\begin{equation*}
\nabla_t^2J^T(t)=\nabla_t^2\langle J(t),v(t)\rangle v(t)=\langle \nabla_t^2J(t),v(t)\rangle v(t)=0, 
\end{equation*}
the normal component $J^N$ is also a Jacobi field, that is it verifies 
\begin{equation*}
\nabla_t^2J^N(t) = -\mathcal R(J^N(t), \gamma'(t))\gamma'(t). 
\end{equation*}
And since $M$ has constant sectional curvature $K$, for any vector field $w$ along $\gamma$ we have
\begin{align*}
\langle \mathcal R(J^N,\gamma')\gamma' ,w \rangle &= K \left( \langle \gamma',\gamma'\rangle \langle J^N,w\rangle \!-\! \langle J^N, \gamma' \rangle\langle \gamma', w \rangle \right)\\
&= \langle K |\gamma'|^2J^N,w \rangle,
\end{align*}
and the differential equation verified by $J^N$ can be rewritten $\nabla_t^2 J^N(t)=-K \, |\gamma'(t)|^2 \, J^N(t)$. Since the speed of the geodesic $\gamma$ has constant norm, the solution to that differential equation is of the form 
\begin{equation*}
J^N(t)=\left( \lambda e^{ |\gamma'(0)|  t } + \mu e^{-|\gamma'(0)| t }\right)\omega(t),
\end{equation*}
when $K=-1$ and
\begin{equation*}
J^N(t)=\left( \lambda e^{ i |\gamma'(0)| t} + \mu e^{ - i|\gamma'(0)| t}\right)\omega(t),
\end{equation*}
when $K=1$. Using the initial conditions $J^N(0)$ and $\nabla_tJ^N(0)$ to find the constants $\lambda,\mu$, we get for $K=-1$
\begin{equation*}
J^N(t) = J^N(0) \cosh\left( |\gamma'(0)| t \right) + \nabla_tJ^N(0) \tfrac{\sinh\left(|\gamma'(0)| t \right)}{|\gamma'(0)|},
\end{equation*}
and for $K=1$, the same formula with cosine and sine functions instead of hyperbolic cosine and sine.
\end{proof}

\noindent
\textbf{Lemma 3}\textit{
The covariant derivatives of the functions $f_k^{(-)}$ and $g_k^{(-)}$ with respect to $s$ are functions $T_{x_{k+1}(s)}M \rightarrow T_{x_k(s)}M$ given by
\begin{align*}
\nabla_s\big(f_k^{(-)}\big) &: w \mapsto (\nabla_sf_k)^{(-)}(w) + f_k\big(\mathcal R\left(Y_k, \tau_k \right)({w_{k+1}}^\parallel)\big),\\
\nabla_s\big(g_k^{(-)}\big) &: w \mapsto (\nabla_sg_k)^{(-)}(w) + g_k\big(\mathcal R\left(Y_k, \tau_k \right)({w_{k+1}}^\parallel)\big),
\end{align*}
where 
\begin{align*}
&(\nabla_sf_k)(s)^{(-)} = \nabla_sf_k(s) \circ P^{x_{k+1}(s),x_k(s)}_{\gamma_k(s)},\\
&(\nabla_sg_k)(s)^{(-)} = \nabla_sg_k(s) \circ P^{x_{k+1}(s),x_k(s)}_{\gamma_k(s)},\\
&Y_k = ({x_k}')^T + b_k ({x_k}')^N + \tfrac{1}{2}{\nabla_s\tau_k}^T + K\frac{1-a_k}{|\tau_k|^2} \,{\nabla_s\tau_k}^N,\label{yk}
\end{align*}
if $K$ is the sectional curvature of the base manifold.
}
\begin{proof}
Fix $0\leq k \leq n$ and let $w_{k+1} : s \mapsto w_{k+1}(s)$ be a vector field along the curve $x_{k+1}:s\mapsto x_{k+1}(s)$. By definition,
\begin{align*}
&\nabla_s\big(f_k^{(-)}(w_{k+1})\big) = \nabla_s\big(f_k^{(-)}\big)(w_{k+1}) + f_k^{(-)}(\nabla_sw_{k+1}),\\
&\nabla_s\big(g_k^{(-)}(w_{k+1})\big) = \nabla_s\big(g_k^{(-)}\big)(w_{k+1}) + g_k^{(-)}(\nabla_sw_{k+1}).
\end{align*}
Consider the path of gedesics $ s \mapsto \gamma_k(s,\cdot)$ such that for all $s\in[0,1]$, $\gamma_k(s,0) = x_k(s)$, $\gamma_k(s,1) = x_{k+1}(s)$ and $t \mapsto \gamma_k(s,t)$ is a geodesic. We denote by ${w_{k+1}}^\parallel$ the vector field along the curve $x_k$ obtained by parallel transporting back the vector $w_{k+1}(s)$ along the geodesic $\gamma_k(s,\cdot)$ for all $s\in[0,1]$, i.e. ${w_{k+1}}^\parallel(s) = P_{\gamma_k(s,\cdot)}^{1,0}(w_{k+1}(s))$. We have
\begin{equation}
\label{nsfkmoins}
\begin{aligned}
\nabla_s\big(f_k^{(-)}(w_{k+1})\big) &= \nabla_s\big(f_k({w_{k+1}}^\parallel)\big) \\
&= \nabla_sf_k({w_{k+1}}^\parallel) + f_k\big(\nabla_s({w_{k+1}}^\parallel)\big),
\end{aligned}
\end{equation}
and so we need to compute $\nabla_s({w_{k+1}}^\parallel)$. Let $V(s,t):=P_{\gamma_k(s,\cdot)}^{1,t}(w_{k+1})$ so that $\nabla_sV(s,1) = \nabla_sw_{k+1}$ and $\nabla_sV(s,0) = \nabla_s({w_{k+1}}^\parallel)$, then 
\begin{align*}
&\nabla_sV(s,1)^{1,0} = \nabla_sV(s,0) + \int_0^1 \nabla_t\nabla_sV(s,t)^{t,0} \mathrm dt, \\
&=\nabla_sV(s,0) + \int_0^1 \mathcal R({\partial_t\gamma_k}^{t,0},{\partial_s\gamma_k}^{t,0})V(s,t)^{t,0} \mathrm dt,
\end{align*}
since $\nabla_tV=0$, and where $\partial_t\gamma_k(s,t)^{t,0} = \tau_k(s)$. We get, since $\nabla R=0$,
\begin{equation}
\label{nswparint}
(\nabla_sw_{k+1})^\parallel = \nabla_s({w_{k+1}}^\parallel)+ \mathcal R\!\!\left(\tau_k, \int_0^1 {\partial_s\gamma_k}^{t,0} \mathrm dt \right)\!\!({w_{k+1}}^\parallel).
\end{equation}
To find an expression for ${\partial_s\gamma_k}^{t,0}$, we consider the Jacobi field $J(t) := \partial_s\gamma_k(s,t)$ along the geodesic $t\mapsto \gamma_k(s,t)$. The vector field $J$ verifies
\begin{align*}
J(0) = {x_k}'(s),\, J(1) = {x_{k+1}}'(s),\, \nabla_tJ(0) = \nabla_s \tau_k(s),
\end{align*}
where the last equality results from the inversion $\nabla_t\partial_s\gamma_k(s,0) = \nabla_s\partial_t\gamma_k(s,0)$ and $\partial_t\gamma_k(s,0) = \tau_k(s)$. Applying Lemma \ref{lemjacobi} gives, for all $k=0,\hdots,n-1$,
\begin{align*}
{\partial_s\gamma_k(s,t)}^{t,0} = {x_k}'(s)^T &+ a_k(s,t) \,{x_k}'(s)^N \\
&+ t\,\nabla_s\tau_k(s)^T +b_k(s,t) \nabla_s\tau_k(s)^N.
\end{align*}
with the coefficients 
\begin{equation*}
a_k(s,t)=\left\{\begin{matrix*}[l]
\cosh\left( |\tau_k(s)| t \right), &  \text{if } K=-1,\\
1 &  \text{if } K=0,\\
\cos\left( |\tau_k(s)| t \right), & \text{if } K=+1,
\end{matrix*} \right.
\end{equation*}
\begin{equation*}
b_k(s,t)=\left\{\begin{matrix*}[l]
\sinh\left(|\tau_k(s)| t \right)/|\tau_k(s)|   & \text{if } K=-1,\\
1 & \text{if } K=0,\\
\sin\left(|\tau_k(s))| t \right)/|\tau_k(s)| & \text{if } K=+1.
\end{matrix*} \right.
\end{equation*}
Integrating this and injecting it in \eqref{nswparint} gives
\begin{equation}
\label{nswpar}
\nabla_s({w_{k+1}}^\parallel) = (\nabla_sw_{k+1})^\parallel + \mathcal R\left(Y_k, \tau_k \right)({w_{k+1}}^\parallel),
\end{equation}
where $Y_k$ is defined by
\begin{equation*}
Y_k = ({x_k}')^T + b_k ({x_k}')^N + \tfrac{1}{2}{\nabla_s\tau_k}^T + K\frac{1-a_k}{|\tau_k|^2} \,{\nabla_s\tau_k}^N,
\end{equation*}
and injecting this in \eqref{nsfkmoins} finally gives, 
\begin{align*}
\nabla_s\big(f_k^{(-)}(w_{k+1})\big) &= \nabla_sf_k({w_{k+1}}^\parallel) + f_k\big( (\nabla_sw_{k+1})^\parallel \big) \\
&\hspace{5em}+ f_k\big(\mathcal R\left(Y_k, \tau_k \right)({w_{k+1}}^\parallel)\big)\\
&= (\nabla_sf_k)^{(-)}(w_{k+1}) + f_k^{(-)}(\nabla_sw_{k+1}) \\
&\hspace{5em}+f_k\big(\mathcal R\left(Y_k, \tau_k \right)({w_{k+1}}^\parallel)\big),
\end{align*}
which is what we wanted. The covariant derivative $\nabla_s\big(g_k^{(-)}(w_{k+1})\big)$ can be computed in a similar way.
\end{proof}

\section*{Appendix B}

\textbf{Proposition 6 (Discrete geodesic equations)}\textit{
A path $s \mapsto \alpha(s) = \left( x_0(s), \hdots, x_n(s) \right)$ in $M^{n+1}$ is a geodesic for metric $G^n$ if and only if its SRV representation $s\mapsto \big(x_0(s), (q_k(s))_{k}\big)$ verifies the following differential equations
\begin{equation*}
\begin{aligned}
&\nabla_s{x_0}' + \frac{1}{n} \Big( R_0 + f_0^{(-)}(R_1)  \\
&\hspace{7em}+ \hdots + f_0^{(-)}\circ \cdots \circ f_{n-2}^{(-)} (R_{n-1})\Big)  = 0, \\
&\nabla_s^2q_k + \frac{1}{n} \,\,g_k^{(-)}\Big( R_{k+1} + f_{k+1}^{(-)}(R_{k+2}) \\
& \hspace{7em}+ \hdots + f_{k+1}^{(-)} \circ \cdots \circ f_{n-2}^{(-)}(R_{n-1})\Big) = 0,
\end{aligned}
\end{equation*}
for all $k=0, \hdots, n-1$, with the notations \eqref{not} and $R_k := \mathcal R(q_k,\nabla_sq_k){x_k}'$.
}

\begin{proof}
We consider a variation $(-\delta, \delta) \ni a \mapsto \alpha(a,\cdot)=(x_0(a,\cdot),\hdots, x_n(a,\cdot))$ of this curve which coincides with $\alpha$ for $a=0$, i.e. $\alpha(0,s) = \alpha(s)$ for all $s\in[0,1]$, and which preserves the end points of $\alpha$, i.e. $\alpha(a,0)=\alpha(0)$ and $\alpha(a,1)=\alpha(1)$ for all $a\in (-\delta, \delta)$. The energy of this variation with respect to metric $G^n$ can be seen as a real function of the variable $a$ and is given by
\begin{equation*}
E^n(a) = \frac{1}{2} \int_0^1 \left( |\partial_s x_0(a,s)|^2 + \frac{1}{n} \sum_{k=0}^{n-1} |\nabla_sq_k(a,s) |^2 \right) \mathrm ds,
\end{equation*}
and its derivative $(E^n)'(a)$ with respect to $a$ is
\begin{align*}
&\int_0^1 \left( \Big\langle \partial_a\partial_s x_0, \partial_s x_0 \Big\rangle
+ \frac{1}{n} \sum_{k=0}^{n-1} \Big\langle \nabla_a\nabla_sq_k, \nabla_sq_k\Big\rangle \right) \mathrm ds \\
&= \int_0^1 \Bigg( \Big\langle \partial_s\partial_a x_0, \partial_s x_0 \Big\rangle \\
&\hspace{1em}+ \frac{1}{n} \sum_{k=0}^{n-1} \Big\langle \nabla_s\nabla_aq_k + \mathcal R(\partial_ax_k,\partial_sx_k)q_k, \nabla_sq_k\Big\rangle \Bigg) \mathrm ds, \\
&= - \int_0^1 \Bigg( \Big\langle \nabla_s\left(\partial_sx_0\right), \partial_ax_0\Big\rangle + \frac{1}{n} \sum_{k=0}^{n-1} \Big\langle \nabla_s^2q_k, \nabla_aq_k \Big\rangle \\
&\hspace{1em}+ \Big\langle \mathcal R(q_k, \nabla_sq_k)\partial_sx_k, \partial_ax_k \Big\rangle \Bigg)\mathrm ds,
\end{align*}
where we integrate by parts to obtain the third line from the second. The goal is to express $\partial_ax_k$ in terms of $\partial_ax_0$ and $\nabla_aq_\ell$, $\ell=0,\cdots,k$. That way, the only elements that depend on $a$ once we take $a=0$ are $(\partial_ax_0, \nabla_aq_0, \cdots, \nabla_aq_{n-1})$ which can be chosen \emph{independently} to be whatever we want. Let us fix $0\leq k \leq n-1$ and $s\in[0,1]$ and consider the path of geodesics $a \mapsto \gamma_k(a,\cdot)$ such that $\gamma_k(a,0) = x_k(a,s)$, $\gamma_k(a,1) = x_{k+1}(a,s)$ and $\partial_t\gamma_k(a,0) = \tau_k(a,s) = \log_{x_k(a,s)}(x_{k+1}(a,s))$. Then by definition, for each $a\in[0,1]$, $t\mapsto J(a,t):=\partial_a\gamma_k(a,t)$ is a Jacobi field along the geodesic $t\mapsto \gamma_k(a,t)$ of $M$, and so Lemma \ref{lemjacobi} gives
\begin{equation}
\label{xkplus1par}
{\partial_ax_{k+1}}^\parallel = {\partial_ax_k}^T + a_k \,{\partial_ax_k}^N + {\nabla_a\tau_k}^T + b_k {\nabla_s\tau_k}^N,
\end{equation}
where ${\partial_a x_{k+1}}^\parallel$ denotes the parallel transport of $\partial_a x_{k+1}$ from $x_{k+1}(s)$ to $x_k(s)$ along the geodesic. Differentiation of $q_k = \sqrt{n} \, \tau_k/|\tau_k|$ gives 
\begin{equation*}
\nabla_sq_k = \sqrt{n} \, |\tau_k|^{-1/2} \left( \nabla_s\tau_k - \tfrac{1}{2} {\nabla_s\tau_k}^T \right),
\end{equation*}
and taking the tangential part on both sides yields ${\nabla_sq_k}^T = \sqrt{n} \, |\tau_k|^{-1/2} \frac{1}{2} {\nabla_s\tau_k}^T$, and so finally 
\begin{align*}
\nabla_s\tau_k &= |\tau_k|^{1/2}/\sqrt{n} \left( \nabla_sq_k + {\nabla_sq_k}^T\right)\\
&= |q_k|/n \left( \nabla_sq_k + {\nabla_sq_k}^T\right). 
\end{align*}
Injecting this in \eqref{xkplus1par} and noticing that $\langle f_k(w),z\rangle = \langle w,f_k(z)\rangle$ and $\langle g_k(w),z\rangle = \langle w,g_k(z)\rangle$ for any pair of vectors $w,z$ gives 
 \begin{align}
&{\partial_a x_{k+1}}^\parallel = f_k(\partial_a x_k) + \frac{1}{n} \, g_k(\nabla_aq_k),  \label{rec1}\\
&\big\langle w_{k+1}, \partial_ax_{k+1} \big\rangle = \big\langle f_k^{(-)}(w_{k+1}), \partial_ax_k \big\rangle \nonumber\\
&\hspace{10em}+ \frac{1}{n} \big\langle g_k^{(-)}(w_{k+1}), \nabla_aq_k \big\rangle,\label{rec2}
\end{align}
for any tangent vector $w_{k+1}\in T_{x_{k+1}}M$. From equation \eqref{rec2} we can deduce, for $k=1,\hdots, n$,
\begin{align*}
\big\langle w_k, \partial_ax_k \big\rangle &= \Big\langle f_0^{(-)} \circ \cdots \circ f_{k-1}^{(-)}(w_k), \partial_ax_0 \Big\rangle \\
+& \frac{1}{n} \sum_{\ell=0}^{k-1} \Big\langle g_\ell^{(-)} \circ f_{\ell +1}^{(-)}\circ \cdots \circ f_{k-1}^{(-)}(w_k), \nabla_aq_\ell \Big\rangle.
\end{align*}
With the notation $R_k := \mathcal R(q_k,\nabla_sq_k){x_k}'$ we get
\begin{align*}
\big\langle R_k, \partial_ax_k \big\rangle &= \Big\langle f_0^{(-)} \circ \cdots \circ f_{k-1}^{(-)}(R_k), \partial_ax_0 \Big\rangle \\
+& \frac{1}{n} \sum_{\ell=0}^{k-1} \Big\langle g_\ell^{(-)} \circ f_{\ell +1}^{(-)}\circ \cdots \circ f_{k-1}^{(-)}(R_k), \nabla_aq_\ell \Big\rangle,
\end{align*}
and we can then write the derivative of the energy $(E^n)'(0)$ for $a=0$ in the following way
\begin{align*}
& - \int_0^1 \bigg( \, \Big\langle \nabla_s{x_0}' + \frac{1}{n} \sum_{k=0}^{n-1} f_0^{(-)}\circ \cdots \circ f_{k-1}^{(-)} (R_k), \partial_a x_0 \Big\rangle \\
&+ \frac{1}{n^2} \sum_{k=1}^{n-1} \sum_{\ell=0}^{k-1} \Big\langle g_\ell^{(-)} \circ f_{\ell+1}^{(-)} \circ \cdots \circ f_{k+1}^{(-)} (R_k), \nabla_aq_\ell \Big\rangle\\
&+ \frac{1}{n} \sum_{k=0}^{n-1} \big\langle \nabla_s^2q_k, \nabla_aq_k \Big\rangle \,\bigg) \mathrm ds,
\end{align*}
where in the first sum we use the notation convention $f_0 \circ \cdots \circ f_{-1} := \text{Id}$. Noticing that the double sum can be rewritten
\begin{equation*}
\sum_{\ell=0}^{n-2}\sum_{k=\ell+1}^{n-1} \Big\langle g_\ell^{(-)} \circ f_{\ell+1}^{(-)} \circ \cdots \circ f_{k-1}^{(-)} (R_k), \nabla_aq_\ell \Big\rangle,
\end{equation*}
we obtain for $(E^n)'(0)$
\begin{equation}
\label{energy}
\begin{aligned}
&- \int_0^1 \bigg( \,\bigg\langle \nabla_s{x_0}' + \frac{1}{n} \sum_{k=0}^{n-1} f_0^{(-)}\circ \cdots \circ f_{k-1}^{(-)}(R_k), \partial_a x_0 \bigg\rangle \\
&+ \frac{1}{n} \sum_{k=0}^{n-1} \bigg\langle \nabla_s^2q_k + \frac{1}{n} \sum_{\ell=k+1}^{n-1} g_k^{(-)} \circ f_{k+1}^{(-)} \circ \cdots \\
&\hspace{13em}\circ f_{\ell-1}^{(-)}(R_\ell), \nabla_aq_k \bigg\rangle \Bigg) \mathrm ds, 
\end{aligned}
\end{equation}
where in the last sum we use the convention $\sum_{\ell=n}^{n-1} =0$. Since this quantity has to vanish for any choice of $(\partial_ax_0(0,\cdot), \nabla_aq_0(0,\cdot), \hdots, \nabla_aq_{n-1}(0,\cdot))$, the geodesic equations for the discrete metric are
\begin{align*}
&\nabla_s{x_0}' + \frac{1}{n} \sum_{k=0}^{n-1} f_0^{(-)}\circ \cdots \circ f_{k-1}^{(-)}(R_k)  = 0, \\
&\nabla_s^2q_k + \frac{1}{n} \sum_{\ell=k+1}^{n-1} g_k^{(-)} \circ f_{k+1}^{(-)} \circ \cdots \circ f_{\ell-1}^{(-)}(R_\ell) = 0, 
\end{align*}
for all $k=0, \hdots, n-1$, with the conventions $\sum_{\ell=n}^{n-1} = 0$ and $f_0 \circ \cdots \circ f_{-1} := \text{Id}$.
\end{proof}

\noindent
\textbf{Remark 4}\textit{
Let $[0,1]\ni s\mapsto c(s,\cdot)\in \mathcal M$ be a $C^1$ path of smooth curves and $[0,1]\ni s\mapsto \alpha(s)\in M^{n+1}$ the discretization of size $n$ of $c$. We denote as usual by $q:=c_t/|c_t|^{1/2}$ and $(q_k)_k$ their respective SRV representations. When $n\rightarrow \infty$ and $|\tau_k|\rightarrow 0$ while $n|\tau_k|$ stays bounded for all $0\leq k\leq n$, the coefficients of the discrete geodesic equation \eqref{disgeodeq} for $\alpha$ converge to the coefficients of the continuous geodesic equation \eqref{contgeodeq} for $c$, i.e.
\begin{align*}
&\nabla_s{x_0}'(s) = - r_0(s) + o(1),\\
&\nabla_s^2q_k(s) = - |q_k(s)| (r_{k}(s) + r_{k}(s)^T) +o(1),
\end{align*}
for all $s\in[0,1]$ and $k=0,\hdots, n-1$, where $r_{n-1}=0$ and for $k = 1,\hdots, n-2$,
\begin{equation*}
r_k(s) := \frac{1}{n} \!\sum_{\ell=k+1}^{n-1} \!\!P_c^{\frac{l}{n},\frac{k}{n}}\big( \mathcal R(q, \nabla_sq)c_s(s,\tfrac{\ell}{n})\big) \,\underset{n\to\infty}{\rightarrow} \, r(s,\tfrac{k}{n}),
\end{equation*}
with the exception that the sum starts at $\ell=0$ for $r_0$. 
}
\begin{proof}
This is due to three arguments : (1) at the limit, $f_k(w) = w + o(1/n)$ and $g_k(w) = |q_k| (w + {w}^T)+ o(1/n)$, (2) parallel transport along a piecewise geodesic curve uniformly converges to the parallel transport along the limit curve, and (3) the discrete curvature term $R_k(s)$ converges to the continuous curvature term $\mathcal R(q,\nabla_sq)c_s(s,\tfrac{k}{n})$ for all $k$. Indeed, let $\hat c$ be the unique piecewise geodesic curve of which $\alpha$ is the discretization, i.e. $c\big(\frac{k}{n}\big)=\hat c\big(\frac{k}{n}\big)=x_k$ for all $k=0,\hdots,n$ and $\hat c$ is a geodesic on each segment $\big[\tfrac{k}{n}, \tfrac{k+1}{n}\big]$. Defining\begin{align*}
&\hat r_0 := \tfrac{1}{n}\big(R_{0}+f_{0}^{(-)}(R_{1})+\hdots + f_{0}^{(-)}\circ \cdots \circ f_{n-2}^{(-)}(R_{n-1})\big)\\
&\hat r_k := \tfrac{1}{n}\big(R_{k+1}+f_{k+1}^{(-)}(R_{k+2})+\hdots \\
&\hspace{4.5em}+ f_{k+1}^{(-)}\circ \cdots \circ f_{n-2}^{(-)}(R_{n-1})\big) \quad 1\leq k \leq n-2,\\
&\hat r_{n-1}:=0,
\end{align*}
the geodesic equations can be written in terms of the vectors $\hat r_k$ 
\begin{align*}
&\nabla_s{x_0}'(s) + \hat r_0(s) = 0,\\
&\nabla_s^2q_k(s) + g_k^{(-)}\big(\hat r_{k}(s)\big) = 0.
\end{align*}
We can show that for any $0\leq k\leq \ell\leq n-2$ and any vector $w\in T_{x_{\ell+1}}M$,
\begin{align*}
&\left| f_k^{(-)}\circ\cdots\circ f_\ell^{(-)}(w) - P_c^{\frac{\ell+1}{n},\frac{k}{n}}(w)\right|\\
&\leq \sum_{j=k}^\ell |a_j-1|\cdot \left| f_{j+1}^{(-)}\circ\cdots\circ f_\ell^{(-)}(w)\right|\\
&\hspace{6em}+ \left| P_{\hat c}^{\frac{\ell+1}{n},\frac{k}{n}}(w)-P_c^{\frac{\ell+1}{n},\frac{k}{n}}(w)\right|.
\end{align*}
Since $|a_j-1|/|\tau_k|^2\to 0$ when $n\to \infty$ and $n|\tau_k|$ stays bounded, we have for all $0\leq j \leq n$ and $n$ large enough $|a_j-1|\leq \frac{1}{n^2}$, and using the fact that parallel transport along a piecewise geodesic curve uniformly converges to the parallel transport along the limit curve, we get
\begin{equation*}
| f_k^{(-)}\circ\cdots\circ f_\ell^{(-)}(w) - P_c^{\frac{\ell+1}{n},\frac{k}{n}}(w)| \to 0 
\end{equation*}
when $n\to \infty$. Now, denoting by 
\begin{equation*}
R(s,t):=\mathcal R(q,\nabla_sq)c_s(s,t) 
\end{equation*}
the curvature term involved in the continuous geodesic equations, we have since ${x_k}'(s)=c_s(s,\tfrac{k}{n})$ and $| \mathcal R(X,Y)Z| \leq |K| \cdot(|\langle Y,Z\rangle| |X| + |\langle X,Z,\rangle| |Y|) \leq 2 |K|\cdot |X|\cdot |Y| \cdot |Z|$ by Cauchy Schwarz,
\begin{align*}
|R_k - R(\tfrac{k}{n})| &\leq |\mathcal R(q_k - q(\tfrac{k}{n}), \nabla_sq_k){x_k}'| \\
& \hspace{3em}+|\mathcal R(q(\tfrac{k}{n}), \nabla_sq_k - \nabla_sq(\tfrac{k}{n})){x_k}'|\\
&\leq |q_k - q(\tfrac{k}{n})| \cdot  |\nabla_sq_k| \cdot  |{x_k}'| \\
&\hspace{3em}+ |q(\tfrac{k}{n})|\cdot |\nabla_sq_k - \nabla_sq(\tfrac{k}{n})| \cdot |{x_k}'|
\end{align*}
Let us show that both summands of this upper bound tend to $0$ when $n\to \infty$.
\begin{align*}
&|q_k - q(\tfrac{k}{n})| = \left| |n\tau_k|^{-\frac{1}{2}} n\tau_k - |c_t(\tfrac{k}{n})|^{-\frac{1}{2}}c_t(\tfrac{k}{n}) \right| \\
&\leq ||n\tau_k|^{-\frac{1}{2}} - |c_t(\tfrac{k}{n})|^{-\frac{1}{2}}|\cdot |n\tau_k| + |c_t(\tfrac{k}{n})|^{-\frac{1}{2}}|n\tau_k - c_t(\tfrac{k}{n})|\\
&= \frac{|n\tau_k| - |c_t(\tfrac{k}{n})|}{|n\tau_k|^{\frac{1}{2}} + |c_t(\tfrac{k}{n})|^{\frac{1}{2}}}\cdot |n\tau_k| + |c_t(\tfrac{k}{n})|^{-\frac{1}{2}}|n\tau_k - c_t(\tfrac{k}{n})|\\
&\leq \left(\frac{|n\tau_k|}{|n\tau_k|^{\frac{1}{2}} + |c_t(\tfrac{k}{n})|^{\frac{1}{2}}} + |c_t(\tfrac{k}{n})|^{-\frac{1}{2}}\right)|n\tau_k - c_t(\tfrac{k}{n})|
\end{align*}
and since the portion of $c(s,\cdot)$ on the segment $[\tfrac{k}{n},\tfrac{k+1}{n}]$ is close to a geodesic at the limit, $|n\tau_k - c_t(\tfrac{k}{n})|\to 0$ when $n\to \infty$, and so does $|q_k(s) - q(\tfrac{k}{n})|$. Similarly,
\begin{align*}
|\nabla_sq_k-\nabla_sq(\tfrac{k}{n})| &= \Big| |n\tau_k|^{-1/2}(n\nabla_s\tau_k - \tfrac{1}{2}{n\nabla_s\tau_k}^T) \\
&\hspace{1em}- |c_t|^{-1/2}(\nabla_sc_t(\tfrac{k}{n}) - \tfrac{1}{2}{\nabla_sc_t(\tfrac{k}{n})}^T)\Big|\\
&\leq ||n\tau_k|^{-1/2} - |c_t(\tfrac{k}{n})|^{-1/2}|\cdot |n\nabla_s\tau_k| \\
&\hspace{1em}+ |c_t|^{-1/2}|n\nabla_s\tau_k - \nabla_sc_t(\tfrac{k}{n})|,
\end{align*}
where once again $||n\tau_k|^{-1/2} - |c_t(\tfrac{k}{n})|^{-1/2}|\to 0$ and $|n\nabla_s\tau_k|$ is bounded. The last term can be bounded, for $n$ large enough, by
\begin{align*}
|n\nabla_s\tau_k - \nabla_sc_t(\tfrac{k}{n})| &\leq |n\nabla_s\tau_k - n\big(c_s(\tfrac{k+1}{n})^\parallel - c_s(\tfrac{k}{n})\big)| \\
&+  |\nabla_tc_s(\tfrac{k}{n}) - n\big(c_s(\tfrac{k+1}{n})^\parallel - c_s(\tfrac{k}{n})\big)|\\
&\leq n|1-b_k^{-1}|\cdot |c_s(\tfrac{k+1}{n})^\parallel - c_s(\tfrac{k}{n})| \\
&\hspace{9em}+ \frac{1}{n}|\nabla_t\nabla_tc_s|_\infty \\
&\leq \frac{1}{n}( |\nabla_tc_s|_\infty + |\nabla_t\nabla_tc_s|_\infty),
\end{align*}
since $\nabla_s\tau_k = (D_\tau \alpha')_k = ({x_{k+1}}^\parallel - x_k)^T + b_k^{-1}({x_{k+1}}'-{x_k}')^N$ and $b_k^{-1}\to 1$. Finally, we can see that
\begin{align*}
&|\hat r_0(s)- r_0(s)|\leq \frac{1}{n} | R_0 - R(0)| + \frac{1}{n} \sum_{\ell=0}^{n-2}|R_{\ell+1}-R(\tfrac{\ell+1}{n})| \\
&+ \frac{1}{n}\sum_{\ell=0}^{n-2} \left| f_0^{(-)}\circ \hdots \circ f_\ell^{(-)} (R_{\ell+1}) - P_c^{\frac{\ell+1}{n},0}(R_{\ell+1})\right|
\end{align*}
goes to $0$ when $n\to \infty$. We can show in a similar way that $|g_k^{(-)}(\hat r_k) - |q_k|(r_k+{r_k}^T)|\to 0$ when $n\to \infty$.
\end{proof}

\noindent
\textbf{Proposition 7 (Discrete exponential map)}\textit{
Let $[0,1] \ni s\mapsto \alpha(s) = (x_0(s), \hdots, x_n(s))$ be a geodesic path in $M^{n+1}$. For all $s\in [0,1]$, the coordinates of its acceleration $\nabla_s\alpha'(s)$ can be iteratively computed in the following way
\begin{equation*}
\begin{aligned}
&\nabla_s{x_0}' =-  \frac{1}{n} \Big( R_0 + f_0^{(-)}(R_1)\\
&\hspace{7em} + \hdots + f_0^{(-)}\circ \cdots \circ f_{n-2}^{(-)} (R_{n-1})\Big),\\
&\nabla_s{x_{k+1}}'^\parallel = \nabla_sf_k({x_k}') + f_k(\nabla_s{x_k}') + \frac{1}{n} \nabla_sg_k(\nabla_sq_k) \\
&\hspace{7em}+ \frac{1}{n} g_k(\nabla_s^2q_k) +\mathcal R( \tau_k,Y_k)({x_{k+1}}'^\parallel),
\end{aligned}
\end{equation*}
for $k=0,\hdots,n-1$, where the $R_k$'s are defined as in Proposition \ref{prop:geodeq}, the symbol $\cdot^\parallel$ denotes the parallel transport from $x_{k+1}(s)$ back to $x_k(s)$ along the geodesic linking them, the maps $\nabla_sf_k$ and $\nabla_sg_k$ are given by Lemma \ref{lem:ns}, $Y_k$ is given by Equation \eqref{yk} and
\begin{equation*}
\begin{aligned}
& \nabla_s\tau_k = (D_\tau \alpha')_k, \quad \nabla_sv_k = \frac{1}{|\tau_k|} \left( \nabla_s\tau_k - {\nabla_s\tau_k}^T \right),\\
&\nabla_sq_k = \sqrt{\frac{n}{|\tau_k|}} \left( \nabla_s\tau_k - \frac{1}{2}{\nabla_s\tau_k}^T \right),\\
&\nabla_s^2q_k = - \frac{1}{n} \,\,g_k^{(-)}\Big( R_{k+1} + f_{k+1}^{(-)}(R_{k+2}) \\
&\hspace{7em}+ \hdots + f_{k+1}^{(-)} \circ \cdots \circ f_{n-2}^{(-)}(R_{n-1})\Big).
\end{aligned}
\end{equation*}
}
\begin{proof}
For all $s\in[0,1]$, we initialize $\nabla_s{x_k}'(s)$ for $k=0$ using the first geodesic equation in \eqref{disgeodeq}; the difficulty lies in deducing $\nabla_s{x_{k+1}}'(s)$ from $\nabla_s{x_k}'(s)$. Just as we have previously obtained \eqref{rec1}, we can obtain by replacing the derivatives with respect to $a$ by derivatives with respect to $s$
\begin{align}
{{x_{k+1}}'}^\parallel  &= {{x_k}'}^T + a_k\, {{x_k}'}^N + {\nabla_sq_k}^T + b_k {\nabla_s\tau_k}^N,\label{xkplus1}\\
{{x_{k+1}}'}^\parallel  &= f_k({x_k}') + \frac{1}{n}\, g_k(\nabla_sq_k),\nonumber
\end{align}
and by differentiating with respect to $s$
\begin{equation}
\label{nspkplus1}
\begin{aligned}
&\nabla_s\left({x_{k+1}}'^\parallel \right) = \nabla_sf_k({x_k}') + f_k(\nabla_s{x_k}') \\
&\hspace{7em}+ \frac{1}{n} \nabla_sg_k(\nabla_sq_k) + \frac{1}{n}g_k(\nabla_s^2q_k).
\end{aligned}
\end{equation}
We have already computed \eqref{nswpar} the covariant derivative of a vector field $s \mapsto {w_{k+1}(s)}^{\parallel}\in T_{x_{k}(s)}M$ and so we can write
\begin{equation*}
\nabla_s\big({x_{k+1}}'^\parallel \big) = \big( \nabla_s{x_{k+1}}'\big)^\parallel +  \mathcal R(Y_k,\tau_k)\big({x_{k+1}}'^\parallel\big),
\end{equation*}
where $Y_k$ is defined by Equation \eqref{yk}. Together with Equation \eqref{nspkplus1}, this gives the desired equation for ${\nabla_s{x_{k+1}}'}^\parallel$. Finally, $\nabla_s\tau_k=(D_\tau \alpha')_k$ results directly from \eqref{xkplus1}, $\nabla_s^2q_k$ is deduced from the second geodesic equation and the remaining equations follow from simple computation.
\end{proof}

\noindent
\textbf{Proposition 8 (Discrete Jacobi fields)}\textit{
Let $[0,1] \ni s\mapsto \alpha(s) = (x_0(s), \hdots, x_n(s))$ be a geodesic path in $M^{n+1}$, $[0,1]\ni s\mapsto J(s)=(J_0(s),\hdots, J_n(s))$ a Jacobi field along $\alpha$, and $(-\delta, \delta) \ni a \mapsto \alpha(a,\cdot)$ a corresponding family of geodesics, in the sense just described. Then $J$ verifies the second order linear ODE
\begin{align*}
&\nabla_s^2J_0 = \mathcal R({x_0}', J_0){x_0}' - \frac{1}{n}\Big( \nabla_aR_0 + f_0^{(-)}(\nabla_aR_1) + \hdots\\
& +  f_0^{(-)}\circ\cdots\circ f_{n-2}^{(-)}(\nabla_aR_{n-1})\Big)\\
&\hspace{0em} -\frac{1}{n} \sum_{k=0}^{n-2} \sum_{\ell = 0}^k f_0^{(-)}\circ \cdots \circ\nabla_a\big(f_\ell^{(-)}\big)\circ\cdots \circ f_{k}^{(-)}(R_{k+1}),\\
&{\nabla_s^2J_{k+1}}^\parallel = f_k(\nabla_s^2J_k) + 2 \nabla_sf_k(\nabla_sJ_k) + \nabla_s^2f_k(J_k) \\
&+ \frac{1}{n}  g_k(\nabla_s^2\nabla_aq_k)+ \frac{2}{n} \nabla_sg_k(\nabla_s\nabla_aq_k)+ \frac{1}{n}\nabla_s^2g_k(\nabla_aq_k)\\
&+ 2\mathcal R(\tau_k, Y_k)({\nabla_sJ_{k+1}}^\parallel) +\mathcal R(\nabla_s\tau_k,Y_k)({J_{k+1}}^\parallel) \\
&+ \mathcal R(\tau_k,\nabla_sY_k)({J_{k+1}}^\parallel) +\mathcal R(\tau_k,Y_k)\Big(\mathcal R(Y_k,\tau_k)({J_{k+1}}^\parallel)\Big), 
\end{align*}
for all $0\leq k \leq n-1$, where $R_k := \mathcal R(q_k,\nabla_sq_k){x_k}'$ and the various covariant derivatives according to $a$ can be expressed as functions of $J$ and $\nabla_sJ$, 
\begin{align*}
& \nabla_aR_k = \mathcal R\big(\nabla_aq_k,\nabla_sq_k\big){x_k}' + \mathcal R\big(q_k,\nabla_s\nabla_aq_k \\
&+ \mathcal R(J,{x_k}')q_k \big){x_k}' + \mathcal R\big(q_k,\nabla_sq_k)\nabla_sJ_k,
\end{align*}
\begin{align*}
&\nabla_aq_k = \sqrt{\frac{n}{|\tau_k|}} \left( \nabla_a\tau_k - \frac{1}{2}{\nabla_a\tau_k}^T \right),\quad \nabla_a\tau_k = (D_\tau J)_k, \\
&\nabla_av_k = \frac{1}{|\tau_k|} \left( \nabla_a\tau_k - {\nabla_a\tau_k}^T \right),\\
&\nabla_s\nabla_aq_k = n \, {g_k}^{-1}\big( (\nabla_sJ_{k+1})^\parallel + \mathcal R(Y_k,\tau_k)({J_{k+1}}^\parallel) \\
&- \nabla_sf_k(J_k) - f_k(\nabla_sJ_k) \big)  + n \, \nabla_s\big({g_k}^{-1}\big)\big({J_{k+1}}^\parallel - f_k(J_k)\big), \\
&\nabla_s^2\nabla_aq_k = - \frac{1}{n} \!\sum_{\ell=k+1}^{n-1} g_k^{(-)} \circ f_{k+1}^{(-)}\circ\cdots\circ f_{\ell-1}^{(-)}(\nabla_aR_\ell)\\
&+\mathcal R(\nabla_s{x_k}',J_k)q_k + \mathcal R({x_k}',\nabla_sJ_k)q_k + 2\mathcal R({x_k}',J_k)\nabla_sq_k \\
&- \frac{1}{n}\! \sum_{\ell=k+1}^{n-1} \!\sum_{j = k}^{\ell-1} g_k^{(-)}\circ \cdots \circ\nabla_a\big(f_j^{(-)}\big)\circ\cdots \circ f_{\ell-1}^{(-)}(R_\ell),\\
&\nabla_sY_k = (\nabla_s{x_k}')^T + b_k(\nabla_s{x_k}')^N+ \partial_sb_k ({x_k}')^N\\
& + (1-b_k)\big(\langle {x_k}',\nabla_sv_k\rangle v_k \langle {x_k}',v_k\rangle \nabla_sv_k\big) +\tfrac{1}{2}(\nabla_s^2\tau_k)^T \\
&+ K\frac{1-a_k}{|\tau_k|^2}(\nabla_s^2\tau_k)^N + \partial_s\Big(K\frac{1-a_k}{|\tau_k|^2}\Big)(\nabla_s\tau_k)^N \\
&+\Big(\tfrac{1}{2}-K\frac{1-a_k}{|\tau_k|^2}\Big)(\langle \nabla_s\tau_k,\nabla_sv_k\rangle v_k + \langle \nabla_s\tau_k,v_k\rangle \nabla_sv_k),
\end{align*}
with the notation conventions $f_{k+1}^{(-)} \circ \hdots \circ f_{k-1}^{(-)}:=Id$, $\sum_{\ell=n}^{n-1}:=0$ and with the maps
\begin{align*}
&\nabla_a\big(f_k^{(-)}\big)(w)=(\nabla_af_k)^{(-)}(w) + f_k\Big(\mathcal R(Z_k,\tau_k)({w_{k+1}}^\parallel)\Big),\\
&\nabla_a\big(g_k^{(-)}\big)(w)=(\nabla_ag_k)^{(-)}(w) + g_k\Big(\mathcal R(Z_k,\tau_k)({w_{k+1}}^\parallel)\Big),
\end{align*}
\begin{align*}
&\nabla_s\big({g_k}^{-1}\big)(w) = \partial_s{|q_k|}^{-1} |q_k| {g_k}^{-1}(w) +\! |q_k|^{-1}\partial_s(b_k^{-1})\, {w}^N\\
&\hspace{1em}+ |q_k|^{-1} \, \big(1/2 - b_k^{-1}\big)\big(\langle w,\nabla_sv_k\rangle v_k + \langle w,v_k\rangle\nabla_sv_k\big),
\end{align*}
and
\begin{equation*}
Z_k = {J_k}^T + b_k {J_k}^N + \tfrac{1}{2}{\nabla_a\tau_k}^T + K\frac{1-a_k}{|\tau_k|^2}{\nabla_a\tau_k}^N.
\end{equation*}
}

\begin{proof}
For all $a\in(-\delta,\delta)$, $\alpha(a,\cdot)$ verifies the geodesic equations \eqref{disgeodeq}. Taking the covariant derivative of these equations according to $a$ we obtain
\begin{align}
&\nabla_a\nabla_s\partial_s{x_0} \\
&+ \frac{1}{n} \sum_{k=0}^{n-1} \nabla_a\Big(f_0^{(-)}\circ \cdots \circ f_{k-1}^{(-)} \big( \mathcal R(q_k,\nabla_sq_k)\partial_sx_k \big) \Big) = 0, \nonumber \\
&\nabla_a\nabla_s^2q_k + \frac{1}{n} \sum_{\ell=k+1}^{n-1} \nabla_a\Big(g_k^{(-)} \circ f_{k+1}^{(-)} \circ \cdots \label{nansnsqk}\\
&\hspace{10em} \circ f_{\ell-1}^{(-)}\big( \mathcal R(q_\ell, \nabla_sq_\ell)\partial_sx_\ell \big) \Big)= 0. \nonumber
\end{align}
Since for $a=0$, $\nabla_a\nabla_s\partial_sx_0 = \nabla_s^2J_0 + \mathcal R(J_0,\partial_sx_0)\partial_sx_0$, we get
\begin{align*}
&\nabla_s^2J_0 = \mathcal R(\partial_sx_0, J_0)\partial_sx_0 \\
&\hspace{2em}- \frac{1}{n} \sum_{k=0}^{n-1} \nabla_a\Big( f_0^{(-)}\circ \cdots \circ f_{k-1}^{(-)}\big(\mathcal R(q_k,\nabla_sq_k)\partial_sx_k \big)\Big),
\end{align*}
and the differentiation
\begin{align*}
&\nabla_a\big( f_0^{(-)}\circ \cdots \circ f_{k-1}^{(-)}(R_k)\big) = f_0^{(-)}\circ \cdots \circ f_{k-1}^{(-)}(\nabla_aR_k) \\
&\hspace{4em}+ \sum_{\ell=0}^{k-1} f_0^{(-)}\circ \cdots \circ\nabla_a\big(f_\ell^{(-)}\big)\circ \cdots \circ f_{k-1}^{(-)}(R_k)
\end{align*}
gives the desired equation for $\nabla_s^2J_0$. Now we will try to deduce $\nabla_s^2J_{k+1}$ from \eqref{nansnsqk}. If ${J_{k+1}}^\parallel(s)$ denotes the parallel transport of the vector $J_{k+1}(s)$ from $x_{k+1}(s)$ back to $x_k(s)$ along the geodesic that links them, we know from \eqref{rec1} that
\begin{equation}
\label{jkplus1}
{J_{k+1}}^\parallel = f_k(J_k) + \frac{1}{n} g_k(\nabla_aq_k).
\end{equation}
We also know from \eqref{nswpar} that
\begin{equation}
\label{nablasjpar}
(\nabla_sJ_{k+1})^\parallel = \nabla_s({J_{k+1}}^\parallel) + \mathcal R(\tau_k,Y_k)({J_{k+1}}^\parallel), 
\end{equation}
and by iterating
\begin{align*}
(\nabla_s^2J_{k+1})^\parallel &= \nabla_s\big((\nabla_s{J_{k+1}})^\parallel\big) + \mathcal R\big(\tau_k,Y_k\big)\big( (\nabla_sJ_{k+1})^\parallel \big)\\
&= \nabla_s^2({J_{k+1}}^\parallel) + \nabla_s\big(\mathcal R(\tau_k,Y_k)({J_{k+1}}^\parallel)\big) \\
& \hspace{8em}+ \mathcal R\big(\tau_k,Y_k\big)\big( (\nabla_sJ_{k+1})^\parallel \big)
\end{align*}
Developping and injecting Equation \eqref{jkplus1} in the latter gives
\begin{align*}
&(\nabla_s^2J_{k+1})^\parallel = \nabla_s^2\big( f_k(J_k) \big) +  \frac{1}{n} \nabla_s^2\big(g_k( \nabla_aq_k)\big) \\
&\hspace{4em}+ \mathcal R(\nabla_s\tau_k,Y_k)({J_{k+1}}^\parallel)+ \mathcal R(\tau_k,\nabla_sY_k)({J_{k+1}}^\parallel) \\
&\hspace{4em}+ \mathcal R(\tau_k,Y_k)(\mathcal R(Y_k,\tau_k)({J_{k+1}}^\parallel))\\
&\hspace{4em}+ 2\mathcal R\big(\tau_k,Y_k\big)\big( (\nabla_sJ_{k+1})^\parallel \big).
\end{align*}
Developping the covariant derivatives $\nabla_s^2\big( f_k(J_k) \big)$ and $\nabla_s^2\big(g_k( \nabla_aq_k)\big)$ gives the desired formula. Now let us explicit the different terms involved in these differential equations. Since $\nabla\mathcal R=0$ and $\nabla_a\partial_sx_k = \nabla_s\partial_ax_k$, we have
\begin{align*}
\nabla_aR_k &= \mathcal R(\nabla_aq_k,\nabla_sq_k)\partial_sx_k +\mathcal R(q_k,\nabla_a\nabla_sq_k)\partial_sx_k \\
&\hspace{12em}+ \mathcal R(q_k,\nabla_sq_k)\nabla_sJ_k\\
&=\mathcal R\big(\nabla_aq_k,\nabla_sq_k\big){x_k}' + \mathcal R\big(q_k,\nabla_s\nabla_aq_k\\
&\hspace{4em}+ \mathcal R(J,{x_k}')q_k \big){x_k}' + \mathcal R\big(q_k,\nabla_sq_k)\nabla_sJ_k.
\end{align*} 
By taking the inverse of \eqref{jkplus1} we get
\begin{equation*}
\nabla_aq_k = ng_k^{-1}\big({J_{k+1}}^\parallel - f_k(J_k)\big),
\end{equation*}
and taking the derivative according to $s$ on both sides and injecting Equation \eqref{nablasjpar} gives
\begin{align*}
\nabla_s\nabla_aq_k &= n \, {g_k}^{-1}\big( (\nabla_sJ_{k+1})^\parallel + \mathcal R(Y_k,\tau_k)({J_{k+1}}^\parallel)\\
& - \nabla_sf_k(J_k) - f_k(\nabla_sJ_k) \big)  \\
& + n \, \nabla_s\big({g_k}^{-1}\big)\big({J_{k+1}}^\parallel - f_k(J_k)\big).
\end{align*}
To obtain $\nabla_s^2\nabla_aq_k$, notice that
\begin{align*}
&\nabla_s^2\nabla_aq_k\\
&= \nabla_s\nabla_a\nabla_sq_k + \nabla_s\big(\mathcal R(\partial_sx_k,J_k)q_k \big), \\
&=\nabla_a\nabla_s^2q_k + \mathcal R(\partial_sx_k,J_k)\nabla_sq_k + \nabla_s\big(\mathcal R(\partial_sx_k,J_k)q_k \big),
\end{align*}
and injecting Equation \eqref{nansnsqk} with
\begin{align*}
&\nabla_a\big(g_k^{(-)} \circ f_{k+1}^{(-)} \circ \cdots \circ f_{\ell-1}^{(-)}( R_\ell ) \big)\\
&\hspace{4em}= g_k^{(-)} \circ f_{k+1}^{(-)}\circ\cdots\circ f_{\ell-1}^{(-)}(\nabla_aR_\ell)\\
&\hspace{4em} + \sum_{j = k}^{\ell-1} g_k^{(-)}\circ \cdots \circ \nabla_a\big(f_j^{(-)}\big)\circ \cdots \circ f_{\ell-1}^{(-)}(R_\ell),
\end{align*}
gives us the desired formula. $\nabla_sY_k$ results from simple differentiation, and differentiating the maps $f_k^{(-)}$ and $g_k^{(-)}$ with respect to $a$ is completely analogous to the the computations of Lemma \ref{lemfun}. Finally, the inverse of $g_k$ is given by ${g_k}^{-1} : T_{x_k}M \rightarrow T_{x_k}M$,
\begin{align*}
{g_k}^{-1} : w \mapsto  |q_k|^{-1} \left( b_k^{-1} w + \left( \tfrac{1}{2} - b_k^{-1}\right) {w}^T\right),
\end{align*}
and since 
\begin{equation*}
\nabla_s({w}^T)= (\nabla_sw)^T + \langle w,\nabla_sv_k\rangle v_k + \langle w,v_k\rangle \nabla_sv_k,
\end{equation*} 
it is straightforward to verify that 
\begin{equation*}
\nabla_s\big({g_k}^{-1}\big)(w) =\nabla_s\big({g_k}^{-1}(w)\big) - {g_k}^{-1}(\nabla_sw)
\end{equation*}
gives 
\begin{align*}
&\nabla_s\big({g_k}^{-1}\big)(w) \!=\! \partial_s{|q_k|}^{-1} |q_k| {g_k}^{-1}(w) \!+\! |q_k|^{-1}\partial_s(b_k^{-1})\, {w}^N\\
&\hspace{2em}+ |q_k|^{-1} \, \big(1/2 - b_k^{-1}\big)\big(\langle w,\nabla_sv_k\rangle v_k + \langle w,v_k\rangle\nabla_sv_k\big).
\end{align*}
\end{proof}


\end{document}